%% file: soficIDS.tex
\documentclass[a4paper,10pt,final]{article}
\usepackage{amssymb,amsmath,amsfonts,amsthm}
\usepackage{ifthen}
\usepackage{paralist}
\usepackage{caption,subcaption}
\usepackage{tikz}
\usetikzlibrary{arrows,calc,decorations.pathmorphing,external}
\usepackage{gnuplot-lua-tikz}

\usepackage{xcolor}
\definecolor{darkred}{rgb}{0.5,0,0}
\definecolor{darkgreen}{rgb}{0,0.5,0}
\definecolor{darkblue}{rgb}{0,0,0.5}
\usepackage[color,notcite,notref]{showkeys}
\usepackage[final]{hyperref}
\hypersetup{linktocpage
           ,pdfborder={0 0 0}
           ,colorlinks=true
           ,linkcolor=darkblue
           ,filecolor=darkgreen
           ,urlcolor=darkred
           ,citecolor=darkblue
           }

\usepackage[capitalise]{cleveref}
\crefname{sofic}{condition}{condition}
\creflabelformat{sofic}{\upshape{(#2S#1#3)}}
\crefname{assumption}{assumption}{assumptions}
\creflabelformat{assumption}{\upshape{(#2A#1#3)}}
\crefname{rassum}{property}{properties}
\creflabelformat{rassum}{\upshape{(#2R#1#3)}}
\crefname{inequality}{inequality}{inequality}
\creflabelformat{inequality}{\upshape{(#2#1#3)}}

\usepackage{enumerate}
\usepackage[style=alphabetic 
           ,bibstyle=alphabetic
           ,citestyle=alphabetic
           ,language=english
           ,backref=true
           ]{biblatex}
\numberwithin{equation}{section}    
\usepackage{aliascnt}           
\newcommand{\mynewtheorem}[4][]{
  \ifthenelse{\equal{#1}{}}{    
    \newtheorem{#2}{#3}         
  }{
    \newaliascnt{#2}{#1}        
    \newtheorem{#2}[#2]{#3}     
    \aliascntresetthe{#2}       
  }
  \crefname{#2}{#3}{#4}         
}
\theoremstyle{plain}
\newtheorem{Theorem}{Theorem}[section]
\crefname{Theorem}{Theorem}{Theorems}
\mynewtheorem[Theorem]{Proposition}{Proposition}{Propositions}

\mynewtheorem[Theorem]{Lemma}{Lemma}{Lemmas}
\theoremstyle{definition}
\mynewtheorem[Theorem]{Assumption}{Assumption}{Assumptions}
\mynewtheorem[Theorem]{Definition}{Definition}{Definitions}
\mynewtheorem[Theorem]{Example}{Example}{Examples}
\theoremstyle{remark}
\mynewtheorem[Theorem]{Remark}{Remark}{Remarks}

\renewcommand{\epsilon}{\varepsilon}
\renewcommand{\phi}{\varphi}

\DeclareMathOperator{\Tr}{Tr}

\DeclareMathOperator{\diam}{diam}
\DeclareMathOperator{\supp}{spt}
\DeclareMathOperator{\spt}{spt}
\DeclareMathOperator{\lin}{lin}

\DeclareMathOperator{\id}{id}
\DeclareMathOperator{\sgn}{sgn}
\DeclareMathOperator{\elc}{elc}

\providecommand{\etdef}{\phantom{:}&\rlap{:}=}
\providecommand{\abs}[1]{\lvert#1\rvert}
\providecommand{\setsize}[1]{\lvert\ifthenelse{\equal{#1}{\cdot}}{{}\cdot{}}{#1}\rvert}
\providecommand{\bigabs}[1]{\bigl\lvert\ifthenelse{\equal{#1}{\cdot}}{{}\cdot{}}{#1}\bigr\rvert}
\providecommand{\Bigabs}[1]{\Bigl\lvert#1\Bigr\rvert}

\providecommand{\norm}[2][]{\lVert#2\rVert\ifthenelse{\equal{}{#1}}{}{_{#1}}}
\providecommand{\bignorm}[2][]{\bigl\Vert#2\bigr\Vert\ifthenelse{\equal{}{#1}}{}{_{#1}}}
\providecommand{\Bignorm}[2][]{\Bigl\Vert#2\Bigr\Vert\ifthenelse{\equal{}{#1}}{}{_{#1}}}

\providecommand{\floor}[1]{\lfloor#1\rfloor}
\providecommand{\isect}{\cap}
\providecommand{\Isect}{\bigcap}
\providecommand{\union}{\cup}
\providecommand{\Union}{\bigcup}
\providecommand{\from}{\colon}
\providecommand{\textq}[1]{\text{#1}\quad}
\providecommand{\qtext}[1]{\quad\text{#1}}
\providecommand{\qtextq}[1]{\quad\text{#1}\quad}
\providecommand{\xto}[2][]{\xrightarrow[#1]{#2}}
\providecommand{\spr}[3][]{\langle#2,#3\rangle\ifthenelse{\equal{#1}{}}{}{_{#1}}}
\providecommand{\ifu}[1]{\chi_{#1}}
\providecommand{\mo}[1]{\pi_{#1}}

\providecommand{\Res}{\mathcal R}
\providecommand{\cc}[1]{\overline{#1}}
\providecommand{\argmt}{{}\cdot{}}
\providecommand{\e}{\mathrm e}%
\renewcommand{\Im}{\operatorname{Im}}

\newcommand{\EE}{\mathbb{E}}
\newcommand{\QQ}{\mathbb{Q}}
\newcommand{\RR}{\mathbb{R}}
\newcommand{\CC}{\mathbb{C}}
\newcommand{\NN}{\mathbb{N}}
\newcommand{\ZZ}{\mathbb{Z}}
\newcommand{\PP}{\mathbb{P}}

\newcommand{\cA}{\mathcal{A}}
\newcommand{\cB}{\mathcal{B}}
\newcommand{\cC}{\mathcal{C}}
\newcommand{\cD}{\mathcal{D}}

\newcommand{\cF}{\mathcal{F}}

\newcommand{\cP}{\mathcal{P}}

\newcommand{\cS}{\mathcal{S}}

\makeatletter
\newcounter{const@ntNo}
\newcommand{\reloadConst@nt}{%
  \stepcounter{const@ntNo}%
  \edef\genericConst@ntInternal{C_{\theconst@ntNo}}%
}
\newcommand{\genericConstant}[1]{%
  \reloadConst@nt%
  \expandafter\let\csname #1\endcsname\genericConst@ntInternal%
}
\makeatother

\newcommand{\hm}[1]{\textbf{*}\leavevmode{\marginpar{\tiny%
$\hbox to 0mm{\hspace*{-0.5mm}$\leftarrow$\hss}%
\vcenter{\vrule depth 0.1mm height 0.1mm width \the\marginparwidth}%
\hbox to 0mm{\hss$\rightarrow$\hspace*{-0.5mm}}$\\\relax\raggedright #1}}}
\usepackage{microtype}
\title{Approximation of the integrated density of states on sofic groups}
\author%
{Christoph Schumacher, Fabian Schwarzenberger}

\begin{document}

\maketitle
\begin{abstract}
  In this paper we study spectral properties of self-adjoint operators on a large class of geometries given via sofic groups.
  We prove that the associated integrated densities of states can be approximated via finite volume analogues.
  This is investigated in the deterministic as well as in the random setting.
  In both cases we cover a wide range of operators including in particular unbounded ones.
  The large generality of our setting allows to treat applications from long-range percolation and the Anderson model.
  Our results apply to operators on~$\ZZ^d$, amenable groups, residually finite groups and therefore in particular to operators on trees.
  All convergence results are established without any ergodic theorem at hand.
\end{abstract}

\section{Introduction}

The study of self-adjoint operators on discrete structures has a long history in mathematical physics, both in the deterministic and as well as the random case.
The investigation of spectral properties of such operators is motivated as essential features of solutions of differential equations are encoded in the spectrum of the corresponding operator.
However it is in many cases hard to obtain results on spectrum by directly studying the operator.
One tool to overcome this difficulty is the investigation of the \emph{integrated density of states} (short IDS, also called spectral distribution function) as a rather simple object which still carries much of the spectral information of the operator.


In order to define the IDS one chooses a sequence of finite dimensional self-adjoint operators approximating the original operator in a suitable sense and considers their eigenvalue counting functions.
For each real number~$\lambda$ this function returns the number of eigenvalues of the approximating operator (counting multiplicity) which are not larger than~$\lambda$.
The IDS is then defined as the pointwise limit of the normalized eigenvalue counting functions, if the limit exists.
In this situation it is in many cases possible to show that the IDS equals a the so called \emph{spectral distribution function} (SDF), given via a trace of certain projections, see \eqref{PFtrace}. This equality is sometimes called the \emph{Pastur-Shubin-trace formula}.
Depending on the context, the SDF is sometimes called von Neumann-trace, see for instance \cite{LenzPV-07}, and in other situations the associated measure is known as the Plancherel measure or Kesten spectral measure, see e.g. \cite{BartholdiW-05}.
Now two questions occur: 
\begin{enumerate}[(a)]
  \item Does the limit of the eigenvalue counting functions exist?
  \item Does the Pastur-Shubin-trace formula hold?
\end{enumerate}

The investigation of these questions has a long history.
In the seminal papers \cite{Pastur-71} and \cite{Shubin-79}, the existence of the limit was first rigorously studied.
The authors thereof studied random ergodic and almost periodic operators in Euclidean space.
Later on many results occurred, both in the random as well as in the deterministic setting and for various geometries.
Convergence results on manifolds for random and periodic Schr\"odinger operators are studied in \cite{Sznitman-89,Sznitman-90,AdachiS-93,PeyerimhoffV-02,LenzPV-04} and in the discrete setting for finite difference operators on periodic graphs in \cite{MathaiY-02,MathaiSY-03,DodziukLMSY-03,Veselic-05b,BartholdiW-05}.
Note that the approximability of the zeroth $\ell^2$-Betti number can be interpreted as the evaluation of the IDS at one single point.
Therefore it is important to mention the works \cite{Lueck94,DodziukM-97,DodziukM-98,Eckmann-99,LueckS-99}, where this problem was studied.

In the present paper we study the questions (a) and (b) in a very general background.
First of all our geometry is very general as we can treat all graphs which are given as Cayley graphs of finitely generated sofic groups.
Note that this class of groups contains all amenable groups, residually finite groups and therefore especially all groups of sub-exponential growth as well as some exponentially growing groups, as for instance the free group. Furthermore it was shown in \cite{Cornulier-11} that there exist sofic groups which are not a limit of amenable groups.
The notion of sofic groups goes back to Gromov \cite{Gromov-99} and Weiss \cite{Weiss2000} and was later on studied for instance in \cite{ElekS-03,ElekS-04,Thom-08, Pestov-08,Bowen-10,Cornulier-11,Bowen-12}.
The class of operators we consider is in the deterministic as well as in the random setting not very restricted, too.
The deterministic operators are assumed to be self-adjoint and translationally invariant, and we assume that the compactly supported functions form a core.
In the random situation we assume translational invariance in distribution and a moment condition~\eqref{summable}.
In both settings we can prove the existence of the limit of the normalized eigenvalue counting functions and that this limit is given by the Pastur-Shubin trace formula.
Hence we give positive answers to the questions~(a) and~(b) in a very general situation.
Note that our assumptions allow in the random and in the non-random case unboundedness of the operators and of their hopping range.

The approximation of trees via finite volume graphs is an intensively studied problem, see e.g.\ \cite{AizenmanW-08} and references therein. 
The main obstacle is the non-amenability of trees, i.e.\ the average over a ball depends drasticly on the contribution of the boundary sphere of the ball.
The sphere of the ball has nodes of altered degree, and that prevents good approximation properties.
As a result, instead of the Cayley graph of the free group, \cite{AizenmanW-08} approximated the canopy tree, which displays the leaves of degree~$1$ in a prominent fashion.
The same phenomenon was encountered in \cite{Sznitman-89,Sznitman-90} in the continuous setting.
In order to construct good approximating graphs for regular trees, one resorts to regular graphs.
As shown by \cite{McKay81}, the correct strategy in order to approximate the regular tree is to avoid large quantities of small cycles.
Other possibilities to improve the approximation properties of balls are studied in \cite{FroeseHaslerSpitzer2011}.
There the authors insert weighted edges connecting the boundary elements.

In this paper we basically show that the definition of sofic groups gives a natural criterion for the choice of the approximating finite objects.
In \cref{freegrp}, we demonstrate that this definition can be exploited to construct approximating graphs for numerical purposes.
We hereby open the way to explore phenomena like eigenvalue statistics, which depend by definition on suitable approximations, for a wide variety of models.
In particular it would be interesting to study Poisson statistics vs.\ level repulsion in more models with absolutely continuous spectrum,
as done for the canopy tree in \cite{AizenmanW-08}.

We present the content of the paper in detail.
In the next subsection we present the setting and the definition of sofic groups, which goes back to Weiss \cite{Weiss2000}.
\Cref{dete} is devoted to prove the convergence result for deterministic operators.
\Cref{thm:weak1} should be compared to \cite[Theorem~2.3.1]{Lueck94}.
While L\"uck covers residually finite groups and bounded operators, we treat the more general class of sofic groups and allow unboundedness of the operators.
\Cref{secrandom} splits in three parts. From here on we study questions (a) and (b) for random operators on sofic groups.
This has, to the best of our knowledge, never been done in this general setting.
In \cref{countable} we explicitly construct random self-adjoint operators on countable groups, translationally invariant in distribution, employing techniques from \cite{PasturFigotin1992}.
We also define random Hamiltonians, the class of operators we study henceforth on sofic groups.
This class includes famous random models from mathematical physics such as the Anderson model on Cayley graphs and even percolation graphs.
\Cref{mean} is the randomized version of \cref{dete} for random Hamiltonians.
Here we define the finite dimensional approximating operators and show the convergence of their eigenvalue counting functions in expectation, see \cref{thm:expected}.
In \cref{almostsure} we improve the convergence of the expectation values to almost sure convergence.
The means of choice here is a well known concentration inequality by McDiarmid.
Note that all of our convergence results are established without any ergodic theorem at hand.
The final \cref{exa} contains a detailed example on long range percolation on sofic groups in \cref{percolation}.
To keep the link to other works quoted above, we also include in \cref{amenable,residuallyfinite} direct proofs that amenable and residually finite groups are sofic.
For free groups, the IDS is known explicitly.
We give a constructive proof, inspired by \cite{Biggs88},
of the fact that free groups are residually finite in \cref{freegrp}.
In \cite{KomarovMcNeillWebster07}, a different approach is suggested,
but it seems not as straight forward.
We illustrate our deterministic convergence result
for the free group with two generators with a numerical implementation.
We further pose open questions
about the quality of the corresponding sofic approximation.

\subsection*{Acknowledgments}
It is a great pleasure to thank
Christian Seifert, Ivan Veseli\'c,
Wolfgang Spitzer, Matthias Keller and Felix Pogorzelski
for enlightening and encouraging discussions.
In particular we thank Ivan Veseli\'c for the input of the concentration inequality by McDiarmid.


\subsection{Setting and Notation}

Let $G$ be group and $S\subseteq G$ a finite and symmetric set of generators.
The \emph{Cayley graph} $\Gamma=\Gamma(G,S)$ is the graph with vertices~$G$
and a directed edge from $x\in G$ to $y\in G$, if $xy^{-1}\in S$.
We label the edge between~$x$ and~$y$ with $xy^{-1}$.
For any graph~$(V,E)$ the \emph{graph distance}~$d^{(V,E)}\from V\times V\to\NN_0$
is given as the length of the shortest path between the arguments,
ignoring the direction of the edges.
We denote the ball around the identity element of~$G$
with respect to $d^\Gamma$ by~$B_r^G$.
\begin{Definition}[cf.~\cite{Weiss2000}]
  In the above setting, $G$ is \emph{sofic},
  if for all $\epsilon>0$ and $r\in\NN$
  there is a finite directed graph $(V_{r,\epsilon},E_{r,\epsilon})$,
  edge labeled by~$S$, which has a finite subset
  $V_{r,\epsilon}^{(0)}\subseteq V_{r,\epsilon}$ such that:
  \begin{enumerate}[({S}1)]
    \item\label[sofic]{S1}
      for all $v\in V_{r,\epsilon}^{(0)}$ the $r$-ball around~$v$
      in the graph distance of $(V_{r,\epsilon},E_{r,\epsilon})$
      is isomorphic as a labeled graph to~$\Gamma|_{B_r^G}$.
    \item\label[sofic]{S2}
      $\setsize{V_{r,\epsilon}^{(0)}}
        \ge(1-\epsilon)\setsize{V_{r,\epsilon}}$.
  \end{enumerate}
\end{Definition}
Note that the property of being sofic is independent of the specific choice of the symmetric generating system $S$, c.f. \cite{Weiss2000}.
Except otherwise mentioned, we assume that the group~$G$ is sofic. In order to simplify notation,
we choose some function $\epsilon\from\NN\to(0,\infty)$
with $\lim_{r\to\infty}\epsilon(r)=0$ and write
\begin{equation}\label{eq:simplify}
  \Gamma_r:=(V_r,E_r):=(V_{r,\epsilon(r)},E_{r,\epsilon(r)})\textq,
  V_r^{(0)}:=V_{r,\epsilon(r)}^{(0)}\textq,
  d_r:=d^{(V_r,E_r)}\text.
\end{equation}
Throughout the paper we deal with the Hilbert space~$\ell^2(G)$
of square summable functions on~$G$.
We denote the Kronecker delta~$\delta_x\in\ell^2(G)$ on $x\in G$
by $\delta_x\from G\to\{0,1\}$, i.e.\ $\delta_x(z)=1$ iff $x=z$.
The set of compactly supported functions on~$G$
is $D_0:=\lin\{\delta_x\mid x\in G\}$.

\section{Deterministic approximation results}\label{dete}

Let $A\from\ell^2(G)\to\ell^2(G)$ be a self-adjoint operator and denote
the matrix element for $x,y\in G$ by $a(x,y):=\spr{\delta_x}{A\delta_y}$.
We assume that~$A$ satisfies
\begin{enumerate}[({A}1)]
  \item\label[assumption]{transl}
    $a(x,y)=a(xz,yz)$ for all $x,y,z\in G$,
  \item\label[assumption]{core}
    The set of compactly supported functions~$D_0$ is a core for~$A$.
\end{enumerate}
Note that \cref{core} implies
$\norm[2]{A\delta_x}^2=\sum_{y\in G}\abs{a(x,y)}^2<\infty$ for all $x\in G$.

\begin{Remark}
  Note that the operators fulfilling \cref{transl,core}
  can be unbounded.
  An example of such an operator on~$\ZZ$
  can easily be constructed with with help of the Fourier transformation
  $\cF\from\ell^2(\ZZ)\to L^2(S^1)$, which is unitary.
  Therefore, $\cF$ conjugates self-adjoint operators with self-adjoint ones
  and unbounded ones with unbounded ones.
  A self-adjoint multiplication operator
  which multiplies with an unbounded function in~$L^2(S^1)$
  is by~$\cF$ conjugated to an unbounded operator on~$\ell^2(G)$.
  The latter can be expressed as convolution operator
  and is thus satisfies \cref{transl}.
  The Fourier transform maps $L^2(S^1)$ onto $\ell^2(\ZZ)$,
  and for a convolution with an $\ell^2$-function,
  \cref{esa} tells us that \cref{core} is satisfied, too.
\end{Remark}


For each $x\in V_r^{(0)}$ we have a labeled graph isomorphism
\begin{equation}\label{psiso}
  \psi_{x,r}\from B_r^{V_r}(x)\to B_r^G\text.
\end{equation}
Note that for $x,y\in V_r^{(0)}$ with $d_r(x,y)<r$ we have
\begin{equation}\label{eq:psiInverse}
  \psi_{x,r}(y)
    =(\psi_{y,r}(x))^{-1}\text,
\end{equation}
since the labels along a path from $x$ to $y$
are preserved and equal the inverse labels of the reversed path.
In particular we have $\psi_{x,r}(x)=\id\in G$ for all~$x$.

\begin{Lemma}\label{la:welldef}
 Let $r\in \NN$. If $x,y\in V_r$ and $v,w\in V_r^{(0)}$ fulfill $x,y\in B_{r/2}^{V_r}(v)\cap B_{r/2}^{V_r}(w)$, then we have 
$$\psi_{v,r}(x)(\psi_{v,r}(y))^{-1}=\psi_{w,r}(x)(\psi_{w,r}(y))^{-1}.$$
\end{Lemma}
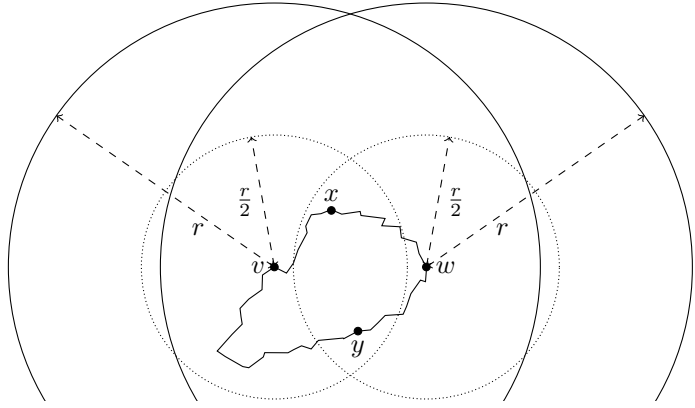
\begin{figure}
  \centering
  \begin{tikzpicture}[scale=.5,decoration={random steps,segment length=1.5mm}]
    \coordinate (v) at (-2,0);
    \coordinate (w) at (2,0);
    \coordinate (x) at (-.5,1.5);
    \coordinate (y) at (.2,-1.7);
    \newlength{\radius}\setlength{\radius}{7cm}
    \clip (-2cm-\radius-1pt,-\radius/2-1pt) rectangle (2cm+\radius+1pt,\radius+1pt);
    \draw [densely dotted] (v) circle (\radius/2);
    \draw (v) circle (\radius);
    \filldraw (v) circle (1mm) node [left] {$v$};
    \draw [densely dotted] (w) circle (\radius/2);
    \draw (w) circle (\radius);
    \filldraw (w) circle (1mm) node [right] {$w$};
    \filldraw (x) circle (1mm) node [above] {$x$};
    \filldraw (y) circle (1mm) node [below] {$y$};
    \draw [decorate] (y) .. controls (-4,-3) .. (v);
    \draw [decorate] (v) to [out=0,in=180] (x);
    \draw [decorate] (y) to [bend right] (w);
    \draw [decorate] (w) to [bend right] (x);
    \draw [<->,dashed] (v) -- +(100:\radius/2)
      node [pos=.5,left] {$\tfrac r2$};
    \draw [<->,dashed] (w) -- +(80:\radius/2)
      node [pos=.5,right] {$\tfrac r2$};
    \draw [<->,dashed] (v) -- +(145:\radius)
      node [pos=.35,below] {$r$};
    \draw [<->,dashed] (w) -- +(35:\radius)
      node [pos=.35,below] {$r$};
  \end{tikzpicture}
  \caption{To \cref{la:welldef}.  Note that all paths stay inside the solid balls.}
  \label{fig:welldef}
\end{figure}
\begin{proof}
  Let $x,y\in V_r$ and $v,w\in V_r^{(0)}$ be such that $x,y\in B_{r/2}^{V_r}(v)\cap B_{r/2}^{V_r}(w)$.
  Then $k:=d_r(x,y)\le r$ and hence all shortest paths in $\Gamma_r$ connecting $x$ and $y$ are completely contained in $B_r^{V_r}(v)$ as well as in $B_r^{V_r}(w)$.
  We consider one of these shortest (directed) paths from $x$ to $y$.
  Let $(s_1,\dotsc,s_k)$ be the vector of the labels of this path.
  Then we have by the choice of $\psi_{v,r}$ that 
  \[
    \psi_{v,r}(x)(\psi_{v,r}(y))^{-1}=s_k\cdots s_1 (\psi_{v,r}(y))(\psi_{v,r}(y))^{-1}=s_k\cdots s_1
  \]
  As we also have $\psi_{w,r}(x)=s_k\cdots s_1 (\psi_{w,r}(y))$, the claim follows.
\end{proof}

We define the projection $A_r\from\ell^2(V_r)\to\ell^2(V_r)$ of~$A$
to the graph~$\Gamma_r$ by
\begin{gather*}
  (A_rf)(x):=\sum_{y\in V_r}a_r(x,y)f(y)\text{, where}\\
  a_r(x,y):=\begin{cases}
    a(\psi_{v,r}(x),\psi_{v,r}(y))
      &\text{if $\exists v\in V_r^{(0)}\colon x,y\in B_{r/3}^{V_r}(v)$}\\
    0 &\text{else.}
  \end{cases}
\end{gather*}
This operator is well-defined by \cref{la:welldef}.
Note that $A_r$ is a symmetric and hence self-adjoint operator on~$\ell^2(V_r)$.
\begin{Remark}\label{remark/3}
  The reason why we use $r/3$ instead of $r/2$ is the following.
  In the proofs of \cref{thm:weak1,thm:expected},
  we need to ensure that $y\in B_r^{V_r}(x_0)$
  whenever $x_0\in V_r^{(0)}$, $x\in B_{r/3}^{V_r}(x_0)$
  and $a_r(x,y)\ne0$.
  As $B_r^{V_r}(x_0)$ is the domain of $\psi_{x_0,r}$,
  we can find corresponding elements in~$B_r^G(\id)$.
  \begin{figure}
    \centering
    \begin{tikzpicture}
      \newlength{\Radius}\setlength{\Radius}{3cm}
      \clip (-\Radius-1pt,-\Radius/3-1pt) rectangle (\Radius+1pt,\Radius+1pt);
      \draw (0,0) circle (\Radius) node [left] {$x_0$};
      \draw [fill] (0,0) circle (.2mm);
      \draw [<->,dashed] (0,0) -- (100:\Radius)
        node [pos=.6,below left] {$r$};
      \draw [densely dotted] (0,0) circle (\Radius/3);
      \draw [<->,dashed] (0,0) -- (-40:\Radius/3)
        node [pos=.4,below left] {$\frac r3$};
      \filldraw (.8,.4) circle (.2mm) node [below left] {$x$};
      \filldraw (1.5,1) circle (.2mm) node [left] {$v$};
      \coordinate (v) at (1.5,1);
      \draw [densely dotted] (v) circle (\Radius/3);
      \draw [<->,dashed] (v) -- +(-40:\Radius/3)
        node [pos=.4,below left] {$\tfrac r3$};
      \filldraw (2.4,1.2) circle (.2mm) node [left] {$y$};
    \end{tikzpicture}
    \caption{Illustration to \cref{remark/3}}
    \label{fig:r/3}
  \end{figure}
\end{Remark}

Define for each $r\in\NN$ the normalized eigenvalue counting function~$N_r$
of~$A_r$ by
\[
  N_r\from\RR\to[0,1]\textq,
  N_r(\lambda):=\frac
    {\setsize{\{\text{eigenvalues of $A_r$ not larger than $\lambda$}\}}}
    {\abs{V_r}}\text,
\]
where the eigenvalues are counted with multiplicity.
If for increasing $r$ the (pointwise) limit of these functions exists, it is called the \emph{integrated density of states}.
Given the operator~$A$ we denote by~$E_\lambda$
the spectral projection on the interval~$(-\infty,\lambda]$.
Using this we set $N\from\RR\to[0,1]$ as
\begin{equation}\label{PFtrace}
  N(\lambda):=\spr{\delta_{\id}}{E_\lambda\delta_{\id}}\text.
\end{equation}
This is a distribution function for a probability measure, which we called the spectral distribution function (SDF). 
The next Theorem shows that the integrated density of states exists and that it equals the spectral distribution function. In other words we show that the \emph{Pastur-Shubin-trace formula} holds.

\begin{Theorem}\label{thm:weak1}
  Let $N$ and $N_n$ be given as above.  Then
  \begin{equation*}
    \lim_{r\to\infty}N_r(\lambda)=N(\lambda)
  \end{equation*}
  at all continuity points~$\lambda$ of~$N$.
\end{Theorem}

In order to proof this \cref{thm:weak1} we will use the following Lemma.
\begin{Lemma}\label{la:compt}
 Let $H:D(H)\to\ell^2(G)$ be a self-adjoint operator, where $D(H)\subseteq \ell^2(G)$ and assume that $D_0$ is core of $H$. Then for each positive constant $\kappa$ and $\xi\in\ell^2(G)$ there exists $\psi\in\ell^2(G)$ such that
\begin{align*}
 \norm[2]{\xi-\psi}<\kappa \qquad\text{and}\qquad (z-H)^{-1}\psi\in D_0.
\end{align*}
\end{Lemma}
\begin{proof}
 Let $\kappa>0$ and $\xi\in\ell^2(G)$ be given. As $D_0$ is a core of $H$, it is dense in $D(H)$ with respect to the norm $\norm[H]{\argmt}$.
  The map 
\[
z-H\from(\cD(H),\norm[H]{\argmt})\to(\ell^2(G),\norm[2]{\argmt}) 
\]
  is continuous and surjective, and so
  \begin{equation}\label{eq:lacompt:1}
    (z-H)(D_0)
      =\{\psi\in\ell^2(G)\mid(z-H)^{-1}\psi\in D_0\}
  \end{equation}
  is dense in~$\ell^2(G)$.
  This construction allows to find an element $\psi\in (z-H)(D_0)$ such that $\norm[2]{\xi-\psi}<\kappa$.
  Furthermore equation \eqref{eq:lacompt:1} shows that $(z-H)^{-1}\psi$ is compactly supported.
\end{proof}

\begin{proof}[Proof of Theorem \ref{thm:weak1}]
  We consider the set
  \begin{equation*}
    \cS:=\Bigl\{f\from\RR\to\CC\Bigm|
         \lim_{r\to\infty}\int_\RR f(x)dN_r(x)
           =\int_\RR f(x)dN(x)\Bigr\}\text.
  \end{equation*}
  $\cS$~is a linear space and closed with respect to
  $\norm[\infty]{\argmt}$-limits, because for $f_n\in\cS$
  and $f\from\RR\to\CC$ with
  $\lim_{n\to\infty}\norm[\infty]{f-f_n}=0$ we have
  \begin{align*}&
    \Bigabs{\int_\RR f(x)dN_r(x)-\int_\RR f(x)dN(x)}
    \le\Bigabs{\int_\RR f_n(x)dN_r(x)-\int_\RR f_n(x)dN(x)}\\&\quad
      +\Bigabs{\int_\RR(f(x)-f_n(x))dN_r(x)+\int_\RR(f(x)-f_n(x))dN(x)}\\&
    \le\Bigabs{\int_\RR f_n(x)dN_r(x)-\int_\RR f_n(x)dN(x)}
      +2\norm[\infty]{f(x)-f_n(x)}\\&
    \xto{r\to\infty}2\norm[\infty]{f(x)-f_n(x)}
    \text.
  \end{align*}
  Below we show
  \begin{equation}\label{RinS}
    \Res:=\{x\mapsto(z-x)^{-1}\mid z\in\CC\setminus\RR\}\subseteq\cS\text.
  \end{equation}
  Since $\cS$ is a vector space, it contains all of $\lin\Res$.
  By Cauchy's integral formula integer powers of functions in~$\Res$
  \begin{equation*}
    x\mapsto(z-x)^{-n}
      =\frac1{2\pi i}\oint_{\partial B_{\Im z/2}(z)}
        \frac{(z-\zeta)^{-n}}{\zeta-x}\,d\zeta
  \end{equation*}
  are uniform limits of linear combinations of functions in~$\Res$
  and thereby members of~$\cS$.
  To make this last fact explicit, let $r:=\abs{\Im z/2}$
  and consider for each $x\in\RR$ the smooth path
  $\gamma_z\from[0,1]\to\partial B_r(x)$, $\gamma_z(t):=z+r\e^{2\pi it}$.
  We approximate the Cauchy integral above by Riemann sums.
  For $k\in\NN$ we have
  \begin{align*}
    D_x\etdef
    \Bigabs{
      \oint_{\partial B_r(z)}\frac{(z-\zeta)^{-n}}{\zeta-x}\,d\zeta
     -\sum_{j=0}^{k-1}\frac{(z-\gamma_z(j/k))^{-n}}{\gamma_z(j/k)-x}
       \cdot\frac{\dot\gamma_z(j/k)}k
    }\\&
    =\Bigabs{\sum_{j=0}^{k-1}
      \int_{(j-1)/k}^{j/k}\frac{(z-\gamma_z(t))^{-n}}{\gamma_z(t)-x}
        \cdot\dot\gamma_z(t)\,dt
     -\frac{(z-\gamma_z(j/k))^{-n}}{\gamma_z(j/k)-x}
       \cdot\frac{\dot\gamma_z(j/k)}k
    }\\&
    \le\sum_{j=0}^{k-1}\int_{j/k}^{(j+1)/k}\Bigabs{
      \frac{(z-\gamma_z(t))^{-n}}{\gamma_z(t)-x}\cdot\dot\gamma_z(t)
     -\frac{(z-\gamma_z(j/k))^{-n}}{\gamma_z(j/k)-x}
       \cdot\dot\gamma_z(j/k)
    }\,dt\text.
  \end{align*}
  Now fix $\epsilon>0$ and define the compact set
  $K:=\bigl\{x\in\RR\bigm|\abs{z-x}\le(1+\frac{4\pi}{\epsilon r^n})r\bigr\}$.
  For all $x\in\RR\setminus K$,
  we have
  \begin{equation*}
    \Bigabs{\frac{(z-\gamma_z(t))^{-n}}{\gamma_z(t)-x}\cdot\dot\gamma_z(t)}
      \le2\pi r\frac{r^{-n}}{\abs{z-x}-r}
      \le\frac\epsilon2\text.
  \end{equation*}
  By the triangle inequality follows $D_x\le\epsilon$.
  For $x\in K$, we use uniform continuity of the map
  \begin{equation*}
    K\times[0,1]\ni(x,t)\mapsto\Gamma(x,t)
      :=\frac{(z-\gamma_z(t))^{-n}}{\gamma_z(t)-x}\cdot\dot\gamma_z(t)
  \end{equation*}
  to choose~$k$ large enough, such that
  $\abs{\Gamma(x,t)-\Gamma(x,j/k)}\le\epsilon$
  for all $x\in K$, $j\in\{0,\dotsc,k-1\}$ and $t\in[j/k,(j+1)/k]$.
  With this, $D_x\le\epsilon$ holds for all $x\in\RR$.
  We have thus established uniform convergence of the Riemann sums.

  Products of functions in~$\Res$ are in~$\cS$, too,
  because by partial fraction expansion
  they can be expressed as linear combinations of functions in~$\Res$
  and powers of those.
  Therefore, $\cS$ contains the algebra generated by $\Res$.
  Note that $\Res$ separates points, is closed under conjugation and for any $x\in\RR$ we find $f\in\Res$ with $f(x)\neq 0$. These fact allow to apply the Stone-Weierstrass theorem \cite{Debranges-59}, which implies
  $C_0(\RR,\CC)\subseteq\cS$.
  Now the claim follows from Portmanteau's theorem.
  Note that the idea to use the resolvents
  instead of polynomials as test-functions was proposed for instance in
  \cite[Section~3]{CyconFKS-08}.
  \par
  We therefore only have to show \eqref{RinS}.
  Fix $z\in\CC\setminus\RR$.  
  We extend the graph isomorphism form~\eqref{psiso} injectively to
  \begin{equation*}
    \psi'_{x,r}\from V_r\to G\text,
  \end{equation*}
  where $x\in V_r^{(0)}$.
  This map induces a projection
  \begin{equation*}
    \phi_{x,r}\from\ell^2(G)\to\ell^2(V_r)\textq,
    \phi_{x,r}(f):=f\circ\psi'_{x,r}\text.
  \end{equation*}
  We use this to transport~$A_r$ to~$\ell^2(G)$ from the point of view of~$x$:
  \begin{equation*}
    \hat A_{r,x}:=\phi_{x,r}^*A_r\phi_{x,r}\from\ell^2(G)\to\ell^2(G)
  \end{equation*}
  and set $\hat a_{r,x}(g,h):=\spr{\delta_g}{\hat A_{r,x}\delta_h}$ for $g,h\in G$.
  This operator can be interpreted as~$A_r$
  filled up with zero rows and columns.
  Therefore it obeys a block structure.
  This and $\phi_{r,x}^*\phi_{r,x}=\id_{\ell^2(G)}$
  ensures the validity of the following calculation:
  \begin{align*}
    \spr{\delta_x}{(z-A_r)^{-1}\delta_x}&
      =\spr{\delta_x}{(z-\phi_{r,x}\hat A_{r,x}\phi_{r,x}^*)^{-1}\delta_x}\\&
      =\spr{\delta_x}{\phi_{r,x}(z-\hat A_{r,x})^{-1}\phi_{r,x}^*\delta_x}\\&
      =\spr{\phi_{r,x}^*\delta_x}{(z-\hat A_{r,x})^{-1}\phi_{r,x}^*\delta_x}
      =\spr{\delta_{\id}}{(z-\hat A_{r,x})^{-1}\delta_{\id}}\text.
  \end{align*}
  On the other hand by the spectral theorem we have
  \begin{align*}
    \int_\RR(z-x)^{-1}dN(x)&
      =\spr{\delta_{\id}}{(z-A)^{-1}\delta_{\id}}\text.
  \end{align*}
  We estimate the difference
  \begin{equation*}
    D_r:=\Bigabs{\int_\RR(z-x)^{-1}dN_r(x)-\int_\RR(z-x)^{-1}dN(x)}\text.
  \end{equation*}
  Using
  \begin{align}\label{trace}
    \int_\RR(z-x)^{-1}dN_r(x)&
      =\frac1{\setsize{V_r}}\sum_{\lambda\in\sigma(A_r)}(z-\lambda)^{-1}
      =\frac1{\setsize{V_r}}\sum_{\lambda\in\sigma((z-A_r)^{-1})}\lambda\\&
      =\frac1{\setsize{V_r}}\Tr((z-A_r)^{-1})
      =\frac1{\setsize{V_r}}\sum_{x\in V_r}\spr{\delta_x}{(z-A_r)^{-1}\delta_x}
      \nonumber
  \end{align}
  and the spectral theorem for the second term, we obtain
  \begin{align*}
    D_r&
    =\Bigabs{
      \frac1{\setsize{V_r}}\sum_{x\in V_r}
        \spr{\delta_x}{(z-A_r)^{-1}\delta_x}
        -\spr{\delta_{\id}}{(z-A)^{-1}\delta_{\id}}
      }\\&
    \le\frac1{\setsize{V_r}}\sum_{x\in V_r^{(0)}}
      \abs{\spr{\delta_{\id}}{(z-\hat A_{r,x})^{-1}\delta_{\id}}
        -\spr{\delta_{\id}}{(z-A)^{-1}\delta_{\id}}}
      \\&\quad
      +\frac1{\setsize{V_r}}\sum_{x\in V_r\setminus V_r^{(0)}}
      \abs{\spr{\delta_x}{(z-A_r)^{-1}\delta_x}
        -\spr{\delta_{\id}}{(z-A)^{-1}\delta_{\id}}}\\&
    \le\sup_{x\in V_r^{(0)}}
      \abs{\spr{\delta_{\id}}
               {\bigl((z-\hat A_{r,x})^{-1}-(z-A)^{-1}\bigr)\delta_{\id}}
          }
      +\frac{2\epsilon(r)}{\abs{\Im z}}\text.
  \end{align*}
  Here we used Cauchy-Schwarz inequality,
  $\norm{(z-A)^{-1}}\le\abs{\Im z}^{-1}$ and \cref{S2}.
  Now we insert $\psi\in\ell^2(G)$
  for later optimization and use the second resolvent identity:
  \begin{align}
    D_r&
    \le\sup_{x\in V_r^{(0)}}
      \abs{\spr{\delta_{\id}}
               {\bigl((z-\hat A_{r,x})^{-1}-(z-A)^{-1}\bigr)\psi}
          }
      +\frac{2(\epsilon(r)+\norm[2]{\delta_{\id}-\psi})}{\abs{\Im z}}\nonumber\\&
    \le\sup_{x\in V_r^{(0)}}
      \abs{\spr{\delta_{\id}}
               {(z-\hat A_{r,x})^{-1}(A-\hat A_{r,x})(z-A)^{-1}\psi}
          }
          \nonumber\\&\qquad{}
      +\frac{2(\epsilon(r)+\norm[2]{\delta_{\id}-\psi})}{\abs{\Im z}}
      \nonumber\\&
    \le\frac1{\abs{\Im z}}\sup_{x\in V_r^{(0)}}\norm[2]{(A-\hat A_{r,x})(z-A)^{-1}\psi}
      +\frac{2(\epsilon(r)+\norm[2]{\delta_{\id}-\psi})}{\abs{\Im z}}\text.
    \label{detDr}
  \end{align}
  Now we choose $\psi$ in an appropriate way. Let $\kappa>0$ be arbitrary.
  Then \cref{la:compt} gives that there exists $\psi\in\ell^2(G)$ such that
\[
 \norm[2]{\delta_{\id}-\psi}<\kappa\quad\text{and}\quad \phi:=(z-A)^{-1}\psi\in D_0.
\]
  For all $r\ge3\diam(\supp\phi)$, we continue to estimate,
  using the properties of the approximation~$\hat A_{r,x}$, $x\in V_r^{(0)}$
  and the triangle inequality:
  \begin{align*}&\phantom{{}={}}
    \norm[2]{(A-\hat A_{r,x})\phi}
      =\biggl(\sum_{g\in G\setminus B_{r/3}^G}\Bigabs{\sum_{h\in\spt\phi}
        \spr{(A-\hat A_{r,x})\delta_g}{\delta_h}\phi(h)}^2\biggr)^{1/2}\\&
      \le\norm[\infty]\phi\setsize{\supp\phi}^{1/2}
        \biggl(\sum_{g\in G\setminus B_{r/3}^G}
        \sum_{h\in\spt\phi}\abs{a(g,h)-\hat a_{r,x}(g,h)}^2\biggr)^{1/2}\\&
      \le\norm[\infty]\phi\setsize{\supp\phi}^{1/2}
          \biggl(\sum_{h\in\spt\phi}\sum_{g\in G\setminus B_{r/3}^G}
          \abs{a(g,h)}^2\biggr)^{1/2}\\&\quad
        +\norm[\infty]\phi\setsize{\supp\phi}^{1/2}
          \biggl(\sum_{h\in\spt\phi}\sum_{g\in G\setminus B_{r/3}^G}
          \abs{\hat a_{r,x}(g,h)}^2\biggr)^{1/2}
        \text.
  \end{align*}
  By \cref{remark/3}, $\hat a_{r,x}(g,h)\ne0$
  with $h\in\supp\phi\subseteq B_{r/3}$ implies $g\in B_r^G$.
  Therefore
  \begin{equation}\label{eq:r/3}
    \begin{split}&\phantom{{}={}}
    \sum_{h\in\spt\phi}\sum_{g\in G\setminus B_{r/3}^G}
          \abs{\hat a_{r,x}(g,h)}^2\\&
    =\sum_{z\in\psi_{x,r}^{-1}(\spt\phi)}\sum_{y\in B_r^{V_r}\setminus B_{r/3}^{V_r}}
          \abs{a_r(y,z)}^2\\&
    \le\sum_{z\in\psi_{x,r}^{-1}(\spt\phi)}\sum_{y\in B_r^{V_r}\setminus B_{r/3}^{V_r}}
          \abs{a(\psi_{x,r}(y),\psi_{x,r}(z))}^2\\&
    \le\sum_{h\in\spt\phi}\sum_{g\in G\setminus B_{r/3}^G}\abs{a(g,h)}^2\text.
    \end{split}
  \end{equation}
  These considerations lead to
  \begin{align*}
    \norm[2]{(A-\hat A_{r,x})\phi}&
      \le2\norm[\infty]\phi\setsize{\supp\phi}^{1/2}
          \biggl(\sum_{h\in\spt\phi}\sum_{g\in G\setminus B_{r/3}^G}
          \abs{a(g,h)}^2\biggr)^{1/2}
        \text.
  \end{align*}
  By \cref{core} and \eqref{detDr},
  the difference~$D_r$ is bounded by $2(\epsilon(r)+\kappa)/\abs{\Im z}$,
  which implies
  \begin{equation*}
    \lim_{r\to\infty}\int_\RR(z-x)^{-1}dN_r(x)
      =\int_\RR(z-x)^{-1}dN(x)\text,
  \end{equation*}
  as $\kappa$ was chosen arbitrary.
\end{proof}

\section{Approximation results in the random setting}\label{secrandom}

We introduce random operators in \cref{countable}
and study the existence of their integrated density of states.
Proceeding in two steps, we first show convergence results in mean,
see \cref{mean},
and improve these to almost sure convergence in \cref{almostsure}. 
Beside this we again obtain a Pastur-Shubin-trace formula.

\subsection{Random operators on countable groups}\label{countable}
Let $G$ be a countable group.
Let $(\Omega,\cA,\PP)$ be a probability space and let
\begin{equation}\label{sym}
  \bigl\{a(x,y)=\cc{a(y,x)}\from\Omega\to\CC\bigm|x,y\in G\bigr\}
\end{equation}
be a set of random variables on~$(\Omega,\cA,\PP)$,
such that for each~$z\in G$ the random vectors $\bigl(a(x,y)\bigr)_{x,y\in G}$
and $\bigl(a(xz,yz)\bigr)_{x,y\in G}$ are identically distributed.
This means that for all finite $F\subseteq G$ and $E\in\cB(\CC^F)$
\begin{equation}\label{train}
  \PP\bigl\{\bigl(a(x,y)\bigr)_{x,y\in F}\in E\bigr\}
    =\PP\bigl\{\bigl(a(xz,yz)\bigr)_{x,y\in F}\in E\bigr\}\text.
\end{equation}
We assume further
\begin{equation}\label{summable}
  \EE\Bigl[\Bigl(\sum_{x\in G}\abs{a(x,\id)}\Bigr)^2\Bigr]
    <\infty\text.
\end{equation}
Note that \eqref{train} and \eqref{summable}
are suitable adaptions of \cref{transl,core} for the randomized setting.

Denote as before the set of complex valued and compactly supported functions on~$G$ by $D_0=\bigl\{f\in\ell^2(G)\bigm|\setsize{\supp f}<\infty\bigr\}$.
\begin{Lemma}\label{owe}
  For almost all $\omega\in\Omega$,
  \emph{matrix operator}~$\tilde A^{(\omega)}$, acting on~$D_0$ via
  \begin{equation}\label{deftildeA}
    (\tilde A^{(\omega)} f)(x):=\sum_{y\in G}a^{(\omega)}(x,y)f(y)
    \quad(x\in G)\text,
  \end{equation}
  is well defined.
  The family $\tilde A:=(\tilde A^{(\omega)})_{\omega\in\Omega}$
  is a random operator and satisfies for each $x\in G$
  \begin{equation}\label{approximable}
    \EE\bigl[\norm[2]{\tilde A\delta_x}^2\bigr]
      =\EE\bigl[\norm[2]{\tilde A\delta_{\id}}^2\bigr]
      \le\EE\bigl[\norm[1]{\tilde A\delta_{\id}}^2\bigr]
      <\infty\text.
  \end{equation}
\end{Lemma}
\begin{proof}
  According to \cite{PasturFigotin1992},
  an operator valued function $\tilde A\from\omega\mapsto\tilde A^{(\omega)}$
  with a common core~$D_0$ is measurable, if the functions
  \begin{equation*}
    \omega\mapsto\spr v{\tilde A^{(\omega)}w}
      =\lim_{r\to\infty}\sum_{x\in B_r^G(\id)}v_x
        \sum_{y\in\supp w}a^{(\omega)}(x,y)w_y
  \end{equation*}
  are measurable for all $v\in\ell^2(G)$ and $w\in D_0$.
  Since limits of sum of random variables is again measurable,
  $\tilde A$ is measurable.
  \par
  \Cref{approximable} follows directly from~\eqref{train} and~\eqref{summable}:
  \begin{align*}
    \EE\bigl[\norm[2]{\tilde A\delta_x}^2\bigr]&
      \le\EE\bigl[\norm[1]{\tilde A\delta_x}^2\bigr]
      =\EE\Bigl[\Bigl(\sum_{z\in G}\abs{a(z,\id)}\Bigr)^2\Bigr]
      <\infty\text.
  \end{align*}
  A direct consequence is $\PP\{\norm[2]{\tilde A\delta_x}=\infty\}=0$
  for all $x\in G$,
  which implies that $\tilde A$~is $\PP$-a.s.\ well defined on~$D_0$.
\end{proof}

By \eqref{train} and \eqref{summable},
the random operator $\tilde A$ has the following properties.
For all $f,g\in D_0$, we have
\begin{enumerate}[({R}1)]
  \item\label[rassum]{rtrain}
    For all $x\in G$ and the corresponding translations
    $\tau_x\from\ell^2(G)\to\ell^2(G)$, $\tau_xf(y):=f(yx^{-1})$
    holds
    \begin{equation*}
      \EE\bigl[\spr{\tilde A\tau_xf}{\tau_xg}\bigr]
        =\EE\bigl[\spr{\tilde Af}g\bigr]\text,
    \end{equation*}
  \item\label[rassum]{rapprox}
    $\EE\bigl[\norm[1]{\tilde Af}^2\bigr]<\infty$ and
  \item\label[rassum]{rsymm}
    $\spr{\tilde A^{(\omega)}f}g=\spr f{\tilde A^{(\omega)}g}$
      for all $\omega\in\Omega$.
\end{enumerate}

Since our operators are up to now only defined on~$D_0$,
we need the multiplication operators $\mo F\from\ell^2(G)\to D_0$,
$\mo F(f):=\ifu Ff$ with the indicator functions $\ifu F$
of finite sets~$F\subseteq G$.
The following lemma is adapted from \cite[Proposition~4.1]{PasturFigotin1992}.
\begin{Lemma}\label{figotinsTrick}
  Let the random operators $A,B\from D_0\to\ell^2(G)$
  satisfy the properties~\labelcref{rtrain} and~\labelcref{rapprox},
  with joint translational invariance.
  Then, for all $x\in G$,
  \begin{equation*}
    \EE\bigl[\norm[2]{A\pi_{B_r^G}B\delta_x}^2\bigr]
      \le\norm[\infty]B^2\EE\bigl[\norm[1]{A\delta_{\id}}^2\bigr]\text,
  \end{equation*}
  where $\norm[\infty]{B}$ is the $\ell^\infty$-norm of the random variable
  $\omega\mapsto\norm{B^{(\omega)}}$ with operator norm.
\end{Lemma}
The operators $A$ and $B$ are \emph{jointly translationally invariant}, if
\begin{equation*}
  \PP\bigl\{\bigl(a(x,y),b(x,y)\bigr)_{x,y\in F}\in E\bigr\}
    =\PP\bigl\{\bigl(a(xz,yz),b(xz,yz)\bigr)_{x,y\in F}\in E\bigr\}
\end{equation*}
for all $z\in G$, $F\subseteq G$ finite and $E\in\cB(\CC^F\times\CC^F)$.
\begin{proof}
  Let $a(x,y):=\spr{\delta_x}{A\delta_y}$ and
  $b(x,y):=\spr{\delta_x}{B\delta_y}$, $x,y\in G$,
  be the matrix elements of~$A$ and~$B$.
  Then with Parseval's identity we get
  \begin{align*}
    \EE\bigl[\norm[2]{A\pi_{B_r^G}B\delta_x}^2\bigr]&
      =\EE\bigl[\spr{A\pi_{B_r^G}B\delta_x}{A\pi_{B_r^G}B\delta_x}\bigr]\\&
      \le\sum\nolimits_{y,z\in G}\EE
        \bigl[\abs{\spr{A\delta_y}{A\delta_z}b(y,x)b(x,z)}\bigr]\\&
      =\sum\nolimits_{y,z}\EE\bigl[\abs{\spr{A\delta_{yz^{-1}}}{A\delta_{\id}}
        b(yz^{-1},xz^{-1})b(xz^{-1},\id)}\bigr]\text,
  \intertext{where we used joint translational invariance of $A$ and $B$.
    We reindex the sums to rearrange the terms
    and use the Cauchy Schwarz inequality:}
    \dots&
      =\sum\nolimits_{y',z}\EE\bigl[\abs{\spr{A\delta_{y'}}{A\delta_{\id}}
        b(y',xz^{-1})b(xz^{-1},\id)}\bigr]\\&
      =\sum\nolimits_{y'}\EE\bigl[\abs{\spr{A\delta_{y'}}{A\delta_{\id}}}
        \sum\nolimits_{z'}\abs{b(y',z')b(z',\id)}\bigr]\\&
      \le\sum\nolimits_{y'}\EE\bigl[\abs{\spr{A\delta_{y'}}{A\delta_{\id}}}
        \norm[2]{B\delta_{y'}}\norm[2]{B\delta_{\id}}\bigr]\\&
      \le\norm[\infty]B^2\sum\nolimits_{y'}
        \EE\bigl[\abs{\spr{A\delta_{y'}}{A\delta_{\id}}}\bigr]\text.
  \intertext{Now, as $B$ is out of the way, we expand the scalar product
    and use translation invariance again and reindex the sums once more:}
    \dots&
      =\norm[\infty]B^2\sum\nolimits_{y'}\EE\Bigl[
        \Bigabs{\sum\nolimits_{z}{a(z,y')}{a(z,\id)}}\Bigr]\\&
      =\norm[\infty]B^2\sum\nolimits_{y'}\EE\Bigl[
        \Bigabs{\sum\nolimits_{z}a(\id,y'z^{-1})a(\id,z^{-1})}\Bigr]\\&
      \le\norm[\infty]B^2\sum\nolimits_{y',z}\EE\bigl[
        \abs{a(\id,y'z^{-1})a(\id,z^{-1})}\bigr]\\&
      =\norm[\infty]B^2\EE\Bigl[\sum\nolimits_{y''}\abs{a(\id,y'')}
        \sum\nolimits_{z'}\abs{a(\id,z')}\Bigr]\\&
      =\norm[\infty]B^2\EE\Bigl[\Bigl(\sum\nolimits_y
        \abs{a(\id,y)}\Bigr)^2\Bigr]
      =\norm[\infty]B^2\EE\bigl[\norm[1]{A\delta_{\id}}^2\bigr]\text.
      \qedhere
  \end{align*}
\end{proof}

\begin{Proposition}\label{esa}
  There exists $\tilde\Omega\in\cA$ of full measure such that for all $\omega\in\tilde\Omega$
  the operator~$\tilde A^{(\omega)}$ defined in~\eqref{deftildeA}
  is essentially self-adjoint.
  Denote the self-adjoint extension of $\tilde A^{(\omega)}$ by $\bar A^{(\omega)}$ and define the random operator $A=(A^{(\omega)})$ for all $\omega\in\Omega$ by
\[
 A^{(\omega)}:=\begin{cases}
              \bar A^{(\omega)}&\text{ if }\omega\in\tilde\Omega \\
              {\rm Id} &\text{ else.}
             \end{cases}
\]
  Then the resolvents $\omega\mapsto(z-A^{(\omega)})^{-1}$,
  $z\in\CC\setminus\RR$, are strongly measurable.
\end{Proposition}
\begin{proof}
  The proof of \cite[(4.2)~Theorem]{PasturFigotin1992}
  for essential self-adjointness generalizes to our more general setting.
  \par
  The operators are symmetric, see \cref{rsymm}.
  By the basic criterion for self-adjointness,
  \cite[Theorem VIII.3, p.~256f]{ReedSimon:1},
  we have to show that, with probability~$1$, $(z-\tilde A^{(\omega)})D_0$
  is dense in~$\ell^2(G)$ for all $z\in\CC\setminus\RR$.
  For this, it suffices to approximate~$\delta_g$ for arbitrary $g\in G$.
  \par
  Define for $r\in\NN$ the bounded random matrix operator~$\tilde A_r^{(\omega)}$,
  acting on~$D_0$, via its matrix elements
  \begin{equation*}
    \tilde a_r^{(\omega)}(x,y)
      :=\ifu{[0,r]}\bigl(\abs{a^{(\omega)}(x,y)}\bigr)
        \ifu{B_r^G(x)}(y)a^{(\omega)}(x,y)\text.
  \end{equation*}
  Using Lebesgue's dominated convergence theorem,
  continuity of squaring and monotone convergence,
  we see that the operators~$\tilde A_r^{(\omega)}$
  approximate~$\tilde A^{(\omega)}$ in the following sense:
  \begin{align*}
    \lim_{r\to\infty}\EE\bigl[\norm[1]{(\tilde A-\tilde A_r)\delta_{\id}}^2\bigr]&
      =\lim_{r\to\infty}\EE\Bigl[\Bigl(\sum_{g\in G}\abs{a(g,\id)-\tilde a_r(g,\id)}\Bigr)^2\Bigr]\\&
      =\EE\Bigl[\Bigr(\sum_{g\in G}\lim_{r\to\infty}\abs{a(g,\id)-\tilde a_r(g,\id)}\Bigr)^2\Bigr]
      =0\text,
  \end{align*}
  since 
  \[
\sum_{g\in G}\abs{a^{(\omega)}(g,\id)-\tilde a_r^{(\omega)}(g,\id)}
    \le2\sum_{g\in G}\abs{a^{(\omega)}(g,\id)}
    =2\norm[1]{\tilde A^{(\omega)}\delta_{\id}}   
  \]
  uniformly in $r\in\NN$ and
  $\EE\bigl[\norm[1]{2\tilde A\delta_{\id}}^2\bigr]<\infty$.
  Consider for $r,n\in\NN$ the element
  $f_{g,r,n}^{(\omega)}
    :=\mo{B_n^G(\id)}(z-\tilde A_r^{(\omega)})^{-1}\delta_g\in D_0$.
  We will show, that its image under $z-\tilde A^{(\omega)}$
  converges to~$\delta_g$ for a suitable limit in $n$ and $r$.
  Therefore, we estimate
  \begin{equation}\label{convergeceestimate}
  \begin{split}&\phantom{{}={}}
    \norm[2]{(z-\tilde A^{(\omega)})f_{g,r,n}^{(\omega)}-\delta_g}\\&
      =\bignorm[2]{(z-\tilde A^{(\omega)})\mo{B_n^G}(z-\tilde A_r^{(\omega)})^{-1}\delta_g\\&\qquad\quad
        -(z-\tilde A_r^{(\omega)})(\mo{B_n^G}+\mo{G\setminus B_n^G})(z-\tilde A_r^{(\omega)})^{-1}\delta_g}
          \\&
      \le\norm[2]{(\tilde A_r^{(\omega)}-\tilde A^{(\omega)})f_{g,r,n}^{(\omega)}}
        +C(r)\norm[2]{\mo{G\setminus B_n^G}(z-\tilde A_r^{(\omega)})^{-1}\delta_g}
        \qtext{a.s.}
  \end{split}
  \end{equation}
  with $C(r):=\sup_{\omega\in\Omega}\norm[2]{z-\tilde A_r^{(\omega)}}<\infty$.
  Note $\lim_{n\to\infty}\norm[2]{\mo{G\setminus B_n^G}
      (z-\tilde A_r^{(\omega)})^{-1}\delta_g}=0$ and
  \begin{equation*}
    \norm[2]{\mo{G\setminus B_n^G}
      (z-\tilde A_r^{(\omega)})^{-1}\delta_g}\le\abs{\Im z}^{-1}
  \end{equation*}
  uniformly in $n\in\NN$.
  By Lebesgue, there exists for all $r>0$ an $\tilde n=\tilde n(r,g)\in\NN$
  such that
  \begin{equation}\label{part2}
    C(r)\EE\bigl[\norm[2]{\mo{G\setminus B_{\tilde n}^G}
      (z-\tilde A_r)^{-1}\delta_g}\bigr]
      \le1/r\text.
  \end{equation}
  Also in expectation, the first summand in
  \eqref{convergeceestimate} can be controlled by \cref{figotinsTrick}
  and $\norm{(z-\tilde A_r^{(\omega)})^{-1}}\le\abs{\Im z}^{-1}$ with the bound
  \begin{align*}
    \bigl(\EE[\norm[2]{(\tilde A_r-\tilde A)f_{g,r,\tilde n}}]\bigr)^2&
      \le\EE\bigl[\norm[2]{(\tilde A_r-\tilde A)f_{g,r,\tilde n}}^2\bigr]\\&
      \le\norm[\infty]{(z-\tilde A_r)^{-1}}^2
        \EE\bigl[\norm[1]{(\tilde A_r-\tilde A)\delta_{\id}}^2\bigr]\\&
      \le\abs{\Im z}^{-2}\EE\bigl[\norm[1]{(\tilde A_r-\tilde A)\delta_{\id}}^2\bigr]
      \xto{r\to\infty}0\text.
  \end{align*}
  Together with~\eqref{part2} we infer from~%
  \eqref{convergeceestimate}
  \begin{equation*}
    \EE\bigl[\norm[2]{(z-\tilde A)f_{g,r,\tilde n}-\delta_g}\bigr]
      \xto{r\to\infty}0\text,
  \end{equation*}
  i.e.\ convergence in~$L^1$.
  Thereby we find a subsequence $(r_k)_{k\in\NN}$ and a set $\tilde \Omega\subseteq \Omega$ of full measure, such that for all $\omega\in\tilde\Omega$
  \begin{equation*}
    \lim_{k\to\infty}\norm[2]{(z-\tilde A^{(\omega)})
      f_{g,r_k,\tilde n(r_k,g)}^{(\omega)}-\delta_g}=0\text,
  \end{equation*}
  and essential self-adjointness of~$\tilde A^{(\omega)}$ is shown.
  \par
  For the measurability statement fix $z\in\CC\setminus\RR$
  and denote for $\omega\in\Omega$ by~$A_r^{(\omega)}\from\ell^2(G)\to\ell^2(G)$
  the self-adjoint operator with matrix elements
  \begin{equation*}
    a_r^{(\omega)}(x,y)
      :=\begin{cases}\ifu{[0,r]}\bigl(\abs{a^{(\omega)}(x,y)}\bigr)
        \ifu{B_r^G(\id)^2}(x,y)a^{(\omega)}(x,y) &\text{if }\omega\in\tilde\Omega\\
        \delta_x(y)\ifu{B_r^G(\id)^2}(x,y) &\text{else.}
\end{cases}
  \end{equation*}
  Note that this is operator is not translationally invariant in distribution
  and not the closure of~$\tilde A_r^{(\omega)}$
  from above, but has only finitely many non-zero matrix elements.
  \par
  By Cramer's rule, the resolvent $\omega\mapsto(z-A_r^{(\omega)})^{-1}$
  is weakly measurable.
  Of course, $\ell^2(G)$ is separable, so by Pettis' measurability theorem
  the resolvent is actually strongly measurable.
  We will now show that the resolvents of~$A_r$
  converge strongly to the corresponding resolvents of~$A$.
  This will show measurability of~$A$.
  \par
  Since $A^{(\omega)}$ and $A_r^{(\omega)}$ are self-adjoint,
  $\norm{(z-A^{(\omega)})^{-1}}\le1/\abs{\Im z}$
  and analogously for $A_r^{(\omega)}$.
  Therefore, fix some $\omega\in\Omega$, $\xi\in\ell^2(G)$ and $\kappa>0$.
  By \cref{la:compt} there exists $\psi\in\ell^2(G)$ with
\[
 \norm[2]{\xi-\psi}<\kappa\quad\text{and}\quad \phi:=(z-A^{(\omega)})^{-1}\psi\in D_0.
\]
  Equipped with these tools and the second resolvent identity we see
  \begin{align*}&\phantom{{}={}}
    \bignorm[2]{\bigl((z-A^{(\omega)})^{-1}-(z-A_r^{(\omega)})^{-1}\bigr)\xi}\\&
      \le\bignorm[2]{\bigl((z-A^{(\omega)})^{-1}-(z-A_r^{(\omega)})^{-1}\bigr)\psi}
        +2\norm[2]{\xi-\psi}/\abs{\Im z}\\&
      \le\norm[2]{(z-A_r^{(\omega)})^{-1}(A^{(\omega)}-A_r^{(\omega)})(z-A^{(\omega)})^{-1}\psi}
        +2\kappa/\abs{\Im z}\\&
      \le\bigl(\norm[2]{(A^{(\omega)}-A_r^{(\omega)})\phi}
        +2\kappa\bigr)/\abs{\Im z}.
  \end{align*}
  Now we use that for all $x,y\in G$
  \begin{equation*}
    a_r^{(\omega)}(x,y)\xto{r\to\infty}a^{(\omega)}(x,y)
  \end{equation*}
  to obtain
  \begin{equation*}
    \limsup_{r\to\infty}\bignorm[2]{\bigl((z-A^{(\omega)})^{-1}
      -(z-A_r^{(\omega)})^{-1}\bigr)\xi}
      \le2\kappa/\abs{\Im z}.
  \end{equation*}
  As $\kappa>0$ was arbitrary, this concludes the proof.
\end{proof}

Motivated by the examples in \cref{exa},
we are interested in a special case of matrix operators,
which we will define now.

Let $\cP_{1,2}:=\{e\subseteq G\mid\setsize e\in\{1,2\}\}$
be the set of all edges of the full graph with vertices~$G$,
and let further $X_e\from\Omega\to\RR$, $e\in\cP_{1,2}$,
be independent random variables such that for each $g\in G$
the random variables in
\begin{equation}\label{id_distr}
  \{X_{\{x,y\}}\mid x,y\in G,xy^{-1}=g\}
\end{equation}
are identically distributed.
We require further
\begin{equation}\label{momentsOfX}
  \EE\Bigl[\Bigl(\sum_{x\in G}\abs{X_{\{\id,x\}}}\Bigr)^2\Bigr]
    <\infty\text.
\end{equation}
The matrix element of~$\tilde A\from D_0\to\ell^2(G)$
for $x,y\in G$ is parametrized by $\alpha\in\RR$ and given by
\begin{equation}\label{laplaceelements}
  a(x,y)
  :=a_{\alpha}(x,y)
  :=X_{\{x,y\}}-\alpha\delta_x(y)\sum\limits_{z\in G\setminus\{x\}}X_{\{x,z\}}
\end{equation}
Since the moment condition~\eqref{momentsOfX} implies \eqref{summable}
and
\begin{align*}
  \EE\bigl[\norm[1]{\tilde A\delta_{\id}}^2\bigr]&
    =\EE\Bigl[\Bigl(\sum_{x\in G}\abs{a(x,\id)}\Bigr)^2\Bigr]
    \le\EE\Bigl[\Bigl(\sum_{x\in G}\abs{X_{\{\id,x\}}}
      +\alpha\sum\limits_{z\in G}\abs{X_{\{\id,z\}}}\Bigr)^2\Bigr]\\&
    =(1+\alpha)^2\EE\Bigl[\Bigl(\sum_{x\in G}\abs{X_{\{\id,x\}}}\Bigr)^2\Bigr]
    <\infty\text,
\end{align*}
the matrix operator $\tilde A$, acting on $D_0$ via~\eqref{deftildeA},
is well-defined and almost surely essentially self-adjoint, see \cref{esa}.
The operator~$A^{(\omega)}$ given as in \cref{esa},
which is for almost all $\omega$ the closure of $\tilde A^{(\omega)}$,
is self-adjoint for all $\omega$ and by construction and \eqref{summable} fulfills
\cref{rtrain,rapprox}.

In the case $\alpha=0$ the operator is an adjacency matrix
on graphs with vertices in~$G$ and weights on the edges.
For $\alpha=1$ and $X_{\lbrace x\rbrace}=0$
a.s.\ we cover random weighted Laplace operators on such graphs.
More generally, one interprets the diagonal terms $X_{\lbrace x\rbrace}$ as random potential. This setting is well studied under the term \emph{Anderson model}.
For all these reasons we call the operators~$A^{(\omega)}$
\emph{random Hamiltonians}.

The next well-known lemma gives conditions for boundedness and unboundedness of the operators in consideration.
\begin{Lemma}\label{la:bdd}
  Let $A$ be the random Hamiltonian defined above
  with the random variables $X_{\lbrace x,y\rbrace}$, $x,y\in G$, and
  $D:=\sup_{x\in G}\norm[\infty]{X_{\lbrace\id,x\rbrace}}\in[0,\infty]$.
\begin{enumerate}[(i)]
  \item If $D=\infty$, then $A$ is unbounded.
  \item If $D<\infty$ and $A$ is of finite hopping range,
    i.e.\ $X_{\lbrace x,y\rbrace}=0$ whenever $d(x,y)\ge R$
    for some fixed $R\in\NN$, then $A$ is bounded.
\end{enumerate}
\end{Lemma}
\begin{proof}
  Let $D=\infty$ and let $m>0$ be given. Note that condition \eqref{momentsOfX} implies that
\[
 k:=\EE\Big(\sum_{z\neq \id}\abs{X_{\{\id,z\}}}\Big)<\infty.
\]
 Without loss of generality we assume $m\geq 2k\abs{\alpha}$. As $D$ is assumed to be infinite there exists $z\in G$ such that 
\begin{equation}\label{eq:la:bdd:1}
\norm[\infty]{X_{\{\id,z\}}} \geq m. 
\end{equation}
Let us distinguish two case. In the first case we consider the situation where there exists $z\in G\setminus\{\id\}$ satisfying \eqref{eq:la:bdd:1}. Then obviously the probability $\PP(a(\id,z)\geq m)$ is strictly positive.
In the case where there exists no $z\in G\setminus\{\id\}$ satisfying \eqref{eq:la:bdd:1} the same holds true, however we need a short calculation to see this. In this situation we have $\norm[\infty]{X_{\{\id\}}}=\infty$. By definition of $a(\id,\id)$ we have by triangle inequality
\begin{align*}
 \PP(\abs{a(\id,\id)}\geq m)
&\geq \PP\bigg( \abs{X_{\{\id\}}}- \Bigabs{\alpha \sum_{z\in G\setminus\{\id\}} X_{\{\id, z\}}} \geq m \bigg) \\
&\geq \PP\bigg( \abs{X_{\{\id\}}}\geq 2m, \  \Bigabs{\alpha \sum_{z\in G\setminus\{\id\}} X_{\{\id, z\}}} \leq m \bigg)\\
&= \PP\big( \abs{X_{\{\id\}}}\geq 2m\big) \PP \bigg( \Bigabs{\alpha \sum_{z\in G\setminus\{\id\}} X_{\{\id, z\}}} \leq m \bigg)
\end{align*}
As $\norm[\infty]{X_{\{\id\}}}=\infty$ we get $\PP(\abs{X_{\{\id\}}}\geq 2m)>0$. We use Tschebyscheff inequality to obtain
\begin{align*}
 \PP \bigg( \Bigabs{\alpha \sum_{z\in G\setminus\{\id\}} X_{\{\id, z\}}} \leq m \bigg)\geq 1-\frac{\abs{\alpha}}{m}\EE\bigg(\sum_{z\in G\setminus\{\id\}} X_{\{\id,z\}} \bigg)\geq \frac12
\end{align*}
This gives $\PP(\abs{a(\id,\id)}\geq m)>0$. Together with the previous case we showed that whenever $D=\infty$ there exists $z\in G$ such that $\PP(\abs{a(\id,z)}>m)$ is positive.
Furthermore, by construction  we have that the random variables $a(x,zx)$, $x\in G$ are independent and identical distributed, such that we get
  \[
    \sum_{x\in G}\PP(\abs{a(x,zx)}>m)
      =\infty\text.
  \]
  Now Borel-Cantelli gives that for almost all $\omega\in\Omega$
  there are infinitely many~$x\in G$ such that $\abs{a^{(\omega)}(x,zx)}>m$.
  For each such $\omega$, we choose one of these~$x$ and obtain
  $(A^{(\omega)}\delta_{zx})(x)=a^{(\omega)}(x,zx)$.
  Hence 
  \begin{equation*}
    \norm{A^{(\omega)}}
      \ge\norm[2]{A^{(\omega)}\delta_{zx}}
      \ge m\text.
  \end{equation*}

  Let $D<\infty$ and~$A$ be of finite hopping range~$R$.
  We set $m:=(1+\abs{\alpha}\setsize{B_R})D$. Then we have
\begin{align*}
  \PP(\exists x,y\in G\text{ with }a(x,y)\ge m)&
    =\PP\Bigl(\Union_{x,y\in G}\bigl\{\omega\in\Omega\mid a(x,y)\ge m\bigr\}\Bigr)\\&
    \le\sum_{x,y\in G}\PP\bigl(\bigl\{\omega\in\Omega\mid a(x,y)\ge m\bigr\}\bigr)
    =0\text.
\end{align*}
Using this we get for $f\in\ell^2(G)$ and almost all realizations $\omega\in\Omega$
\begin{align*}
    \norm[2]{A^{(\omega)}f}^2&
      =\sum_{v\in G}\Bigabs{\sum_{w\in B_R^G(v)} a^{(\omega)}(v,w)f(w)}^2
      \le \sum_{v\in G} m^2 \Bigabs{\sum_{w\in B_R^G(v)} \abs{f(w)}}^2 \\&
      \le\sum_{v\in G} m^2\setsize{B_R^G}\sum_{w\in B_R^G(v)}\abs{f(w)}^2 
      \le m^2\abs{B_R^G}^2\norm[2]f^2\text.\qedhere
  \end{align*}
\end{proof}

\subsection{Convergence in mean}\label{mean}
From now on we restrict ourselves to infinite sofic groups~$G$,
generated by a finite and symmetric generating system~$S$.
Furthermore, the operators in consideration are random Hamiltonians~%
$A=(A^{(\omega)})_{\omega\in\Omega}$
as defined right before Lemma \ref{la:bdd}.

The first step is to define finite dimensional approximations to~$A$.
We consider the approximating graphs $\Gamma_r$, $r\in\NN$,
and use the simplified notation~\eqref{eq:simplify}.
Recall the map~$\psi_{x,r}$ from~\eqref{psiso}.
For the construction of the approximating operators~$A_r$, $r\in\NN$,
we choose enumerations of
$V_r^{(0)}=\{v_1,\dotsc,v_{k_r}\}$, $k_r:=\setsize{V_r^{(0)}}$.
We set for
$e'\in\cP_{1,2}^{(r)}:=\{e'\subseteq V_r\mid\setsize{e'}\in\{1,2\}\}$
\begin{equation*}
  j_r(e')=\max\{0,j\in\{1,\dotsc,k_r\}\bigm|
    e'\subseteq B_{\rho(r)}^{V_r}(v_j)\bigr\}
\end{equation*}
with $\rho\from\NN\to\RR$ satisfying $\rho(r)\le\frac r3$
and $\rho(r)\xto{r\to\infty}\infty$.

We now realise for each $r\in\NN$, $j\in\{1,\dotsc,k_r\}$
and $e'=\{v,w\}\in\cP_{1,2}^{(r)}$ such that $e'\subseteq B_r^{V_r}(v_j)$,
a random variable~$X_{e'}^{r,j}$
with the same distribution as~$X_{\lbrace\psi_{v_j,r}(v),\psi_{v_j,r}(w)\rbrace}$.
We also want
\begin{equation*}
  \bigl\{X_e,X_{e'}^{r,j}
    \bigm|e\in\cP_{1,2}
    ,r\in\NN
    ,j\in\{1,\dotsc,k_r\}
    ,e'\in\cP_{1,2}^{(r)}
    \bigr\}
\end{equation*}
to be independent.

Analogous to~\eqref{laplaceelements},
we define the operator~$A_r^{(\omega)}\from \ell^2(V_r)\to\ell^2(V_r)$,
$\omega\in\Omega$,
with $\alpha\in\RR$ by its matrix elements for $x,y\in V_r$ via
\begin{equation*}
  a_r^{(\omega)}(x,y):=
    \begin{cases}
      X_{\{x,y\}}^{r,j}(\omega)-\alpha\delta_x(y)
        \sum\limits_{z\in V_r\setminus\{x\}}X_{\{x,z\}}^{r,j}(\omega)
        &\text{if $j:=j_r(\{x,y\})>0$,}\\
      0 &\text{else.}
    \end{cases}
\end{equation*}
Note, that $A_r^{(\omega)}$ has hopping range~$2\rho(r)$,
i.e.\ $a_r(x,y)=0$, as soon as $d_r(x,y)>2\rho(r)$.

The operator~$A_r^{(\omega)}$ is symmetric and hence self-adjoint.
Let $x,y\in V_r$ and $1\le i<j\le k_r$ be such that $x,y\in B_{\rho(r)}^{V_r}(v_i)\cap B_{\rho(r)}^{V_r}(v_j)$.
Then we have by \cref{la:welldef}
\[
  \psi_{v_i,r}(x)(\psi_{v_i,r}(y))^{-1}
    =\psi_{v_j,r}(x)(\psi_{v_j,r}(y))^{-1},
\]
which gives that the random variables 
\[
  X_{\{\psi_{v_i,r}(x),\psi_{v_i,r}(y)\}}^{r,i}\quad\text{and}\quad
  X_{\{\psi_{v_j,r}(x),\psi_{v_j,r}(y)\}}^{r,j}
\]
are identically distributed.
Therefore the joint distribution of the matrix elements of~$A_r$
does not dependent on the specific choice of the enumeration
$\{v_1,\dotsc,v_{k_r}\}=V_r^{(0)}$.

As in \cref{dete} we define eigenvalue counting functions.
For each $\omega\in\Omega$ and $r\in\NN$ we set
\[
  N_r^{(\omega)}\from\RR\to[0,1]\textq,
  N_r^{(\omega)}(\lambda):=\frac
    {\setsize{\{\text{eigenvalues of $A_r^{(\omega)}$ not larger than $\lambda$}\}}}
    {\abs{V_r}}\text,
\]
where the eigenvalues are counted with multiplicity.
Beside this we define again $N^{(\omega)}\from\RR\to[0,1]$ as
\begin{equation}\label{Nomega}
  N^{(\omega)}(\lambda):=\spr{E_\lambda^{(\omega)}\delta_{\id}}{\delta_{\id}}\text,
\end{equation}
where again $E_\lambda^{(\omega)}$ is the spectral projection of~$A^{(\omega)}$ on the interval $(-\infty,\lambda]$.
We furthermore define the functions $\bar N_r,\bar N\from \RR\to[0,1]$ by
\begin{equation}\label{Nbar}
  \bar N(\lambda)=\EE[N(\lambda)]\quad\text{and}\quad
  \bar N_r(\lambda)=\EE[N_r(\lambda)]\qquad(\lambda\in\RR,r\in\NN)\text.
\end{equation}
As before in the deterministic setting, the function $\bar N$ is called \emph{spectral distribution function} of the random operator $A$.
\begin{Theorem}\label{thm:expected}
  Let $\bar N_r,\bar N\from \RR\to[0,1]$ be as above.
  Then
  \[
    \lim_{r\to\infty}\bar N_r(\lambda)
      =\bar N(\lambda)
  \]
  at all continuity points~$\lambda$ of~$\bar N$.
\end{Theorem}
\begin{proof}
  As in the proof of \cref{thm:weak1} we need to show that for all $z\in\CC\setminus\RR$ 
  \begin{equation}
    \lim_{r\to\infty}\int_\RR(z-x)^{-1}d\bar N_r(x)
      =\int_\RR (z-x)^{-1}d\bar N(x)\text.
  \end{equation}
  The integral on the left hand side is actually a finite sum.
  By linearity we have
  \begin{equation*}
    \int_\RR(z-x)^{-1}d\bar N_r(x)
      =\EE\left[\int_\RR(z-x)^{-1}dN_r(x)\right]
    \text.
  \end{equation*}
  By definition, the Riemann-Stieltjes integral can be written as a limit of a sum. Using the boundedness of $x\mapsto (z-x)^{-1}$,
  and dominated convergence allows to interchange limits and integration, such that we obtain
  \[
   \int_{\RR}(z-x)^{-1}d\bar N(x)
        =\EE\left[\int_{\RR}(z-x)^{-1}dN(x)\right]\text.
  \]

  Therefore we fix $z\in\CC\setminus\RR$ and consider the following difference
  using spectral theorem and the calculation from~\eqref{trace}:
  \begin{align*}
    D_r\etdef\Bigabs{\EE\Bigl[\int_\RR(z-x)^{-1}dN_r(x)\Bigr]
      -\EE\Bigl[\int_\RR(z-x)^{-1}dN(x)\Bigr]}\\
      &=\Bigabs{\EE\Bigl[\frac{1}{\setsize{V_r}}\sum_{x\in V_r}\spr{\delta_x}{(z-A_r)^{-1}\delta_x}\Bigr]
        -\EE\bigl[\spr{\delta_{\id}}{(z-A)^{-1}\delta_{\id}}\bigr]}\\
      &\le\frac1{\setsize{V_r}}\sum_{x\in V_r^{(0)}}\left|\EE\bigl[\spr{\delta_x}{(z-A_r)^{-1}\delta_x}\bigr]
        -\EE\bigl[\spr{\delta_{\id}}{(z-A)^{-1}\delta_{\id}}\bigr]\right|\\
      &\quad+\frac1{\setsize{V_r}}\sum_{x\in V_r\setminus V_r^{(0)}}\left|\EE\bigl[\spr{\delta_x}{(z-A_r)^{-1}\delta_x}\bigr]
        -\EE\bigl[\spr{\delta_{\id}}{(z-A)^{-1}\delta_{\id}}\bigr]\right|\\
      &\le\sup_{x\in V_r^{(0)}}\bigabs{\EE\bigl[\spr{\delta_x}{(z-A_r)^{-1}\delta_x}\bigr]
        -\EE\bigl[\spr{\delta_{\id}}{(z-A)^{-1}\delta_{\id}}]\bigr]}
        +\frac{2\epsilon(r)}{\abs{\Im z}}
      \text.
  \end{align*}
  Here we used $\norm{(z-H)^{-1}}\le\abs{\Im z}^{-1}$ for self-adjoint~$H$.
  Again we use the construction of the proof of \cref{thm:weak1}
  of the operator $A_{r,x}^{(\omega)}$.
  As before we extend the graph isomorphism~$\psi_{x,r}$ at $x\in V_r^{(0)}$
  from~\eqref{psiso} injectively to
  \begin{equation*}
    \psi'_{x,r}\from V_r\to G\text.
  \end{equation*}
  This map induces a projection
  \begin{equation*}
    \phi_{x,r}\from\ell^2(G)\to\ell^2(V_r)\textq,
    \phi_{x,r}(f):=f\circ\psi'_{x,r}\text.
  \end{equation*}
  We use this to transport~$A_r^{(\omega)}$ to~$\ell^2(G)$
  from the point of view of~$x$:
  \begin{equation*}
    A_{r,x}^{(\omega)}:=\phi_{x,r}^*A_r^{(\omega)}\phi_{x,r}
      \from\ell^2(G)\to\ell^2(G)\text.
  \end{equation*}
  As before we have for all $x\in V_r^{(0)}$
  \begin{align*}
    \spr{\delta_x}{(z-A_r^{(\omega)})^{-1}\delta_x}&
      =\spr{\delta_{\id}}{(z-A_{r,x}^{(\omega)})^{-1}\delta_{\id}}\text.
  \end{align*}
  Now for each $x\in V_r^{(0)}$ we define a new approximating operator
  $\hat A_{r,x}^{(\omega)}\from\ell^2(G)\to\ell^2(G)$ of~$A^{(\omega)}$
  by its matrix elements
  \begin{equation*}
    \hat a_{r,x}^{(\omega)}(g,h)
      :=\begin{cases}
        a^{(\omega)}(g,h)&\text{ if }g,h\in B_{\rho(r)}^G,g\ne h\\
        X_{\{g\}}(\omega)-\alpha\sum_{z\in G\setminus\{g\}}
          \hat a_{r,x}^{(\omega)}(g,z)
          &\text{ if }g=h\in B_{\rho(r)}^G\\
        a_{r,x}^{(\omega)}(g,h)&\text{ else.}
      \end{cases}
  \end{equation*}
  Here $X_{\lbrace g\rbrace}(\omega)$
  is the potential part of the diagonal matrix element~$a^{(\omega)}(g,g)$
  of~$A^{(\omega)}$, see~\eqref{laplaceelements}.
  This gives that still for each $g,h\in G$,
  the distribution of $a_{r,x}^{(\omega)}(g,h)$
  equals the distribution of $\hat a_{r,x}^{(\omega)}(g,h)$.
  Thereby, with expectation we get
  \begin{align*}
    \EE\spr{\delta_{\id}}{(z-A_{r,x})^{-1}\delta_{\id}}
      =\EE\spr{\delta_{\id}}{(z-\hat A_{r,x})^{-1}\delta_{\id}}\text.
  \end{align*}
  \par
  In order to control error terms, we approximate in a two step scheme.
  For $M\in\cA$, using Cauchy-Schwarz and the triangle inequality, we have
  \begin{align*}
    D_r&
      \le\sup_{x\in V_r^{(0)}}\bigabs{\EE\spr{\delta_{\id}}{(z-\hat A_{r,x})^{-1}\delta_{\id}}-\EE\spr{\delta_{\id}}{(z-A)^{-1}\delta_{\id}}}
        +\frac{2\epsilon(r)}{\abs{\Im z}}  \\&
      \le\sup_{x\in V_r^{(0)}}\bigabs{\EE\bigl[\ifu{M}\spr{\delta_{\id}}{\bigl((z-\hat A_{r,x})^{-1}-(z-A)^{-1}\bigr)\delta_{\id}}\bigr]}\\&\quad
        +\sup_{x\in V_r^{(0)}}\bigabs{\EE\bigl[\ifu{M^c}\spr{\delta_{\id}}{\bigl((z-\hat A_{r,x})^{-1}-(z-A)^{-1}\bigr)\delta_{\id}}\bigr]}
        +\frac{2\epsilon(r)}{\abs{\Im z}}\\&
      \le\sup_{x\in V_r^{(0)}}\bigabs{\EE\bigl[\ifu{M}\spr{\delta_{\id}}{\bigl((z-\hat A_{r,x})^{-1}-(z-A)^{-1}\bigr)\delta_{\id}}\bigr]}
        +2\frac{\epsilon(r)+\PP(M^c)}{\abs{\Im z}}\text.
  \end{align*}
  In the second step we insert a random vector $\psi\from\Omega\to\ell^2(G)$:
  \begin{align*}
    D_r&
      \le\sup_{x\in V_r^{(0)}}\bigabs{\EE\bigl[\ifu M\spr{\delta_{\id}}
        {\bigr((z-\hat A_{r,x})^{-1}-(z-A)^{-1}\bigl)\psi}\bigr]}   
         +2\frac{\epsilon(r)+\PP(M^c)}{\abs{\Im z}} \\&
      \quad+\sup_{x\in V_r^{(0)}}\bigabs{\EE\bigl[\ifu M\spr{\delta_{\id}}
        {\bigl((z-\hat A_{r,x})^{-1}-(z-A)^{-1}\bigr)(\delta_{\id}-\psi)}\bigr]}\\&
      \le\sup_{x\in V_r^{(0)}}\bigl|\EE\bigl[\ifu M\spr{\delta_{\id}}
        {\bigr((z-\hat A_{r,x})^{-1}-(z-A)^{-1}\bigr)\psi}\bigr]\bigr|\\&\quad
    +2\frac{\epsilon(r)+\PP(M^c)+\EE[\ifu M\norm[2]{\delta_{\id}-\psi}]}{\abs{\Im z}}\text.
  \end{align*}
  Applying the second resolvent identity and again Cauchy Schwarz inequality,
  we arrive at
  \begin{equation}\label{echim}
  \begin{split}
    D_r&
      \le\sup_{x\in V_r^{(0)}}\bigabs{\EE\bigl[\ifu M\spr{\delta_{\id}}{(z-\hat A_{r,x})^{-1}
        (A-\hat A_{r,x})(z-A)^{-1}\psi}\bigr]}\\&\quad
          +2\frac{\epsilon(r)+\PP(M^c)+\EE[\ifu M\norm[2]{\delta_{\id}-\psi}]}{\abs{\Im z}}\\&
      \le\frac1{\abs{\Im z}}\sup_{x\in V_r^{(0)}}\EE\bigl[\ifu M\norm[2]{(A-\hat A_{r,x})(z-A)^{-1}\psi}\bigr]\\&\quad
        +2\frac{\epsilon(r)+\PP(M^c)+\EE[\ifu M\norm[2]{\delta_{\id}-\psi}]}{\abs{\Im z}}\text.
  \end{split}
  \end{equation}

  We now choose $M$ and $\psi$ appropriately.

\begin{Lemma}\label{la:measurability}
  Let $A=(A_\omega)$ be a random operator with domain $D_0$ which is self-adjoint for all $\omega$ and let $\kappa>0$ and $z\in\CC\setminus\RR$.
  Then there exist $n=n_\kappa\in\NN$, a set $M=M_{\kappa}\in\cA$ and a random vector $\psi=\psi_\kappa\from\Omega\to\ell^2(G)$,
  such that $\PP(M)>1-\kappa$ and for $\omega\in M$ we have
  \begin{align*}
    \norm[2]{\delta_{\id}-\psi(\omega)} \leq \kappa \quad\text{and}\quad
    \supp\bigl((z-A^{(\omega)})^{-1}\psi(\omega)\bigr)\subseteq B_n^G.
  \end{align*}
\end{Lemma}
\begin{proof}
   For each $n\in\NN$ we define the set 
    \begin{equation*}
      M_{n,\kappa}
       :=\bigl\{\omega\in\Omega\bigm|\exists f\in\ell^2(G)\colon
          \supp\bigl((z-A^{(\omega)})^{-1}f\bigr)\subseteq B_n^G,
            \norm[2]{\delta_{\id}-f}\le\kappa\bigr\}\text.
    \end{equation*}
    In order to verify the measurability of $M_{n,\kappa}\subseteq\Omega$ we claim that one can rewrite this set in the following way
    \begin{equation}\label{Mmb}
      M_{n,\kappa}
        =\Isect_{m\in\NN}
          \Union_{\substack{f\in\ell_\QQ^2(G)\\ 
                  \norm[2]{f-\delta_{\id}}<\kappa+m^{-1}}}
          \Isect_{g\in G\setminus B_n^G}
          \bigl\{\omega\in\Omega\bigm|
            \abs{\spr{\delta_g}{(z-A^{(\omega)})^{-1}f}}<m^{-1}\bigr\}.
    \end{equation}
    To prove \eqref{Mmb}, we note that the inclusion~``$\subseteq$''
    holds because of the continuity of
    $f\mapsto\spr{\delta_g}{(z-A^{(\omega)})^{-1}f}$ uniformly in~$g$.
    The reverse inclusion~``$\supseteq$'' requires more care.
    For each $\omega$ in the right hand side,
    we find for all $m\in\NN$ an $f_m\in\ell_\QQ^2(G)$
    with $\norm[2]{f_m-\delta_{\id}}<\kappa+m^{-1}$,
    such that $\abs{((z-A^{(\omega)})^{-1}f_m)(g)}<m^{-1}$
    for all $g\in G\setminus B_n^G(\id)$.
    Since
    \begin{equation}\label{dom}
      \sup_{m\in\NN}\norm[2]{f_m}
       \le2+\kappa\text,
    \end{equation}
Therefore for all $g\in G$ we have that $(f_m(g))_{m\in\NN}$ is a bounded sequence and hence contains a convergent subsequence.
Using a diagonal sequence we obtain a subsequence such that $(f_{m_k}(g))_{k\in\NN}$ converges for all $g\in G$. We set $f:=\lim_{k\to\infty}f_{m_k}$ to be the pointwise limit of these functions. Then we obtain by Fatou's Lemma and \eqref{dom} that $f\in\ell^2(G)$ and
\begin{align*}
 \norm[2]{f-\delta_{\id}}\leq \kappa \qquad\text{and}\qquad \supp\bigl((z-A^{(\omega)})^{-1}f\bigr)\subseteq B_n^G,
\end{align*}
which shows $\omega\in M_{n,\kappa}$.
    \par
    Note that for any $n\in\NN$ we have $M_{n,\kappa}\subseteq M_{n+1,\kappa}$.
    For each $\omega\in\Omega$, the compactly supported functions
    $D_0$ form a core for $A^{(\omega)}$.
    Therefore Lemma \ref{la:compt} shows that there exists $\psi\in\ell^2(G)$ with
\[
 \norm[2]{\delta_{\id}-\psi}<\kappa \qquad\text{and}\qquad (z-A^{(\omega)})^{-1}\psi\in D_0.
\]
We rephrase this fact as
    \begin{equation*}
      \Omega=\bigcup\nolimits_{n\in\NN}M_{n,\kappa}\qquad(\kappa>0)\text.
    \end{equation*}
    We note $\PP(M_{n,\kappa})\xto{n\to\infty}1$,
    and we choose $n_\kappa\in\NN$ with $\PP(M_{n_\kappa,\kappa}^c)<\kappa$.
    We hereby found the claimed
    $M:=M_{n_\kappa,\kappa}$ and $n:=n_\kappa$.
    \par
    Finally, note that
    \begin{align*}
      \widetilde{M_{n,\kappa}}\etdef
        \{(\omega,f)\in\Omega\times B_\kappa^{\ell^2(G)}(\delta_{\id})
          \mid\supp\bigl((z-A^{(\omega)})^{-1}f\bigr)\subseteq B_n^G
        \}\\&
        =\{(\omega,f)\in\Omega\times B_\kappa^{\ell^2(G)}(\delta_{\id})
          \mid\forall g\in G\setminus B_n^G\colon
            \spr{(\cc z-A^{(\omega)})^{-1}\delta_g}f=0
        \}\text,
    \end{align*}
    is measurable,
    since the inner product is continuous and the factors are measurable,
    see \cref{esa}.
    Thus, the random vector~$\psi$
    exists by the measurable choice theorem by R.~J.~Aumann,
    see \cite[18.27 Corollary]{AliprantisBorder2006}
\end{proof}

  We continue to isolate the effects of the approximation for $\rho(r)\ge n$,
  using the random vector $\phi(\omega):=(z-A^{(\omega)})^{-1}\psi(\omega)$:
  \begin{align*}
    \EE\bigl[\ifu M\norm[2]{(A-\hat A_{r,x})\phi}\bigr]
      =\EE\Bigl[\ifu M\Bigl(\sum_{g\in G}
        \bigabs{(A-\hat A_{r,x})\phi(g)}^2\Bigr)^{1/2}\Bigr]
      \le T_1(r)+\abs\alpha T_2(r)\text,
  \end{align*}
  where we used subadditivity of the square root to separate the non-diagonal terms in~$T_1(r)$
  and the diagonal terms in~$T_2(r)$.
  \par
  For all $\omega\in M$ the vector $\phi(\omega)$
  is supported in~$B_n^G\subseteq B_{\rho(r)}^G$ and
  \begin{align*}
    \norm[\infty]{\phi(\omega)}&
      \le\norm[2]{\phi(\omega)}
      \le\norm[2]{(z-A^{(\omega)})^{-1}\psi(\omega)}
      \le(1+\kappa)/\abs{\Im z}\text.
  \end{align*}
  We apply this and bound $T_1(r)$ as in the deterministic setting:
  \begin{align*}
    T_1(r)\etdef
     \EE\Bigl[\ifu M\Bigl(\sum_{g\in G\setminus B_{\rho(r)}^G}
        \Bigabs{\sum_{h\in\supp\phi}\bigl(a(g,h)-\hat a_{r,x}(g,h)\bigr)\phi(h)}^2
          \Bigr)^{1/2}\Bigr]\\&
      \le\EE\Bigl[\ifu M\norm[\infty]\phi\setsize{\supp\phi}^{1/2}\Bigl(\sum_{g\in G\setminus B_{\rho(r)}^G}
        \sum_{h\in\supp\phi}\bigabs{a(g,h)-\hat a_{r,x}(g,h)}^2\Bigr)^{1/2}\Bigr]\\&
      \le\frac{(1+\kappa)\setsize{B_n^G}}{\abs{\Im z}}\EE\Bigl[\Bigl(\sum_{h\in B_n^G}\sum_{g\in G\setminus B_{\rho(r)}^G}
        \bigabs{a(g,h)-\hat a_{r,x}(g,h)}^2\Bigr)^{1/2}\Bigr]\\&
      \le\frac{(1+\kappa)\setsize{B_n^G}}{\abs{\Im z}}\biggl(
      \Bigl(\sum_{h\in B_n^G}\sum_{g\in G\setminus B_{\rho(r)}^G}
        \EE\bigl[\abs{a(g,h)}^2\bigr]\Bigr)^{1/2}\\&\qquad\qquad\qquad
      +\Bigl(\sum_{h\in B_n^G}\sum_{g\in G\setminus B_{\rho(r)}^G}
        \EE\bigl[\abs{\hat a_{r,x}(g,h)}^2\bigr]\Bigr)^{1/2}\biggr)\text.
  \end{align*}
  Note that in the last step we use Jensen's inequality.
  Now we proceed as in~\eqref{eq:r/3} using again \cref{remark/3}
  and obtain
  \begin{align*}
    T_1(r)&
      \le2\frac{(1+\kappa)\setsize{B_n^G}}{\abs{\Im z}}\Bigl(\sum_{h\in B_n^G}\EE\Bigl[ \sum_{g\in G\setminus B_{\rho(r)}^G}\bigabs{a(g,h)}^2\Bigr]\Bigr)^{1/2} 
      \xto{r\to\infty}0
  \end{align*}
  by dominated convergence, since $\EE\bigl[\norm[2]{A\delta_g}^2\bigr]<\infty$.
  For $\alpha\ne0$, the diagonal term
  \begin{equation*}
    T_2(r)
     :=\abs\alpha^{-1}\EE\Bigl[\ifu M\Bigl(\sum_{h\in\supp\phi}\bigabs{\bigl(a(h,h)-\hat a_{r,x}(h,h)\bigr)
        \phi(h)}^2\Bigr)^{1/2}\Bigr]
  \end{equation*}
  is dealt with as follows, using Jensen's inequality:
  \begin{align*}
    T_2(r)&
      \le\EE\Bigl[\ifu M\norm[\infty]\phi\Bigl(\sum_{h\in\supp\phi}\Bigabs{
          \sum_{g\in G\setminus B_{\rho(r)}^G}\bigl(\hat a_{r,x}(h,g)-a(h,g)\bigr)
        }^2\Bigr)^{1/2}\Bigr]\\&
      \le\frac{1+\kappa}{\abs{\Im z}}\EE\Bigl[\Bigl(\sum_{h\in B_n^G}\Bigl(
          \sum_{g\in G\setminus B_{\rho(r)}^G}\abs{\hat a_{r,x}(h,g)}
         +\sum_{g\in G\setminus B_{\rho(r)}^G}\abs{a(h,g)}
        \Bigl)^2\Bigr)^{1/2}\Bigr]\\&
      \le\sqrt2\frac{1+\kappa}{\abs{\Im z}}\biggl(
         \sum_{h\in B_n^G}\EE\Bigl[\Bigl(\sum_{g\in G\setminus B_{\rho(r)}^G}\abs{\hat a_{r,x}(h,g)}\Bigr)^2\Bigr]
         \\&\qquad\qquad\qquad
        +\sum_{h\in B_n^G}\EE\Bigl[\Bigl(\sum_{g\in G\setminus B_{\rho(r)}^G}\abs{a(h,g)}\Bigr)^2\Bigr]
      \biggr)^{1/2}\text.
  \end{align*}
  Analogous to \eqref{eq:r/3} we see for $h\in B_n^G$:
  \begin{align*}&\phantom{{}={}}
    \EE\Bigl[\Bigl(\sum_{g\in G\setminus B_{\rho(r)}^G}\abs{\hat a_{r,x}(h,g)}\Bigr)^2\Bigr]\\&
      =\EE\Bigl[\sum_{g,g'\in B_r^G\setminus B_{\rho(r)}^G}\abs{\hat a_{r,x}(h,g)}\abs{\hat a_{r,x}(h,g')}\Bigr]\\&
      =\sum_{g\ne g'\in B_r^G\setminus B_{\rho(r)}^G}\EE\bigl[\abs{\hat a_{r,x}(h,g)}\bigr]\EE\bigl[\abs{\hat a_{r,x}(h,g')}\bigr]
        +\sum_{g\in B_r^G\setminus B_{\rho(r)}^G}\EE\bigl[\abs{\hat a_{r,x}(h,g)}^2\bigr]\\&
      \le\sum_{g\ne g'\in G\setminus B_{\rho(r)}^G}\EE\bigl[\abs{a(h,g)}\bigr]\EE\bigl[\abs{a(h,g')}\bigr]
        +\sum_{g\in G\setminus B_{\rho(r)}^G}\EE\bigl[\abs{a(h,g)}^2\bigr]\\&
      =\EE\Bigl[\Bigl(\sum_{g\in G\setminus B_{\rho(r)}^G}\abs{a(h,g)}\Bigr)^2\Bigr]\text.
  \end{align*}
  A consequence of this is
  \begin{align*}
    T_2(r)&
      \le\sqrt{8}\frac{1+\kappa}{\abs{\Im z}}\Bigl(\sum_{h\in B_n^G}\EE\Bigl[\Bigl(
          \sum_{g\in G\setminus B_{\rho(r)}^G}\bigabs{a(h,g)}
        \Bigl)^2\Bigr]\Bigr)^{1/2}
      \xto{r\to\infty}0\text,
  \end{align*}
  again by Lebesgue, this time with $\EE\bigl[\norm[1]{A\delta_h}^2\bigr]<\infty$.
  Now conclude from~\eqref{echim}
  \begin{align*}
    \limsup_{r\to\infty}D_r&
      \le2\limsup_{r\to\infty}\frac{\epsilon(r)+\PP(M^c)
        +\EE[\ifu M\norm[2]{\delta_{\id}-\psi}]}{\abs{\Im z}}
      \xto{r\to\infty}
      \frac{4\kappa}{\abs{\Im z}}
    \text.
  \end{align*}
  Since $\kappa>0$ was arbitrary, we conclude $D_r\xto{r\to\infty}0$.
\end{proof}

\subsection{Almost sure convergence}\label{almostsure}

The following concentration inequality is taken from
\cite[Theorem~3.1]{McDiarmid1998}.
\begin{Theorem}[{\cite[Theorem~3.1]{McDiarmid1998}}]\label{thm:McDia}
  Let $X=(X_1,\dotsc,X_n)$ be a family of independent random
  variables with values in $\RR$,
  and let $f\colon\RR^n\to\RR$ be a function, such that whenever $x\in\RR^n$ and $x'\in \RR^n$ differ only in one coordinate we have
  \begin{equation*}
    \abs{f(x)-f(x')}\le c.
  \end{equation*}
  Then, for $\mu:=\EE[f(X)]$ and any $\epsilon\ge0$,
  \begin{equation*}
    \PP(\abs{f(X)-\mu}\ge\epsilon)\le2 \exp\Big(-\frac{2\epsilon^2}{nc^2}\Big)\text.
  \end{equation*}
\end{Theorem}

We use \cref{thm:McDia} to upgrade the convergence in \cref{thm:expected}.
We obtain almost sure convergence as well as a Pastur-Shubin-trace formula.
Here we need an additional assumption on~$\rho$,
namely that~$\rho$ does not grow too fast.
\begin{Theorem}\label{thm:surprise}
  Let $N_r$ and $\bar N$ be as in~\eqref{Nomega} and~\eqref{Nbar}
  and $\rho(r):=\frac{\ln r}{4\ln\setsize S}-1$.
  Then there is a set $\tilde\Omega\in\cA$ with full probability
  $\PP(\tilde\Omega)=1$ and
  \begin{equation*}
    \lim_{r\to\infty}N_r^{(\omega)}(\lambda)
      =\bar N(\lambda)
  \end{equation*}
  for all $\omega\in\tilde\Omega$
  and all continuity points~$\lambda$ of~$\bar N$.
\end{Theorem}
\begin{proof}
  Denote the set of continuity points of $\bar N$ by $\cC$
  and let $\lambda\in\cC$ and $\epsilon>0$.
  For $r$ large enough, we have by \cref{thm:expected}
  \begin{align}\label{triangle}
    \begin{split}&\phantom{{}={}}
      \PP\bigl(\abs{N_r(\lambda)-\bar N(\lambda)}\ge\epsilon\bigr)\\&
        \le\PP\bigl(\abs{N_r(\lambda)-\bar N_r(\lambda)}
                \ge\epsilon-\abs{\bar N_r(\lambda)-\bar N(\lambda)}\bigr)\\&
        \le\PP\bigl(\abs{N_r(\lambda)-\bar N_r(\lambda)}\ge\epsilon/2\bigr)
        \text.
    \end{split}
  \end{align}
  We want to apply \cref{thm:McDia}.
  Each random variable on a bond of the graph~$\Gamma_r$
  in a sofic approximation has bounded effect on the eigenvalue counting function, namely
  \begin{equation*}
    \abs{N_r^{(\omega)}(\lambda)-N_r^{(\omega')}(\lambda)}
      \le c:=2/\setsize{V_r}
  \end{equation*}
  for all $\omega,\omega'\in\Omega$ which differ only on a single bond.
  This is clear as the associated operators differ only by a rank~$2$ perturbation.
  \par
  By construction, the number~$n$ of random variables is bounded by
  \begin{align*}
    n&
      \le\setsize{V_r}\setsize S^{2(\rho(r)+1)}
      =\setsize{V_r}\setsize S^{\frac{\ln r}{2\ln\setsize S}}
      =\setsize{V_r}\sqrt r\text.
  \end{align*}
  The relevant quantity in \cref{thm:McDia} is
  \begin{equation*}
    nc^2
      \le\setsize{V_r}\sqrt r\cdot4/\setsize{V_r}^2
      =4\sqrt r/\setsize{V_r}
      \le4/\sqrt r\text.
  \end{equation*}
  Use $\setsize{V_r}\ge r$ for the last step.
  \Cref{thm:McDia} and~\eqref{triangle} combined give
  \begin{align}\label{enterHere}
    \sum_{r\in\NN}\PP(\abs{N_r(\lambda)-\bar N(\lambda)}\ge\epsilon)&
      \le\sum_{r\in\NN}2\e^{-\epsilon^2\sqrt r/8}
      <\infty\text.
  \end{align}
  This is by definition almost complete convergence   of~$N_r(\lambda)$ to~$\bar N(\lambda)$ and implies almost sure convergence, i.e.,
  the existence of~$\Omega_\lambda\in\cA$ with $\PP(\Omega_\lambda)=1$ and
  \begin{equation*}
    N_r^{(\omega)}(\lambda)\xto{r\to\infty}\bar N(\lambda)
    \qquad(\omega\in\Omega_\lambda)\text.
  \end{equation*}
  The monotone and bounded function~$\bar N$
  has only countably many discontinuities.
  We choose $M\subseteq\cC$ countable and dense.
  Then, the set $\tilde\Omega:=\Isect_{\lambda\in M}\Omega_\lambda$ has probability~$1$, too.
  Now fix $\omega\in\tilde\Omega$.
  We know for all $\lambda\in\RR$
  \begin{align*}
    \limsup_{r\to\infty}N_r^{(\omega)}(\lambda)&
      \le\inf_{\lambda'\in M\isect[\lambda,\infty)}\lim_{r\to\infty}N_r^{(\omega)}(\lambda')
      =\inf_{\lambda'\in M\isect[\lambda,\infty)}\bar N(\lambda')
      =\bar N(\lambda)\text,
  \end{align*}
  since~$\bar N$ is monotone and continuous from the right, and~$M$ is dense.
  In the other direction, for all $\lambda\in\cC$ we have
  \begin{align*}
    \liminf_{r\to\infty}N_r^{(\omega)}(\lambda)&
      \ge\sup_{\lambda'\in M\isect(-\infty,\lambda]}\lim_{r\to\infty}N_r^{(\omega)}(\lambda')
      =\sup_{\lambda'\in M\isect(-\infty,\lambda]}\bar N(\lambda')
      =\bar N(\lambda)\text.
  \end{align*}
  Hereby, $\lim_{r\to\infty}N_r^{(\omega)}(\lambda)$
  exists and equals $\bar N(\lambda)$
  for all $\omega\in\tilde\Omega$ and $\lambda\in\cC$.
\end{proof}

\begin{Remark}
  In many cases, we can allow~$\rho$ to grow much faster.
  Let~$n_r$ denote the number of non-trivial random varibles in
  $\{a(x,y)\mid x,y\in B_r^G\}$.
  The condition on~$\rho$ needed in \cref{thm:surprise},
  namely in~\eqref{enterHere}, is
  \begin{equation*}
    \sum_{r\in\NN}\exp\Bigl(-\frac{\epsilon^2}{2n_r}\setsize{V_r}^2\Bigr)
      <\infty
  \end{equation*}
  for all $\epsilon>0$.
  Assume,~$A$ has finite hopping range, i.e., there exists~$R\in\NN$
  such that $a(x,y)=0$ whenever $d(x,y)\ge R$.
  This is, e.g., the case in the well known Anderson model.
  Then $n_r\le\setsize{V_r}\binom{\setsize S^R}2$, and consequently
  \begin{equation*}
    \sum_{r\in\NN}\exp\Bigl(-\frac{\epsilon^2}{2n_r}\setsize{V_r}^2\Bigr)
      \le\sum_{r\in\NN}\exp\Bigl(-\frac{\epsilon^2}{2\binom{\setsize S^R}2}\setsize{V_r}\Bigr)
      <\infty
  \end{equation*}
  for all $\epsilon>0$, since $\setsize{V_r}\ge r$.
  We see that for operators with finite hopping range,
  $\rho(r):=\floor{r/3}$ suffices, as in the deterministic setting.
\end{Remark}

Note that for \cref{thm:surprise} the matrix elements of~$A$ and~$A_r$ actually have to be random variables defined on the same probability space $(\Omega,\cA,\PP)$.
This makes the operator~$A$ \emph{not ergodic}
with respect to the usual definition of ergodicity of operators,
see e.g.\ \cite{PasturFigotin1992}.
The problem is that we can not define an appropriate group action on this probability space.
However if one considers $A$ on its canonical probability spaces, which is embedded in our space $(\Omega,\cA,\PP)$, then its ergodicity is easy to check. Of course this restriction does not change the properties of $A$, as applying this restriction we only drop information which has no influence on $A$.


\section{Examples and applications}\label{exa}

As an application we show in \cref{percolation}
that the well known percolation model on sofic groups
is covered by our abstract theory.

The class of sofic groups is quite large.
In fact, according to \cite{Weiss2000},
there is no known example for a (finitely generated)
group which fails to be sofic.
As already proved in \cite{Weiss2000},
amenable as well as residually finite groups are surjunctive,
and surjunctive groups are a subclass of sofic groups.
\Cref{amenable,residuallyfinite} are devoted to show these inclusions directly,
including the free group as an example for a non-amenable but residually finite group.

\subsection{Percolation}\label{percolation}
In this section we apply the above results in the random setting to percolation models on graphs.
We will study the approximability of the IDS of the corresponding Laplacian.
The models in consideration will contain short range as well as long range percolation on sofic groups.

As before let~$G$ be a finitely generated sofic group and~$S$ a finite, symmetric set of generators.
Let $\Gamma_{\rm{co}}=(V,E_{\rm{co}})$ be the complete graph over the vertex set $V=G$, i.e.\ the edge set is
\[
  E_{\rm{co}}=\cP_{1,2}=\{e \subseteq G \mid |e|= 2\}.
\]
Furthermore let $p\in \ell^1(G)$ be such that
\[
  0\leq p(x)=p(x^{-1})\leq 1 \quad \quad (x\in G)
\]
and define for distinct $x,y\in G$ the random variables
\begin{equation}
  X_{\{x,y\}}^{(\omega)}=\begin{cases}
                 1 & \text{with probability }p(xy^{-1})\\
		 0 & \text{else.}
                 \end{cases}
\end{equation}
We assume that all these random variables are independent.
Using these random variables we define for each~$\omega$
a random subgraph $\Gamma_\omega=(V,E_\omega)$ of~$\Gamma_{\rm{co}}$, with
\[
  E_\omega=\{e\in E_{\rm{co}}\mid X_e(\omega) =1\}.
\]
Such a graph $\Gamma_\omega$ may contain edges between two arbitrary vertices.
The assumption $p\in\ell^1(G)$ implies that this subgraph is almost surely locally finite, see \cite[Lemma~3.2]{ASV2012}. However, if $p$ is not finitely supported there exists almost surely no upper bound for the vertex degree. In this situation there also exist with probability one edges of arbitrary length, measured in the word metric induced by the generating system $S$. This implies that the Laplacian which we are going to define now is almost surely unbounded.

A special case of this model is the (typical) percolation of the Cayley graph $\Gamma=\Gamma(G,S)$. Here one sets all $p(x)=0$ for all $x\notin S$. Then, obviously $p$ is finitely supported and we have a random graph $\Gamma_\omega$ where the only edges which may exist are the edges of $\Gamma$.
   
The matrix elements of the operator in consideration are given by
\begin{equation}\label{percmatrix}
  a^{(\omega)}(x,y)=\begin{cases}
                     X_{\{x,y\}}& \text{ if }x\ne y\\
                     -\sum_{z\neq x} X_{\{x,z\}}(\omega) & \text{ else.}
                   \end{cases}
\end{equation}
This immediately gives that conditions \eqref{sym} and \eqref{train} are fulfilled.
In order to show essential self-adjointness it remains to verify \eqref{summable},
c.f.\ \cref{esa}.  Therefore we define for each $\omega$ the set $N(\omega):=\{x\in G\mid X_{\{\id,x\}}(\omega)=1\}$ and calculate for all $\omega$ where $\Gamma_\omega$ is locally finite
\begin{align*}
 \Bigl(\sum_{x\in G}\abs{X_{\{\id,x\}}{(\omega)}}\Bigr)^2 
= \Bigl(\sum_{x\in N(\omega)}\abs{X_{\{\id,x\}}{(\omega)}}\Bigr)^2
\leq \abs{N(\omega)} \sum_{x\in N(\omega)}\abs{X_{\{\id,x\}}{(\omega)}}
\end{align*}
By monotone convergence we get
\begin{align*}
  \EE\biggl(\Bigl(\sum_{x\in G}\abs{X_{\{\id,x\}}}\Bigr)^2 \biggr)
\leq  \sum_{x\in G} \EE\left( \abs{N}\abs{X_{\{\id,x\}}}\right).
\end{align*}
Note that $N$ and $X_{\{\id,x\}}$ are not independent. Therefore we define for each $x\in G$ a random variable by $N_x(\omega):=\abs{\{y\in G\setminus\{x\}\mid X_{\{\id,y\}}(\omega)=1\}}$ which is independent of $X_{\{\id,x\}}$. For each $\omega$ we have $N_x(\omega) \leq \abs{N(\omega)} \leq N_x(\omega) +1$. This implies $\EE(N_x)\leq \norm[1]{p}$. We apply these observations to obtain
\begin{align*}
  \EE\biggl(\Bigl(\sum_{x\in G}\abs{X_{\{\id,x\}}}\Bigr)^2 \biggr)
\leq  \sum_{x\in G} \EE\left( \abs{N_x}+1\right) \EE\left(\abs{X_{\{\id,x\}}}\right)
\leq (\norm[1]p+1)\norm[1]p<\infty.
\end{align*}
This finally proves \eqref{summable}.
Therefore there exists for almost all~$\omega$
a unique self-adjoint operator $\Delta_\omega\from D_\omega\to \ell^2(G)$
with matrix elements given by~\eqref{percmatrix}.
This operator is called the Laplacian of $\Gamma_\omega$.
As the conditions given in \eqref{sym}, \eqref{train} and \eqref{summable}
hold and also condition \eqref{id_distr} is satisfied, the operator $(\Delta_\omega)_\omega$ is a random Hamiltonian, with the definition at the end of \cref{countable}. Therefore the theory developed in \cref{secrandom} is valid for this operator.
In particular the IDS exists for almost all realizations~$\omega$
and does not depend on~$\omega$.

Note that here we show weak convergence of distribution functions for almost all $\omega$. In more restricted settings one can obtain even more. For instance in \cite{Schwarzenberger-12} and \cite{ASV2012} the authors consider long-range percolation models over amenable groups and obtain uniform convergence.
However their methods rely massively on the existence of sets with an arbitrary small boundary, which is per definition not the case for non-amenable groups.

\subsection{Amenable Groups}\label{amenable}

The group~$G$ is \emph{amenable}, if it admits a \emph{F\o lner sequence},
i.e.\ an increasing sequence
$F_1\subseteq F_2\subseteq\dotsb\subseteq G$ of finite subsets of~$G$
such that $\Union_{j\in\NN}F_j=G$ and for every finite $K\subseteq G$
\begin{equation}\label{eq:amenable}
  \lim_{j\to\infty}\frac{\setsize{KF_j\Delta F_j}}{\setsize{F_j}}=0\text.
\end{equation}
Here $\Delta$ denotes the symmetric difference,
and the quotient is the size of the $K$-boundary of~$F_j$ relative to~$F_j$ itself.

Examples of amenable groups include all finitely generated Abelian groups,
since for all $d\in\NN$ the balls $B_j^{\ZZ}(0)\subseteq\ZZ^d$ of radius $j\in\NN$
with respect to any complete metric on~$\ZZ^d$ form such a sequence.
Furthermore all groups of sub-exponential growth are amenable. Therefore the famous examples $\ZZ^d$, Heisenberg group, Grigorchuk group fit in our setting. Also the Lamplighter group is amenable. 

Amenable (finitely generated) groups are easily seen to be sofic. 
Actually, let $\epsilon>0$ and $r\in\NN$ be given.
Now set $K:=B_r^G(\id)$ and choose $j\in\NN$ by~\eqref{eq:amenable}
such that $\setsize{KF_j\Delta F_j}\le\epsilon\setsize{F_j}$.
Then, define $(V_{r,\epsilon},E_{r,\epsilon})$
as the restriction $\Gamma(G,S)|_{F_j}$ of the Cayley graph of~$G$ to~$F_j$
and $V_{r,\epsilon}^{(0)}:=\Isect_{k\in K}kF_j\subseteq V_{r,\epsilon}=F_j$.

Now, by construction, we have $B_r^{V_{r,\epsilon}}(v)\subseteq V_{r,\epsilon}$
for all $v\in V_{r,\epsilon}^{(0)}$.
The required graph isomorphism in \cref{S1} is the translation
$x\mapsto xv^{-1}$.
\Cref{S2} is fulfilled, too, since
\begin{equation*}
  \setsize{V_{r,\epsilon}^{(0)}}
    =\setsize{F_j\setminus KF_j}
    \ge\setsize{F_j}-\setsize{KF_j\Delta F_j}
    \ge(1-\epsilon)\setsize{V_{r,\epsilon}}
    \text.
\end{equation*}
Studying spectral properties on discrete operators, the geometric setting of amenable groups can be seen as the natural generalization of $\ZZ^d$. This is the case as many proves rely on the property that boxes or balls in $\ZZ^d$ have a vanishing boundary, if one increases the radius. This property remains true for the above defined F\o{}lner sequences. 

For this reason convergence results for the IDS on amenable groups have intensively been studied in the literature. As mentioned before, it is for amenable groups often possible to prove even uniform convergence, see \cite{LenzV-09,LenzSV-10,PogorzelskiS-12}.%
\nocite{LenzSV-10-e}

\subsection{Residually Finite Groups}\label{residuallyfinite}
Let the group~$G$ be \emph{residually finite},
i.e.\ there exists a sequence $(G_n)_{n\in\NN}$
of normal subgroups of~$G$ such that
\begin{enumerate}[({F}1)]
  \item $[G:G_n]<\infty$ for all $n\in\NN$,
  \item $\Isect_n G_n =\{\id\}$.
\end{enumerate}

Without loss of generality we can assume $G_{n+1}\subseteq G_n$
for all $n\in\NN$, since $\tilde G_n:=\Isect_{j\le n}G_j$
is normal and of finite index
$[G:\tilde G_n]\le\prod_{j=1}^n[G:G_j]<\infty$ (count cosets).

The factor groups
\[
  H_n:={G}/{G_n} = \{gG_n\mid g\in G\}\text,
\]
$n\in\NN$, consist of the equivalence classes,
which are induced by the subgroup~$G_n$.
The group structure is inherited from~$G$.

We additionally assume $G$ to be finitely generated, namely $G=\langle S\rangle$
with a finite and symmetric set of generators $S\subseteq G$.
Then $H_n$ is generated by the set $S_n:=\{sG_n\mid s\in S\}$.

As before, the Cayley graph of~$G$ with the generators~$S$
will be denoted by~$\Gamma=\Gamma(G,S)$,
and the Cayley-Graphs of the finite groups $H_n$ with generators $S_n$
will be denoted by~$\Gamma_n=\Gamma_n(H_n,S_n)$.
The induced word metrics are $d:=d^\Gamma\from G\times G\to\NN_0$
and $d_n:=d^{\Gamma_n}\from H_n\times H_n\to\NN_0$, respectively.
We write
\[
  B_r=\{x\in G\mid d(x,\id)\le r\}
  \qtextq{and}
  B_r^{(n)}=\{x\in H_n\mid d_n(x,\id)\le r\}
\]
for the balls of radius $r\in\NN_0$.

The following Lemma shows that for increasing~$n$
the Cayley graphs of the factor groups equal the Cayley graphs of the group~$G$
on larger and larger scales.
\begin{Lemma}\label{lemma:ball}
  Let $\Gamma$ and $\Gamma_n$, $n\in\NN$ be given as above.
  Then for all $r\in \NN$ there is $n(r)\in\NN$ such that
  \begin{equation}\label{eq:ass}
    \Gamma|_{B_{r}}\simeq\Gamma_n|_{B_r^{(n)}}\qquad(n\ge n(r))\text.
  \end{equation}
  Here $\simeq$ means that the induced subgraphs are isomorphic.
  Furthermore, every residually finite group is sofic.
\end{Lemma}
\begin{proof}
  For given $r\in \NN$ choose $n(r)$ such that $B_{2r} \cap G_n = \{\id \}$ for all $n\ge n(r)$, which is possible as $G$ is residually finite.
  Then for all $n\ge n(r)$, $gG_n\isect B_r$
  contains at most one element, since for
  $h,h'\in gG_n\isect B_r$ we have
  $h^{-1}h'\in G_n \isect B_{2r}=\{\id\}$, i.e.~$h=h'$.

  The $r$-ball in $H_n$ around the identity element is
  \[
    B_r^{(n)} = \{gG_n\mid g\in B_r\}\text,
  \]
  and, still for $n\ge n(r)$, the mapping
  \[
    \psi_{n,r}:B_r^{(n)}\to B_r\quad\text{with}\quad
    \{\psi_{n,r}(gG_n)\}= gG_n\isect B_r
  \]
  is well-defined.
  Since for $g\in B_r$ we have $\psi_{n,r}(gG_n)=g$,
  $\psi_{n,r}$ is bijective.
  \par
  Because of $S_n=\{sG_n\mid s\in S\}$,
  $\psi_{n,r}$ respects the labels of the induced graphs on its domain
  and is a therefore a directed graph isomorphism.
  In particular, we can choose $V^{(0)}:=V$ to show that $G$ is sofic.
\end{proof}

\subsubsection{The Free Group}\label{freegrp}

The free group $F_s=\langle S_s\rangle$
with $s\in\NN$~generators $S_s:=\{a_1,\dotsc,a_s,a_1^{-1},\dotsc,a_s^{-1}\}$
is an example of a non-amenable residually finite group, see \cref{figB3}.
\begin{figure}
  \centering
  \input{B3F2.tex}
  \caption{The Cayley graph of the ball~$B_3$ of radius~$3$
    in~$F_2$, with $a:=a_1$, $b:=a_2$, $A:=a^{-1}$ and $B:=b^{-1}$.
    The arrows indicate the corresponding generator:
    $x\rightarrow y$ means $ax=y$,
    $x\twoheadrightarrow y$ is synonymous for $bx=y$.
    }
  \label{figB3}
\end{figure}
For $F_s$ we provide an explicit construction of the sequence of normal
subgroups and their factor groups, along which the eigenvalue counting
functions converge.  The idea goes back to \cite{Biggs88}.

Note that, for the free group, the property~\eqref{eq:ass}
is satisfied if $\gamma_n\to\infty$,
where~$\gamma_n$ denotes the girth of the graph~$\Gamma_n$,
i.e.\ the length of the shortest circle in the Cayley graph~$\Gamma_n$.
Hence, the goal in this construction is to find a sequence of finite groups
with increasing girth.

Let $B_n$ be the ball of radius~$n\in\NN$ in $F_s$
centered at the identity, which is the empty word~$\epsilon$,
with respect to the metric on the Cayley graph~$\Gamma(F_s,S_s)$.
A short calculation reveals
\begin{equation}\label{Bnsize}
  \setsize{B_n}
    =\frac{s(2s-1)^n-1}{s-1}\text.
\end{equation}
For each generator $x\in S_s$
we define the permutation~$p_x^{(n)}$ on~$B_n$ by
\begin{equation*}
  p_x^{(n)}(w):=\begin{cases}
    xw  &(xw\in B_n)\\
    w_1^{-1}w_2^{-1}\dotsc w_m^{-1}
      &(xw\notin B_n)
  \end{cases}
\end{equation*}
for reduced words $w=w_1w_2\dotsc w_m\in B_n$
with $w_j\in S_s$, $j\in\{1,\dotsc,m\}$, see \cref{figperm1}.
\begin{figure}
  \centering
  \input{perm1.tex}
  \caption{The action of the permutation $p_a^{(1)}$ on $B_1$
    and of $p_a^{(2)}$ on $B_2$.}
  \label{figperm1}
\end{figure}
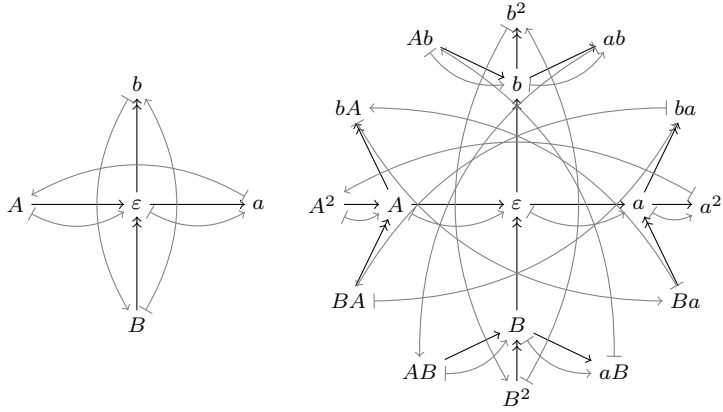
The permutation~$p_x^{(n)}$ maps elements of~$B_{n-1}$
to their neighbors in direction~$x$
and members of the sphere~$B_n\setminus B_{n-1}$
are sent into $B_n\setminus B_{r-2}$.

The group $H_n:=\langle\{p_x^{(n)}\mid x\in S_s\}\rangle$
generated by this permutations is a subgroup of the symmetric group~%
$\cS_{B_n}$ on~$B_n$.
See \cref{figH1} for an illustration of~$H_1$.
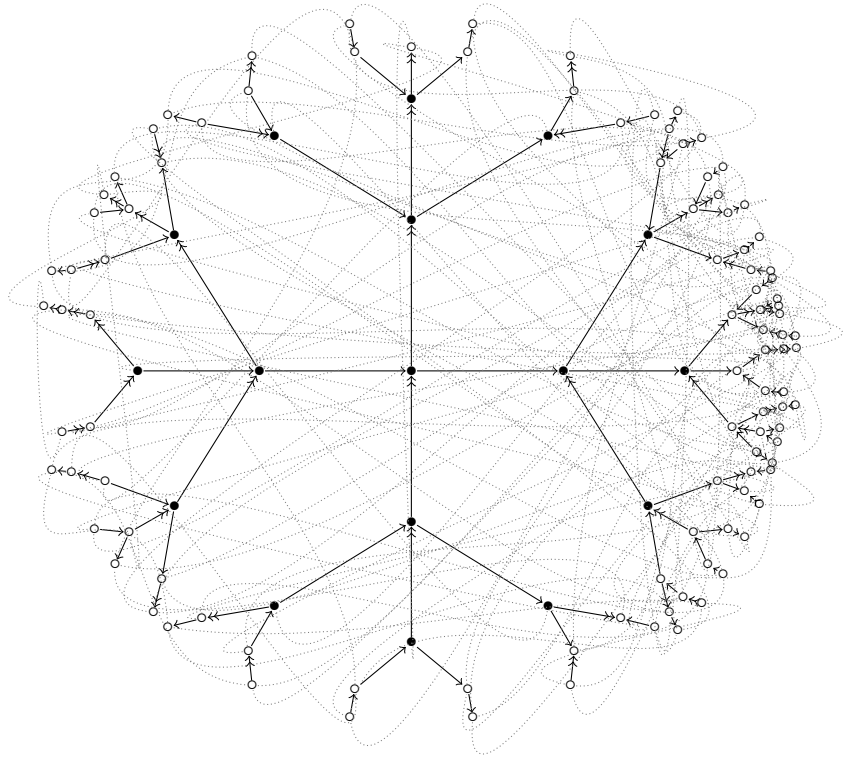
\begin{figure}
  \centering
  \input{H1}
  \caption{The Cayley graph of
    $H_1=\bigl\langle\{p_a^{(1)},p_b^{(1)}\}\bigr\rangle=\cS_{B_1}$.
    The ball of radius~$2$ around the center is marked with~$\bullet$.
    The nodes were found by a depth first search.
    Whenever an already found node was encountered again,
    a thin dotted line was added to indicate
    the neighbourship in the Cayley graph of~$H_1$.
    This drawing does not indicate the high symmetry this graph has.
    In fact, each node is center of an $F_2$-ball of radius~$2$.}
  \label{figH1}
\end{figure}

The map $S_s\ni x\mapsto p_x^{(n)}\in\cS_{B_n}$
has an extension to a group homomorphism
$\varrho_n\colon F_s\to H_n\subseteq\cS_{B_n}$ via
$\varrho_n(w_1\dots w_m):=p_{w_1}^{(n)}\circ\dotsb\circ p_{w_m}^{(n)}$.
We consider the normal subgroups $G_n:=\ker\varrho_n$ of $F_s$.

Observe that a non-empty reduced word $w=w_1\dots w_m\in G_n$
is either the empty word $w=\epsilon$ or has at least length~$m\ge2n+1$,
since the orbit of the empty word~$\epsilon\in B_n$ passes each sphere:
$\varrho_n(w_1\dots w_j)(\epsilon)\in B_j\setminus B_{j-1}$ for $j\le n$
and $\varrho_n(w_1\dots w_j)(\epsilon)\in B_n\setminus B_{2n-j}$ for $j>n$.
Therefore the girth of~$H_n$ is at least~$2n+1$ and $\Isect_n G_n=\{1\}$.

Note that this argument only guarantees girth~$3$ for $H_1$ shown in \cref{figH1},
i.e., a ball of radius~$1$ isomorphic to a ball in~$F_2$.
In fact, $H_1$ has girth~$6$ and an $F_2$-ball of radius~$2$
embedded around each element.

Whenever it seems advantageous,
one can replace the permutation~$p_x^{(n)}$ by
\begin{equation*}\label{perm2}
  \tilde p_x^{(n)}(w):=\begin{cases}
    xw  &(xw\in B_n)\\
    w   &(xw\notin B_n\text{ and }w_1\ne x)\\
    w_1^{-1}w_2^{-1}\dotsc w_m^{-1}
      &(xw\notin B_n\text{ and }w_1=x)\text,
  \end{cases}
\end{equation*}
as $\tilde p_x^{(n)}$ behaves very similar to $p_x^{(n)}$, see \cref{figperm2}.
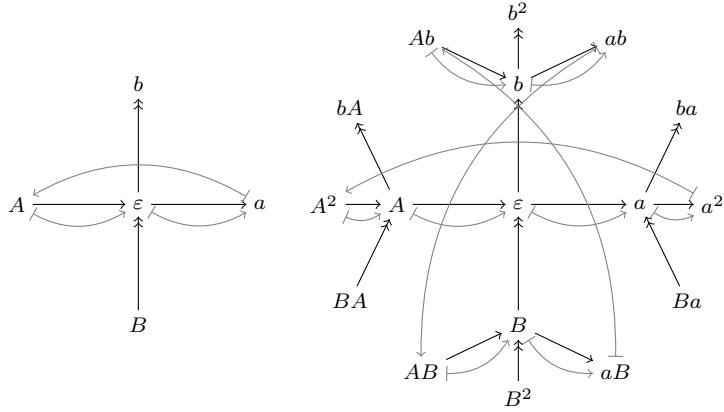
\begin{figure}
  \centering
  \input{perm2.tex}
  \caption{The action of the permutation $\tilde p_a^{(1)}$ on $B_1$
    and of $\tilde p_a^{(2)}$ on $B_2$.
    Points without $\mapsto$, e.g.~$b^2$, are fixed points of the permutation
    for the generator~$a$, but of course not of the corresponding permutation
    for the generator~$b$.}
  \label{figperm2}
\end{figure}
That is the case in \cref{Halternates}.

The Cayley graph of $(F_s,S_s)$ is a regular tree.
For such graphs the spectral distribution function for the adjacency operator
is explicitly calculated in \cite{McKay81} as
\begin{equation}\label{eq:McKay}
  x\mapsto
    \frac{s\sqrt{4(2s-1)-x^2}}{\pi(4s^2-x^2)}
    \chi_{[0,2\sqrt{2s-1}]}(\abs x)\text.
\end{equation}
This shows in particular that the spectral distribution function,
which equals according to \cref{thm:weak1}
the integrated density of states, is continuous.
We conclude that the limit in \cref{thm:weak1}
exists for all $\lambda\in\RR$ and is actually uniform in~$\lambda$.
See \cref{figConv} for numerical visualisation.
Note that the graphs behind \cref{fig60,fig120,fig5040} are Ramanujan,
i.e.\ the spectrum of the corresponding Laplacian is contained in
$[-2\sqrt{2s-1},2\sqrt{2s-1}]$,
while the ones in \cref{fig720,fig2520} are not.

A criterion for the quality of the approximation should measure
the growth of the radius~$r$ in \cref{S1}
relative to the growth~$\setsize{V_r}$ of the approximating graphs
as well as the proportion of~$\setsize{V_r^{(0)}}$ in \cref{S2}.
In the case of the free group, we have $V_r^{(0)}=V_r$,
so the growth of the girth~$\gamma_n$
compared with the growth of the groups~$H_n$
appears to be a reasonable choice.
The best possible situation is
\genericConstant{girthconst}%
\begin{equation}\label{girthgrowth}
  \gamma_n\ge\girthconst\ln{\setsize{H_n}}
\end{equation}
for some $\girthconst>0$.
This is so, because an $s$-regular graph with girth~$\gamma$
contains at least $\setsize{B_{\floor{\gamma/2}}}$ vertices, so
\begin{equation*}
  \ln\setsize{H_n}\ge\ln\setsize{B_{\floor{\gamma_n/2}}}
    \approx\floor{\gamma_n/2}\ln(2s-1)\text.
\end{equation*}
Equations like \eqref{girthgrowth} have far reaching implications,
see e.g.\ \cite{Brooks-Lindenstrauss-2009}.
Unfortunately, the lower bound~$\gamma_n\ge2n+1$
on the girth of~$H_n$ seems hard to improve.

The trivial upper bound to $\setsize{H_n}$
is the number of permutations on~$B_n$:
\begin{equation*}
  \setsize{H_n}
    \le\setsize{B_n}!
    =\Bigl(\frac{s(2s-1)^n-1}{s-1}\Bigr)!\enspace\text.
\end{equation*}
The following lemma improves this bound.
\begin{Lemma}\label{Halternates}
  In the construction of $H_n$,
  use $p_x^{(n)}$ for odd~$n$ and $\tilde p_x^{(n)}$ for even~$n$.
  Then
  \begin{equation*}
    \setsize{H_n}
      \le\frac{\setsize{B_n}!}2\text.
  \end{equation*}
\end{Lemma}
\begin{Remark}
  In the case~$s=2$, we verified $\setsize{H_1}=60=\setsize{B_1}!/2$.
  (Here, we used $H_1=\bigl\langle\tilde p_a^{(n)},\tilde p_b^{(n)}\bigr\rangle$.
  \Cref{figH1} shows $\bigl\langle p_a^{(n)},p_b^{(n)}\bigr\rangle$,
  which has~$120$ elements.)
\end{Remark}
\begin{proof}
  We study the sign of the permutations $p_x^{(n)}$, $x\in S$, via the formula
  \begin{equation*}
    \sgn(p):=(-1)^{\elc(p)}\text,
  \end{equation*}
  where~$\elc(p)$ is the number of even-length cycles of the permutation~$p$.
  For each generator $x\in S$ and its permutation~$p_x^{(n)}$,
  the cycle containing the empty word~$\epsilon$ has length~$2n+1$, since
  \begin{equation}\label{oddcycle}
    \bigl(p_x^{(n)}\bigr)^{2n+1}(\epsilon)
      =\bigl(p_x^{(n)}\bigr)^{n+1}(x^n)
      =\bigl(p_x^{(n)}\bigr)^n(x^{-n})
      =\epsilon\text.
  \end{equation}
  The length of a cycle containing the reduced word
  $w=w_1\dots w_m\in B_n$ with $w_1\notin\{x^{-1},x\}$
  is $4(n-m)+2$, since, with $d:=n-m$,
  \begin{equation}\label{evencycle}
    \bigl(p_x^{(n)}\bigr)^j(w)=
    \begin{cases}
      x^jw
        &(j\in\{0,\dotsc,d\})\\
      x^{j-2d-1}w_1^{-1}\dots w_m^{-1}
        &(j\in\{d+1,\dotsc,3d+1\})\\
      x^{j-4d-2}w
        &(j\in\{3d+2,\dotsc,4d+2\})
    \end{cases}
  \end{equation}
  Note, that there are exactly two words not beginning in $\{x^{-1},x\}$
  in every cycle not containing~$\epsilon$.
  If one of these words is $w_1\dots w_m$,
  then the other one is $w_1^{-1}\dots w_m^{-1}$.
  Hence, the number of even-length cycles equals
  half the number of reduced words not beginning in $\{x^{-1},x\}$:
  \begin{align*}
    \elc(p_x^{(n)})&
      =\frac12\sum_{\ell=0}^{n-1}(2s-2)(2s-1)^\ell
      =\frac{(2s-1)^n-1}2\\&
      =\frac{(2(s-1)+1)^n-1}2
      =\frac12\sum_{k=1}^n\binom nk(2(s-1))^k
      \text.
  \end{align*}
  For all $k\ge2$ the summand $\frac12\binom nk(2(s-1))^k$ is even, hence
  \begin{equation*}
    \sgn(p_x^{(n)})
      =(-1)^{\frac12\sum_{k=1}^n\binom nk(2(s-1))^k}
      =(-1)^{n(s-1)}\text.
  \end{equation*}
  \par
  A similar analysis shows
  $\elc(\tilde p_x^{(n)})=\elc(p_x^{(n-1)})$ for $n\ge1$,
  and consequently:
  \begin{equation*}
    \sgn(\tilde p_x^{(n)})=(-1)^{(n-1)(s-1)}\text.
  \end{equation*}
  Therefore, $H_n$
  can be implemented as a subgroup of the alternating group on~$B_n$,
  and we get the claimed bound.
\end{proof}

The size of the groups explodes very rapidly.
Actually, the group corresponding to~$B_2\subseteq F_2$
is bounded in size only by $17!/2\approx1.778\cdot10^{14}$.
In the following we construct intermediate sizes.
The strategy is the same, we use permutations,
but this time not on balls but intermediate connected sets $\Lambda\subseteq F_s$
containing the empty word.
Let $w=w_1w_2\dotsc w_m\in \Lambda$ be a word of length~$m$.
For a generator $x\in S_s$,
we count how often we can append~$x^{-1}$ to~$w$
without leaving the set~$\Lambda$:
\begin{equation*}
  \ell_{x,\Lambda}(w):=\max\{j\in\NN\union\{0\}\mid x^{-j}w\in\Lambda\}\text.
\end{equation*}
The permutation is defined by
\begin{equation*}
  p_x^{(\Lambda)}(w):=\begin{cases}
    xw  &(xw\in\Lambda)\\
    x^{-\ell_{x,\Lambda}(w)}w
      &(xw\notin\Lambda)\text,
  \end{cases}
\end{equation*}
see \cref{figperm3}.
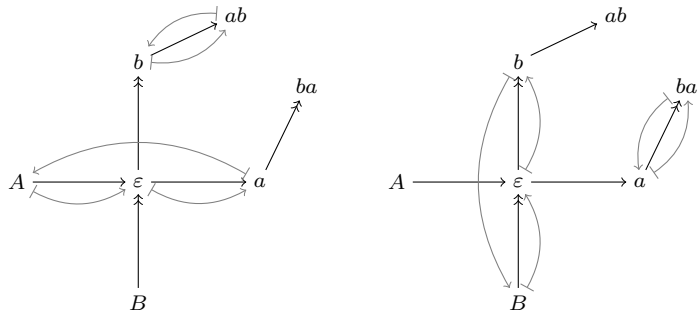
\begin{figure}
  \centering
  \input{perm3.tex}
  \caption{The action of the permutation $p_a^{(\Lambda)}$ and $p_b^{(\Lambda)}$
    on the set $\Lambda=\{\epsilon,a,b,A,B,ab,ba\}$.
    Points without $\mapsto$ are fixed.}
  \label{figperm3}
\end{figure}
A similar argument as above shows that the girth of the corresponding
Cayley graph is at least~$2n+1$, when $B_n\subseteq\Lambda$.

\begin{figure}[p]
  \centering
  \begin{subfigure}[b]{.48\textwidth}
    \centering
    \input{ecf60.tex}
    \caption{\centering
      Generators: $\tilde p_a^{(1)},\tilde p_b^{(1)}$,
      $5!/2=60$~nodes, girth~$3$}
    \label{fig60}
  \end{subfigure}
  \hfill
  \begin{subfigure}[b]{.48\textwidth}
    \centering
    \input{ecf120.tex}
    \caption{\centering
      Generators: $p_a^{(1)},p_b^{(1)}$,
      $5!=120$~nodes, girth~$6$, cf.~\cref{figH1}}
    \label{fig120}
  \end{subfigure}

  \begin{subfigure}[b]{.48\textwidth}
    \centering
    \input{ecf720.tex}
    \caption{\centering
      Generators: $p_a^{(\Lambda)},p_b^{(\Lambda)}$
      with $\Lambda=\{\epsilon,a,b,A,B,ab\}$,
      $720$~nodes, girth~$3$}
    \label{fig720}
  \end{subfigure}
  \hfill
  \begin{subfigure}[b]{.48\textwidth}
    \centering
    \input{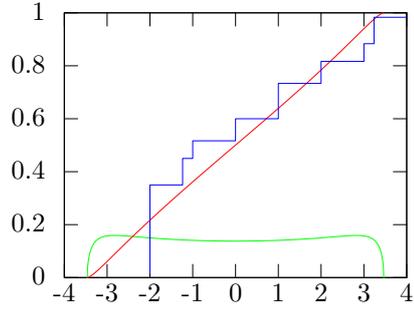}
    \caption{\centering
      Generators: $p_a^{(\Lambda)},p_b^{(\Lambda)}$
      with $\Lambda=$ $\{\epsilon,a,b,A,B,ab,AB\}$,
      $2520$~n., girth~$3$}
    \label{fig2520}
  \end{subfigure}

  \begin{subfigure}[b]{\textwidth}
    \centering
    \input{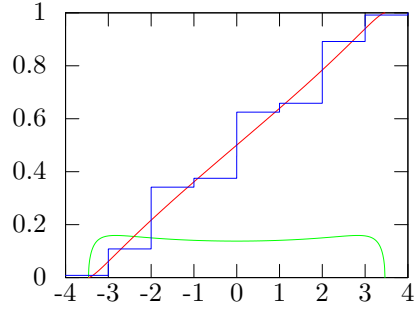}
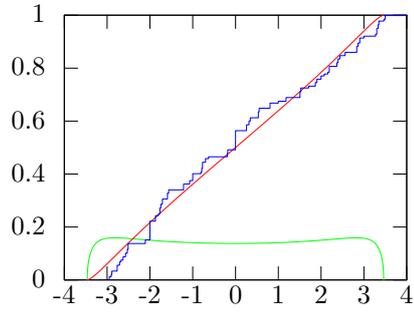
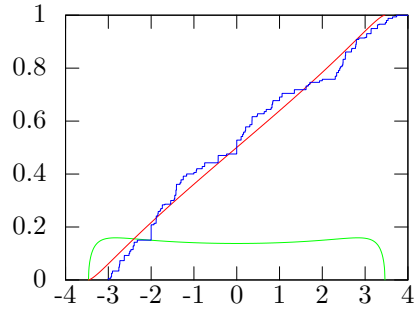
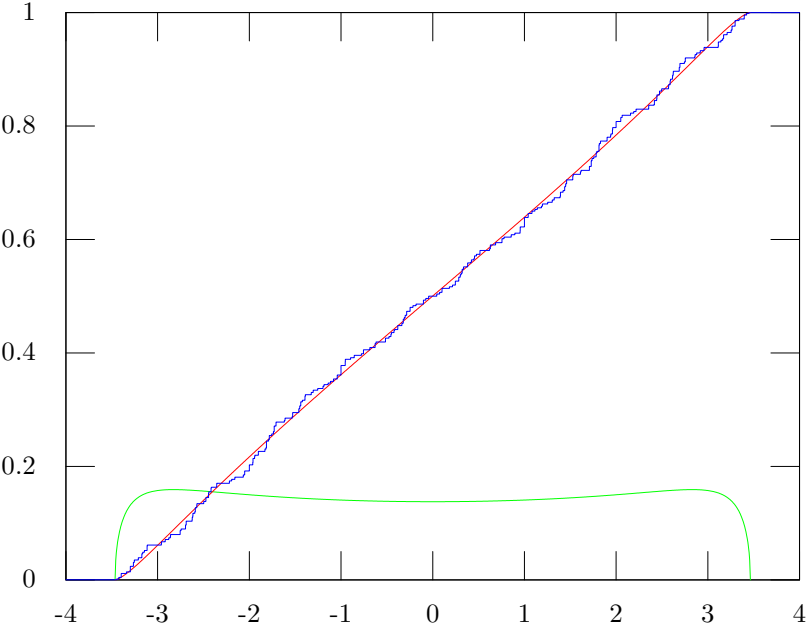
    \caption{Generators: $p_a^{(\Lambda)},p_b^{(\Lambda)}$
      with $\Lambda=\{\epsilon,a,b,A,B,ab,ba\}$, cf.~\cref{figperm3},
      $5040$~nodes, girth~$6$}
    \label{fig5040}
  \end{subfigure}
  \caption{%
    Green: density of states, cf.~\eqref{eq:McKay},
    Red: integrated density of states,
    Blue: eigenvalue counting function of the adjacency matrix
    }
  \label{figConv}
\end{figure}

\clearpage

\renewbibmacro{in:}{
  \ifentrytype{article}
    {}
    {\printtext{\bibstring{in}\intitlepunct}}}
\addcontentsline{toc}{section}{References}
\printbibliography

\end{document}

%% file: B3F2.tex
\begin{tikzpicture}
[scale=1
,inner sep=.5mm
,outer sep=0mm
,font=\footnotesize
,a/.style={->}
,A/.style={<-}
,b/.style={->>}
,B/.style={<<-}
,p/.style={red,|->}
,n/.style={outer sep=.5mm}
]
\draw (0.0,0.0) node[name={e},n] {$\epsilon$};
\draw (1.6,0.0) node[name={a},n] {$a$};
\draw [a] ({e}) -- ({a});
\draw (-1.6,0.0) node[name={A},n] {$A$};
\draw [A] ({e}) -- ({A});
\draw (0.0,1.6) node[name={b},n] {$b$};
\draw [b] ({e}) -- ({b});
\draw (0.0,-1.6) node[name={B},n] {$B$};
\draw [B] ({e}) -- ({B});
\draw (2.56,0.0) node[name={a^2},n] {$a^2$};
\draw [a] ({a}) -- ({a^2});
\draw (2.21702,1.28) node[name={ba},n] {$ba$};
\draw [b] ({a}) -- ({ba});
\draw (2.21702,-1.28) node[name={Ba},n] {$Ba$};
\draw [B] ({a}) -- ({Ba});
\draw (-2.56,0.0) node[name={A^2},n] {$A^2$};
\draw [A] ({A}) -- ({A^2});
\draw (-2.21702,1.28) node[name={bA},n] {$bA$};
\draw [b] ({A}) -- ({bA});
\draw (-2.21702,-1.28) node[name={BA},n] {$BA$};
\draw [B] ({A}) -- ({BA});
\draw (1.28,2.21703) node[name={ab},n] {$ab$};
\draw [a] ({b}) -- ({ab});
\draw (-1.28,2.21702) node[name={Ab},n] {$Ab$};
\draw [A] ({b}) -- ({Ab});
\draw (0.0,2.56) node[name={b^2},n] {$b^2$};
\draw [b] ({b}) -- ({b^2});
\draw (1.28,-2.21703) node[name={aB},n] {$aB$};
\draw [a] ({B}) -- ({aB});
\draw (-1.28,-2.21703) node[name={AB},n] {$AB$};
\draw [A] ({B}) -- ({AB});
\draw (0.0,-2.56) node[name={B^2},n] {$B^2$};
\draw [B] ({B}) -- ({B^2});
\draw (4.096,0.0) node[name={a^3},n] {$a^3$};
\draw [a] ({a^2}) -- ({a^3});
\draw (4.03377,0.71126) node[name={ba^2},n] {$ba^2$};
\draw [b] ({a^2}) -- ({ba^2});
\draw (4.03377,-0.71126) node[name={Ba^2},n] {$Ba^2$};
\draw [B] ({a^2}) -- ({Ba^2});
\draw (3.84898,1.40091) node[name={aba},n] {$aba$};
\draw [a] ({ba}) -- ({aba});
\draw (3.13772,2.63286) node[name={Aba},n] {$Aba$};
\draw [A] ({ba}) -- ({Aba});
\draw (3.54724,2.048) node[name={b^2a},n] {$b^2a$};
\draw [b] ({ba}) -- ({b^2a});
\draw (3.84898,-1.40091) node[name={aBa},n] {$aBa$};
\draw [a] ({Ba}) -- ({aBa});
\draw (3.13772,-2.63286) node[name={ABa},n] {$ABa$};
\draw [A] ({Ba}) -- ({ABa});
\draw (3.54724,-2.048) node[name={B^2a},n] {$B^2a$};
\draw [B] ({Ba}) -- ({B^2a});
\draw (-4.096,0.0) node[name={A^3},n] {$A^3$};
\draw [A] ({A^2}) -- ({A^3});
\draw (-4.03377,0.71126) node[name={bA^2},n] {$bA^2$};
\draw [b] ({A^2}) -- ({bA^2});
\draw (-4.03377,-0.71126) node[name={BA^2},n] {$BA^2$};
\draw [B] ({A^2}) -- ({BA^2});
\draw (-3.13772,2.63286) node[name={abA},n] {$abA$};
\draw [a] ({bA}) -- ({abA});
\draw (-3.84898,1.40091) node[name={AbA},n] {$AbA$};
\draw [A] ({bA}) -- ({AbA});
\draw (-3.54724,2.048) node[name={b^2A},n] {$b^2A$};
\draw [b] ({bA}) -- ({b^2A});
\draw (-3.13772,-2.63286) node[name={aBA},n] {$aBA$};
\draw [a] ({BA}) -- ({aBA});
\draw (-3.84898,-1.40092) node[name={ABA},n] {$ABA$};
\draw [A] ({BA}) -- ({ABA});
\draw (-3.54724,-2.048) node[name={B^2A},n] {$B^2A$};
\draw [B] ({BA}) -- ({B^2A});
\draw (2.048,3.54724) node[name={a^2b},n] {$a^2b$};
\draw [a] ({ab}) -- ({a^2b});
\draw (1.40091,3.84898) node[name={bab},n] {$bab$};
\draw [b] ({ab}) -- ({bab});
\draw (2.63286,3.13772) node[name={Bab},n] {$Bab$};
\draw [B] ({ab}) -- ({Bab});
\draw (-2.048,3.54724) node[name={A^2b},n] {$A^2b$};
\draw [A] ({Ab}) -- ({A^2b});
\draw (-1.40091,3.84898) node[name={bAb},n] {$bAb$};
\draw [b] ({Ab}) -- ({bAb});
\draw (-2.63286,3.13772) node[name={BAb},n] {$BAb$};
\draw [B] ({Ab}) -- ({BAb});
\draw (0.71126,4.03377) node[name={ab^2},n] {$ab^2$};
\draw [a] ({b^2}) -- ({ab^2});
\draw (-0.71126,4.03377) node[name={Ab^2},n] {$Ab^2$};
\draw [A] ({b^2}) -- ({Ab^2});
\draw (0.0,4.096) node[name={b^3},n] {$b^3$};
\draw [b] ({b^2}) -- ({b^3});
\draw (2.048,-3.54724) node[name={a^2B},n] {$a^2B$};
\draw [a] ({aB}) -- ({a^2B});
\draw (2.63286,-3.13772) node[name={baB},n] {$baB$};
\draw [b] ({aB}) -- ({baB});
\draw (1.40092,-3.84898) node[name={BaB},n] {$BaB$};
\draw [B] ({aB}) -- ({BaB});
\draw (-2.048,-3.54724) node[name={A^2B},n] {$A^2B$};
\draw [A] ({AB}) -- ({A^2B});
\draw (-2.63286,-3.13772) node[name={bAB},n] {$bAB$};
\draw [b] ({AB}) -- ({bAB});
\draw (-1.40091,-3.84898) node[name={BAB},n] {$BAB$};
\draw [B] ({AB}) -- ({BAB});
\draw (0.71126,-4.03377) node[name={aB^2},n] {$aB^2$};
\draw [a] ({B^2}) -- ({aB^2});
\draw (-0.71126,-4.03377) node[name={AB^2},n] {$AB^2$};
\draw [A] ({B^2}) -- ({AB^2});
\draw (0.0,-4.096) node[name={B^3},n] {$B^3$};
\draw [B] ({B^2}) -- ({B^3});
\end{tikzpicture}

%% file: perm1.tex
\begin{tikzpicture}
[scale=1
,inner sep=.5mm
,outer sep=0mm
,font=\footnotesize
,a/.style={->}
,A/.style={<-}
,b/.style={->>}
,B/.style={<<-}
,p/.style={gray,|->}
,n/.style={outer sep=.5mm}
]
\draw (0.0,0.0) node[name={e},n] {$\epsilon$};
\draw (1.6,0.0) node[name={a},n] {$a$};
\draw [a] ({e}) -- ({a});
\draw (-1.6,0.0) node[name={A},n] {$A$};
\draw [A] ({e}) -- ({A});
\draw (0.0,1.6) node[name={b},n] {$b$};
\draw [b] ({e}) -- ({b});
\draw (0.0,-1.6) node[name={B},n] {$B$};
\draw [B] ({e}) -- ({B});
\draw [p] (e) to [bend right] (a);
\draw [p] (a) to [bend right] (A);
\draw [p] (A) to [bend right] (e);
\draw [p] (b) to [bend right] (B);
\draw [p] (B) to [bend right] (b);

\begin{scope} [xshift=5cm]
\draw (0.0,0.0) node[name={e},n] {$\epsilon$};
\draw (1.6,0.0) node[name={a},n] {$a$};
\draw [a] ({e}) -- ({a});
\draw (-1.6,0.0) node[name={A},n] {$A$};
\draw [A] ({e}) -- ({A});
\draw (0.0,1.6) node[name={b},n] {$b$};
\draw [b] ({e}) -- ({b});
\draw (0.0,-1.6) node[name={B},n] {$B$};
\draw [B] ({e}) -- ({B});
\draw (2.56,0.0) node[name={a^2},n] {$a^2$};
\draw [a] ({a}) -- ({a^2});
\draw (2.21702,1.28) node[name={ba},n] {$ba$};
\draw [b] ({a}) -- ({ba});
\draw (2.21702,-1.28) node[name={Ba},n] {$Ba$};
\draw [B] ({a}) -- ({Ba});
\draw (-2.56,0.0) node[name={A^2},n] {$A^2$};
\draw [A] ({A}) -- ({A^2});
\draw (-2.21702,1.28) node[name={bA},n] {$bA$};
\draw [b] ({A}) -- ({bA});
\draw (-2.21702,-1.28) node[name={BA},n] {$BA$};
\draw [B] ({A}) -- ({BA});
\draw (1.28,2.21703) node[name={ab},n] {$ab$};
\draw [a] ({b}) -- ({ab});
\draw (-1.28,2.21702) node[name={Ab},n] {$Ab$};
\draw [A] ({b}) -- ({Ab});
\draw (0.0,2.56) node[name={b^2},n] {$b^2$};
\draw [b] ({b}) -- ({b^2});
\draw (1.28,-2.21703) node[name={aB},n] {$aB$};
\draw [a] ({B}) -- ({aB});
\draw (-1.28,-2.21703) node[name={AB},n] {$AB$};
\draw [A] ({B}) -- ({AB});
\draw (0.0,-2.56) node[name={B^2},n] {$B^2$};
\draw [B] ({B}) -- ({B^2});
\draw [p] (e) to [bend right] (a);
\draw [p] (a) to [bend right] (a^2);
\draw [p] (A) to [bend right] (e);
\draw [p] (b) to [bend right] (ab);
\draw [p] (B) to [bend right] (aB);
\draw [p] (a^2) to [bend right] (A^2);
\draw [p] (ba) to [bend right] (BA);
\draw [p] (Ba) to [bend right] (bA);
\draw [p] (A^2) to [bend right] (A);
\draw [p] (bA) to [bend right] (Ba);
\draw [p] (BA) to [bend right] (ba);
\draw [p] (ab) to [bend right] (AB);
\draw [p] (Ab) to [bend right] (b);
\draw [p] (b^2) to [bend right] (B^2);
\draw [p] (aB) to [bend right] (Ab);
\draw [p] (AB) to [bend right] (B);
\draw [p] (B^2) to [bend right] (b^2);
\end{scope}

\end{tikzpicture}

%% file: H1.tex
\begin{tikzpicture}
  [scale=6
  ,inner sep=0mm
  ,outer sep=0mm
  ,font=\footnotesize
  ]
  \clip (0,0) circle (1);
  \node (1) at (0.0,0.0) {$\bullet$};
  \node (A) at (-0.33333,0.0) {$\bullet$};
  \node (AB) at (-0.51962,-0.3) {$\bullet$};
  \node (Ab) at (-0.51962,0.3) {$\bullet$};
  \node (Ab2) at (-0.61859,0.35714) {$\circ$};
  \node (B) at (0.0,-0.33333) {$\bullet$};
  \node (a) at (0.33333,0.0) {$\bullet$};
  \node (a2) at (0.6,0.0) {$\bullet$};
  \node (a2B) at (0.70343,-0.12403) {$\circ$};
  \node (a2b) at (0.70343,0.12403) {$\circ$};
  \node (a2b2) at (0.76596,0.13506) {$\circ$};
  \node (aB) at (0.51962,-0.3) {$\bullet$};
  \node (ab) at (0.51962,0.3) {$\bullet$};
  \node (ab2) at (0.61859,0.35714) {$\circ$};
  \node (b) at (0.0,0.33333) {$\bullet$};
  \node (b2) at (0.0,0.6) {$\bullet$};
  \node (A2) at (-0.6,0.0) {$\bullet$};
  \node (A2B) at (-0.70343,-0.12403) {$\circ$};
  \node (A2B2) at (-0.76596,-0.13506) {$\circ$};
  \node (A2b) at (-0.70343,0.12403) {$\circ$};
  \node (A2b2) at (-0.76596,0.13506) {$\circ$};
  \node (A2b3) at (-0.80575,0.14208) {$\circ$};
  \node (AB2) at (-0.61859,-0.35714) {$\circ$};
  \node (AB2A) at (-0.69505,-0.34907) {$\circ$};
  \node (AB2a) at (-0.64982,-0.4274) {$\circ$};
  \node (ABA) at (-0.67121,-0.2443) {$\circ$};
  \node (ABAb) at (-0.7451,-0.22307) {$\circ$};
  \node (ABAba) at (-0.78821,-0.21941) {$\circ$};
  \node (ABa) at (-0.54717,-0.45913) {$\circ$};
  \node (ABab) at (-0.56574,-0.53374) {$\circ$};
  \node (Ab2A) at (-0.69505,0.34907) {$\circ$};
  \node (Ab2a) at (-0.64982,0.4274) {$\circ$};
  \node (Ab3) at (-0.67358,0.38889) {$\circ$};
  \node (AbA) at (-0.67121,0.2443) {$\circ$};
  \node (AbAB) at (-0.7451,0.22307) {$\circ$};
  \node (AbABa) at (-0.78821,0.21941) {$\circ$};
  \node (Aba) at (-0.54717,0.45913) {$\circ$};
  \node (AbaB) at (-0.56574,0.53374) {$\circ$};
  \node (B2) at (0.0,-0.6) {$\bullet$};
  \node (B2A) at (-0.12403,-0.70343) {$\circ$};
  \node (B2A2) at (-0.13506,-0.76596) {$\circ$};
  \node (B2a) at (0.12403,-0.70343) {$\circ$};
  \node (B2a2) at (0.13506,-0.76596) {$\circ$};
  \node (BA) at (-0.3,-0.51962) {$\bullet$};
  \node (BA2) at (-0.35714,-0.61859) {$\circ$};
  \node (BA2B) at (-0.34907,-0.69505) {$\circ$};
  \node (BAb) at (-0.45913,-0.54717) {$\circ$};
  \node (BAba) at (-0.53374,-0.56574) {$\circ$};
  \node (Ba) at (0.3,-0.51962) {$\bullet$};
  \node (Ba2) at (0.35714,-0.61859) {$\circ$};
  \node (Ba2B) at (0.34907,-0.69505) {$\circ$};
  \node (Bab) at (0.45913,-0.54717) {$\circ$};
  \node (BabA) at (0.53374,-0.56574) {$\circ$};
  \node (a2B2) at (0.76596,-0.13506) {$\circ$};
  \node (a2B2A) at (0.80285,-0.15767) {$\circ$};
  \node (a2B2a) at (0.80836,-0.12642) {$\circ$};
  \node (a2BA) at (0.75681,-0.17937) {$\circ$};
  \node (a2BAb) at (0.79232,-0.20409) {$\circ$};
  \node (a2Ba) at (0.77252,-9.029e-2) {$\circ$};
  \node (a2Bab) at (0.81434,-7.921e-2) {$\circ$};
  \node (a2BabA) at (0.84269,-7.647e-2) {$\circ$};
  \node (a2b2A) at (0.80285,0.15767) {$\circ$};
  \node (a2b2a) at (0.80836,0.12642) {$\circ$};
  \node (a2b3) at (0.80575,0.14208) {$\circ$};
  \node (a2bA) at (0.75681,0.17937) {$\circ$};
  \node (a2bAB) at (0.79232,0.20409) {$\circ$};
  \node (a2ba) at (0.77252,9.029e-2) {$\circ$};
  \node (a2baB) at (0.81434,7.921e-2) {$\circ$};
  \node (a2baBA) at (0.84269,7.647e-2) {$\circ$};
  \node (a3) at (0.71429,0.0) {$\circ$};
  \node (a3B) at (0.77646,-4.522e-2) {$\circ$};
  \node (a3B2) at (0.8168,-4.757e-2) {$\circ$};
  \node (a3b) at (0.77646,4.522e-2) {$\circ$};
  \node (a3b2) at (0.8168,4.757e-2) {$\circ$};
  \node (a3b3) at (0.84472,4.92e-2) {$\circ$};
  \node (aB2) at (0.61859,-0.35714) {$\circ$};
  \node (aB2A) at (0.64982,-0.4274) {$\circ$};
  \node (aB2A2) at (0.68358,-0.4496) {$\circ$};
  \node (aB2a) at (0.69505,-0.34907) {$\circ$};
  \node (aB2a2) at (0.73115,-0.3672) {$\circ$};
  \node (aBA) at (0.54717,-0.45913) {$\circ$};
  \node (aBA2) at (0.59581,-0.49995) {$\circ$};
  \node (aBA2B) at (0.63684,-0.51366) {$\circ$};
  \node (aBAb) at (0.56574,-0.53374) {$\circ$};
  \node (aBAba) at (0.58412,-0.57291) {$\circ$};
  \node (aBa) at (0.67121,-0.2443) {$\circ$};
  \node (aBa2) at (0.73087,-0.26602) {$\circ$};
  \node (aBa2B) at (0.76327,-0.29469) {$\circ$};
  \node (aBab) at (0.7451,-0.22307) {$\circ$};
  \node (aBabA) at (0.78821,-0.21941) {$\circ$};
  \node (ab2A) at (0.64982,0.4274) {$\circ$};
  \node (ab2A2) at (0.68358,0.4496) {$\circ$};
  \node (ab2a) at (0.69505,0.34907) {$\circ$};
  \node (ab2a2) at (0.73115,0.3672) {$\circ$};
  \node (ab3) at (0.67358,0.38889) {$\circ$};
  \node (abA) at (0.54717,0.45913) {$\circ$};
  \node (abA2) at (0.59581,0.49995) {$\circ$};
  \node (abA2b) at (0.63684,0.51366) {$\circ$};
  \node (abAB) at (0.56574,0.53374) {$\circ$};
  \node (abABa) at (0.58412,0.57291) {$\circ$};
  \node (aba) at (0.67121,0.2443) {$\circ$};
  \node (aba2) at (0.73087,0.26602) {$\circ$};
  \node (aba2b) at (0.76327,0.29469) {$\circ$};
  \node (abaB) at (0.7451,0.22307) {$\circ$};
  \node (abaBA) at (0.78821,0.21941) {$\circ$};
  \node (b2A) at (-0.12403,0.70343) {$\circ$};
  \node (b2A2) at (-0.13506,0.76596) {$\circ$};
  \node (b2a) at (0.12403,0.70343) {$\circ$};
  \node (b2a2) at (0.13506,0.76596) {$\circ$};
  \node (b3) at (0.0,0.71429) {$\circ$};
  \node (bA) at (-0.3,0.51962) {$\bullet$};
  \node (bA2) at (-0.35714,0.61859) {$\circ$};
  \node (bA2b) at (-0.34907,0.69505) {$\circ$};
  \node (bAB) at (-0.45913,0.54717) {$\circ$};
  \node (bABa) at (-0.53374,0.56574) {$\circ$};
  \node (ba) at (0.3,0.51962) {$\bullet$};
  \node (ba2) at (0.35714,0.61859) {$\circ$};
  \node (ba2b) at (0.34907,0.69505) {$\circ$};
  \node (baB) at (0.45913,0.54717) {$\circ$};
  \node (baBA) at (0.53374,0.56574) {$\circ$};
  \draw [->] (1) -- (a);
  \draw [->] (AB) -- (ABa);
  \draw [->] (AB2) -- (AB2a);
  \draw [->] (ABAb) -- (ABAba);
  \draw [->] (Ab) -- (Aba);
  \draw [->] (Ab2) -- (Ab2a);
  \draw [->] (AbAB) -- (AbABa);
  \draw [->] (B) -- (Ba);
  \draw [->] (B2) -- (B2a);
  \draw [->] (B2a) -- (B2a2);
  \draw [->] (BAb) -- (BAba);
  \draw [->] (Ba) -- (Ba2);
  \draw [->] (a) -- (a2);
  \draw [->] (a2) -- (a3);
  \draw [->] (a2B) -- (a2Ba);
  \draw [->] (a2B2) -- (a2B2a);
  \draw [->] (a2b) -- (a2ba);
  \draw [->] (a2b2) -- (a2b2a);
  \draw [->] (aB) -- (aBa);
  \draw [->] (aB2) -- (aB2a);
  \draw [->] (aB2a) -- (aB2a2);
  \draw [->] (aBAb) -- (aBAba);
  \draw [->] (aBa) -- (aBa2);
  \draw [->] (ab) -- (aba);
  \draw [->] (ab2) -- (ab2a);
  \draw [->] (ab2a) -- (ab2a2);
  \draw [->] (abAB) -- (abABa);
  \draw [->] (aba) -- (aba2);
  \draw [->] (b) -- (ba);
  \draw [->] (b2) -- (b2a);
  \draw [->] (b2a) -- (b2a2);
  \draw [->] (bAB) -- (bABa);
  \draw [->] (ba) -- (ba2);
  \draw [->>] (1) -- (b);
  \draw [->>] (A) -- (Ab);
  \draw [->>] (A2) -- (A2b);
  \draw [->>] (A2b) -- (A2b2);
  \draw [->>] (A2b2) -- (A2b3);
  \draw [->>] (ABA) -- (ABAb);
  \draw [->>] (ABa) -- (ABab);
  \draw [->>] (Ab) -- (Ab2);
  \draw [->>] (Ab2) -- (Ab3);
  \draw [->>] (BA) -- (BAb);
  \draw [->>] (Ba) -- (Bab);
  \draw [->>] (a) -- (ab);
  \draw [->>] (a2) -- (a2b);
  \draw [->>] (a2BA) -- (a2BAb);
  \draw [->>] (a2Ba) -- (a2Bab);
  \draw [->>] (a2b) -- (a2b2);
  \draw [->>] (a2b2) -- (a2b3);
  \draw [->>] (a3) -- (a3b);
  \draw [->>] (a3b) -- (a3b2);
  \draw [->>] (a3b2) -- (a3b3);
  \draw [->>] (aBA) -- (aBAb);
  \draw [->>] (aBa) -- (aBab);
  \draw [->>] (ab) -- (ab2);
  \draw [->>] (ab2) -- (ab3);
  \draw [->>] (abA2) -- (abA2b);
  \draw [->>] (aba2) -- (aba2b);
  \draw [->>] (b) -- (b2);
  \draw [->>] (b2) -- (b3);
  \draw [->>] (bA2) -- (bA2b);
  \draw [->>] (ba2) -- (ba2b);
  \draw [<-] (1) -- (A);
  \draw [<-] (A) -- (A2);
  \draw [<-] (AB) -- (ABA);
  \draw [<-] (AB2) -- (AB2A);
  \draw [<-] (Ab) -- (AbA);
  \draw [<-] (Ab2) -- (Ab2A);
  \draw [<-] (B) -- (BA);
  \draw [<-] (B2) -- (B2A);
  \draw [<-] (B2A) -- (B2A2);
  \draw [<-] (BA) -- (BA2);
  \draw [<-] (Bab) -- (BabA);
  \draw [<-] (a2B) -- (a2BA);
  \draw [<-] (a2B2) -- (a2B2A);
  \draw [<-] (a2Bab) -- (a2BabA);
  \draw [<-] (a2b) -- (a2bA);
  \draw [<-] (a2b2) -- (a2b2A);
  \draw [<-] (a2baB) -- (a2baBA);
  \draw [<-] (aB) -- (aBA);
  \draw [<-] (aB2) -- (aB2A);
  \draw [<-] (aB2A) -- (aB2A2);
  \draw [<-] (aBA) -- (aBA2);
  \draw [<-] (aBab) -- (aBabA);
  \draw [<-] (ab) -- (abA);
  \draw [<-] (ab2) -- (ab2A);
  \draw [<-] (ab2A) -- (ab2A2);
  \draw [<-] (abA) -- (abA2);
  \draw [<-] (abaB) -- (abaBA);
  \draw [<-] (b) -- (bA);
  \draw [<-] (b2) -- (b2A);
  \draw [<-] (b2A) -- (b2A2);
  \draw [<-] (bA) -- (bA2);
  \draw [<-] (baB) -- (baBA);
  \draw [<<-] (1) -- (B);
  \draw [<<-] (A) -- (AB);
  \draw [<<-] (A2) -- (A2B);
  \draw [<<-] (A2B) -- (A2B2);
  \draw [<<-] (AB) -- (AB2);
  \draw [<<-] (AbA) -- (AbAB);
  \draw [<<-] (Aba) -- (AbaB);
  \draw [<<-] (B) -- (B2);
  \draw [<<-] (BA2) -- (BA2B);
  \draw [<<-] (Ba2) -- (Ba2B);
  \draw [<<-] (a) -- (aB);
  \draw [<<-] (a2) -- (a2B);
  \draw [<<-] (a2B) -- (a2B2);
  \draw [<<-] (a2bA) -- (a2bAB);
  \draw [<<-] (a2ba) -- (a2baB);
  \draw [<<-] (a3) -- (a3B);
  \draw [<<-] (a3B) -- (a3B2);
  \draw [<<-] (aB) -- (aB2);
  \draw [<<-] (aBA2) -- (aBA2B);
  \draw [<<-] (aBa2) -- (aBa2B);
  \draw [<<-] (abA) -- (abAB);
  \draw [<<-] (aba) -- (abaB);
  \draw [<<-] (bA) -- (bAB);
  \draw [<<-] (ba) -- (baB);
  \draw [densely dotted, thin, gray] (A2B) .. controls ($ (-0.3701,-0.12403) + .1*(rand,rand) $) and ($ (0.26248,0.49995) + .1*(rand,rand) $) .. (abA2);
  \draw [densely dotted, thin, gray] (A2B2) .. controls ($ (-0.43262997,-0.13506) + .1*(rand,rand) $) and ($ (0.35024998,0.4496) + .1*(rand,rand) $) .. (ab2A2);
  \draw [densely dotted, thin, gray] (A2b) .. controls ($ (-0.3701,0.12403) + .1*(rand,rand) $) and ($ (0.26248,-0.49995) + .1*(rand,rand) $) .. (aBA2);
  \draw [densely dotted, thin, gray] (A2b2) .. controls ($ (-0.43262997,0.13506) + .1*(rand,rand) $) and ($ (0.35024998,-0.4496) + .1*(rand,rand) $) .. (aB2A2);
  \draw [densely dotted, thin, gray] (A2b3) .. controls ($ (-0.47242,0.14208) + .1*(rand,rand) $) and ($ (0.51139,4.92e-2) + .1*(rand,rand) $) .. (a3b3);
  \draw [densely dotted, thin, gray] (A2b3) .. controls ($ (-0.80575,0.47540998) + .1*(rand,rand) $) and ($ (-0.76596,-0.46839) + .1*(rand,rand) $) .. (A2B2);
  \draw [densely dotted, thin, gray] (AB2A) .. controls ($ (-0.69505,-1.5740007e-2) + .1*(rand,rand) $) and ($ (0.68358,-0.78293) + .1*(rand,rand) $) .. (aB2A2);
  \draw [densely dotted, thin, gray] (AB2a) .. controls ($ (-0.31649002,-0.4274) + .1*(rand,rand) $) and ($ (0.46952,0.15767) + .1*(rand,rand) $) .. (a2b2A);
  \draw [densely dotted, thin, gray] (AB2a) .. controls ($ (-0.64982,-9.406999e-2) + .1*(rand,rand) $) and ($ (-0.12403,0.3701) + .1*(rand,rand) $) .. (b2A);
  \draw [densely dotted, thin, gray] (ABAb) .. controls ($ (-0.7451,0.11026001) + .1*(rand,rand) $) and ($ (-0.12403,-1.03676) + .1*(rand,rand) $) .. (B2A);
  \draw [densely dotted, thin, gray] (ABAba) .. controls ($ (-0.45487997,-0.21941) + .1*(rand,rand) $) and ($ (-0.6824,-0.69505) + .1*(rand,rand) $) .. (BA2B);
  \draw [densely dotted, thin, gray] (ABAba) .. controls ($ (-0.78821,0.11392) + .1*(rand,rand) $) and ($ (0.84269,-0.25686002) + .1*(rand,rand) $) .. (a2baBA);
  \draw [densely dotted, thin, gray] (ABa) .. controls ($ (-0.21383998,-0.45913) + .1*(rand,rand) $) and ($ (0.42348,0.17937) + .1*(rand,rand) $) .. (a2bA);
  \draw [densely dotted, thin, gray] (ABab) .. controls ($ (-0.23240998,-0.53374) + .1*(rand,rand) $) and ($ (0.42994002,0.29469) + .1*(rand,rand) $) .. (aba2b);
  \draw [densely dotted, thin, gray] (ABab) .. controls ($ (-0.56574,-0.20040998) + .1*(rand,rand) $) and ($ (0.68358,0.116270006) + .1*(rand,rand) $) .. (ab2A2);
  \draw [densely dotted, thin, gray] (Ab2A) .. controls ($ (-0.69505,0.6824) + .1*(rand,rand) $) and ($ (0.76327,-0.62802005) + .1*(rand,rand) $) .. (aBa2B);
  \draw [densely dotted, thin, gray] (Ab2a) .. controls ($ (-0.31649002,0.4274) + .1*(rand,rand) $) and ($ (0.46952,-0.15767) + .1*(rand,rand) $) .. (a2B2A);
  \draw [densely dotted, thin, gray] (Ab2a) .. controls ($ (-0.64982,0.76073) + .1*(rand,rand) $) and ($ (0.45913,0.21383998) + .1*(rand,rand) $) .. (baB);
  \draw [densely dotted, thin, gray] (Ab3) .. controls ($ (-0.34025,0.38889) + .1*(rand,rand) $) and ($ (-1.13908,0.14208) + .1*(rand,rand) $) .. (A2b3);
  \draw [densely dotted, thin, gray] (Ab3) .. controls ($ (-0.67358,0.72222) + .1*(rand,rand) $) and ($ (-0.61859,-0.69047) + .1*(rand,rand) $) .. (AB2);
  \draw [densely dotted, thin, gray] (AbA) .. controls ($ (-0.67121,0.57763) + .1*(rand,rand) $) and ($ (0.3,-0.85295) + .1*(rand,rand) $) .. (Ba);
  \draw [densely dotted, thin, gray] (AbABa) .. controls ($ (-0.45487997,0.21941) + .1*(rand,rand) $) and ($ (-0.6824,0.69505) + .1*(rand,rand) $) .. (bA2b);
  \draw [densely dotted, thin, gray] (AbABa) .. controls ($ (-0.78821,0.55274) + .1*(rand,rand) $) and ($ (-0.53374,0.23240998) + .1*(rand,rand) $) .. (bABa);
  \draw [densely dotted, thin, gray] (Aba) .. controls ($ (-0.21383998,0.45913) + .1*(rand,rand) $) and ($ (0.42348,-0.17937) + .1*(rand,rand) $) .. (a2BA);
  \draw [densely dotted, thin, gray] (Aba) .. controls ($ (-0.54717,0.79245996) + .1*(rand,rand) $) and ($ (0.73087,-6.7310005e-2) + .1*(rand,rand) $) .. (aba2);
  \draw [densely dotted, thin, gray] (AbaB) .. controls ($ (-0.23240998,0.53374) + .1*(rand,rand) $) and ($ (0.42994002,-0.29469) + .1*(rand,rand) $) .. (aBa2B);
  \draw [densely dotted, thin, gray] (B2A) .. controls ($ (-0.12403,-0.3701) + .1*(rand,rand) $) and ($ (-0.64982,9.406999e-2) + .1*(rand,rand) $) .. (Ab2a);
  \draw [densely dotted, thin, gray] (B2A2) .. controls ($ (-0.13506,-0.43262997) + .1*(rand,rand) $) and ($ (0.63684,-0.84699) + .1*(rand,rand) $) .. (aBA2B);
  \draw [densely dotted, thin, gray] (B2a) .. controls ($ (0.12403,-0.3701) + .1*(rand,rand) $) and ($ (0.64982,9.406999e-2) + .1*(rand,rand) $) .. (ab2A);
  \draw [densely dotted, thin, gray] (B2a2) .. controls ($ (0.13506,-0.43262997) + .1*(rand,rand) $) and ($ (0.79232,-0.12924) + .1*(rand,rand) $) .. (a2bAB);
  \draw [densely dotted, thin, gray] (B2a2) .. controls ($ (0.46839,-0.76596) + .1*(rand,rand) $) and ($ (0.48347,4.757e-2) + .1*(rand,rand) $) .. (a3b2);
  \draw [densely dotted, thin, gray] (BA2) .. controls ($ (-0.35714,-0.28526) + .1*(rand,rand) $) and ($ (0.75681,-0.5127) + .1*(rand,rand) $) .. (a2BA);
  \draw [densely dotted, thin, gray] (BA2B) .. controls ($ (-1.5740007e-2,-0.69505) + .1*(rand,rand) $) and ($ (0.45899,0.20409) + .1*(rand,rand) $) .. (a2bAB);
  \draw [densely dotted, thin, gray] (BAb) .. controls ($ (-0.45913,-0.21383998) + .1*(rand,rand) $) and ($ (0.64982,-0.76073) + .1*(rand,rand) $) .. (aB2A);
  \draw [densely dotted, thin, gray] (BAba) .. controls ($ (-0.20040998,-0.56574) + .1*(rand,rand) $) and ($ (0.30350998,-0.51366) + .1*(rand,rand) $) .. (aBA2B);
  \draw [densely dotted, thin, gray] (BAba) .. controls ($ (-0.53374,-0.23240998) + .1*(rand,rand) $) and ($ (-0.78821,-0.55274) + .1*(rand,rand) $) .. (ABAba);
  \draw [densely dotted, thin, gray] (Ba2) .. controls ($ (0.35714,-0.28526) + .1*(rand,rand) $) and ($ (0.59581,0.16661999) + .1*(rand,rand) $) .. (abA2);
  \draw [densely dotted, thin, gray] (Ba2) .. controls ($ (0.69047,-0.61859) + .1*(rand,rand) $) and ($ (0.44313,4.522e-2) + .1*(rand,rand) $) .. (a3b);
  \draw [densely dotted, thin, gray] (Ba2B) .. controls ($ (0.6824,-0.69505) + .1*(rand,rand) $) and ($ (0.45487997,-0.21941) + .1*(rand,rand) $) .. (aBabA);
  \draw [densely dotted, thin, gray] (Bab) .. controls ($ (0.45913,-0.21383998) + .1*(rand,rand) $) and ($ (-0.64982,-0.76073) + .1*(rand,rand) $) .. (AB2a);
  \draw [densely dotted, thin, gray] (Bab) .. controls ($ (0.79245996,-0.54717) + .1*(rand,rand) $) and ($ (-1.07843,-0.22307) + .1*(rand,rand) $) .. (ABAb);
  \draw [densely dotted, thin, gray] (BabA) .. controls ($ (0.53374,-0.23240998) + .1*(rand,rand) $) and ($ (0.78821,-0.55274) + .1*(rand,rand) $) .. (aBabA);
  \draw [densely dotted, thin, gray] (a2B2A) .. controls ($ (0.80285,0.17566) + .1*(rand,rand) $) and ($ (0.69505,1.5740007e-2) + .1*(rand,rand) $) .. (ab2a);
  \draw [densely dotted, thin, gray] (a2B2a) .. controls ($ (0.80836,0.20691) + .1*(rand,rand) $) and ($ (0.13506,-1.09929) + .1*(rand,rand) $) .. (B2a2);
  \draw [densely dotted, thin, gray] (a2B2a) .. controls ($ (1.14169,-0.12642) + .1*(rand,rand) $) and ($ (-1.02838,0.34907) + .1*(rand,rand) $) .. (Ab2A);
  \draw [densely dotted, thin, gray] (a2BAb) .. controls ($ (0.79232,0.12924) + .1*(rand,rand) $) and ($ (0.13506,0.43262997) + .1*(rand,rand) $) .. (b2a2);
  \draw [densely dotted, thin, gray] (a2BAb) .. controls ($ (1.12565,-0.20409) + .1*(rand,rand) $) and ($ (0.20040998,0.56574) + .1*(rand,rand) $) .. (baBA);
  \draw [densely dotted, thin, gray] (a2Ba) .. controls ($ (1.10585,-9.029e-2) + .1*(rand,rand) $) and ($ (-1.00454,0.2443) + .1*(rand,rand) $) .. (AbA);
  \draw [densely dotted, thin, gray] (a2Bab) .. controls ($ (0.81434,0.25412) + .1*(rand,rand) $) and ($ (0.69505,-0.6824) + .1*(rand,rand) $) .. (aB2a);
  \draw [densely dotted, thin, gray] (a2Bab) .. controls ($ (1.14767,-7.921e-2) + .1*(rand,rand) $) and ($ (0.23240998,-0.53374) + .1*(rand,rand) $) .. (aBAb);
  \draw [densely dotted, thin, gray] (a2BabA) .. controls ($ (0.84269,0.25686002) + .1*(rand,rand) $) and ($ (-0.78821,-0.11392) + .1*(rand,rand) $) .. (AbABa);
  \draw [densely dotted, thin, gray] (a2b2A) .. controls ($ (0.80285,0.491) + .1*(rand,rand) $) and ($ (0.56574,0.20040998) + .1*(rand,rand) $) .. (abAB);
  \draw [densely dotted, thin, gray] (a2b2a) .. controls ($ (0.80836,0.45975) + .1*(rand,rand) $) and ($ (-0.34907,-1.02838) + .1*(rand,rand) $) .. (BA2B);
  \draw [densely dotted, thin, gray] (a2b2a) .. controls ($ (1.14169,0.12642) + .1*(rand,rand) $) and ($ (-1.02838,-0.34907) + .1*(rand,rand) $) .. (AB2A);
  \draw [densely dotted, thin, gray] (a2b3) .. controls ($ (0.80575,0.47540998) + .1*(rand,rand) $) and ($ (0.76596,-0.46839) + .1*(rand,rand) $) .. (a2B2);
  \draw [densely dotted, thin, gray] (a2b3) .. controls ($ (1.13908,0.14208) + .1*(rand,rand) $) and ($ (0.34025,0.38889) + .1*(rand,rand) $) .. (ab3);
  \draw [densely dotted, thin, gray] (a2bA) .. controls ($ (0.75681,0.5127) + .1*(rand,rand) $) and ($ (-0.35714,0.28526) + .1*(rand,rand) $) .. (bA2);
  \draw [densely dotted, thin, gray] (a2bAB) .. controls ($ (1.12565,0.20409) + .1*(rand,rand) $) and ($ (0.20040998,-0.56574) + .1*(rand,rand) $) .. (BabA);
  \draw [densely dotted, thin, gray] (a2ba) .. controls ($ (0.77252,0.42362002) + .1*(rand,rand) $) and ($ (0.54717,-0.79245996) + .1*(rand,rand) $) .. (aBA);
  \draw [densely dotted, thin, gray] (a2ba) .. controls ($ (1.10585,9.029e-2) + .1*(rand,rand) $) and ($ (-1.00454,-0.2443) + .1*(rand,rand) $) .. (ABA);
  \draw [densely dotted, thin, gray] (a2baB) .. controls ($ (1.14767,7.921e-2) + .1*(rand,rand) $) and ($ (0.23240998,0.53374) + .1*(rand,rand) $) .. (abAB);
  \draw [densely dotted, thin, gray] (a2baBA) .. controls ($ (0.84269,0.4098) + .1*(rand,rand) $) and ($ (0.78821,-0.11392) + .1*(rand,rand) $) .. (abaBA);
  \draw [densely dotted, thin, gray] (a3) .. controls ($ (1.04762,0.0) + .1*(rand,rand) $) and ($ (-0.93333006,0.0) + .1*(rand,rand) $) .. (A2);
  \draw [densely dotted, thin, gray] (a3B) .. controls ($ (1.10979,-4.522e-2) + .1*(rand,rand) $) and ($ (-0.69047,0.61859) + .1*(rand,rand) $) .. (bA2);
  \draw [densely dotted, thin, gray] (a3B2) .. controls ($ (1.15013,-4.757e-2) + .1*(rand,rand) $) and ($ (-0.46839,0.76596) + .1*(rand,rand) $) .. (b2A2);
  \draw [densely dotted, thin, gray] (a3b) .. controls ($ (1.10979,4.522e-2) + .1*(rand,rand) $) and ($ (-0.69047,-0.61859) + .1*(rand,rand) $) .. (BA2);
  \draw [densely dotted, thin, gray] (a3b2) .. controls ($ (1.15013,4.757e-2) + .1*(rand,rand) $) and ($ (-0.46839,-0.76596) + .1*(rand,rand) $) .. (B2A2);
  \draw [densely dotted, thin, gray] (a3b3) .. controls ($ (0.84472,0.38253) + .1*(rand,rand) $) and ($ (0.8168,-0.3809) + .1*(rand,rand) $) .. (a3B2);
  \draw [densely dotted, thin, gray] (a3b3) .. controls ($ (1.17805,4.92e-2) + .1*(rand,rand) $) and ($ (0.47242,0.14208) + .1*(rand,rand) $) .. (a2b3);
  \draw [densely dotted, thin, gray] (aB2A) .. controls ($ (0.64982,-9.406999e-2) + .1*(rand,rand) $) and ($ (0.12403,0.3701) + .1*(rand,rand) $) .. (b2a);
  \draw [densely dotted, thin, gray] (aB2A2) .. controls ($ (0.68358,-0.116270006) + .1*(rand,rand) $) and ($ (-0.56574,0.20040998) + .1*(rand,rand) $) .. (AbaB);
  \draw [densely dotted, thin, gray] (aB2a) .. controls ($ (0.69505,-1.5740007e-2) + .1*(rand,rand) $) and ($ (0.80285,-0.17566) + .1*(rand,rand) $) .. (a2b2A);
  \draw [densely dotted, thin, gray] (aB2a2) .. controls ($ (0.73115,-3.386998e-2) + .1*(rand,rand) $) and ($ (0.34907,-1.02838) + .1*(rand,rand) $) .. (Ba2B);
  \draw [densely dotted, thin, gray] (aB2a2) .. controls ($ (1.06448,-0.3672) + .1*(rand,rand) $) and ($ (-1.09929,0.13506) + .1*(rand,rand) $) .. (A2b2);
  \draw [densely dotted, thin, gray] (aBA2) .. controls ($ (0.59581,-0.16661999) + .1*(rand,rand) $) and ($ (0.35714,0.28526) + .1*(rand,rand) $) .. (ba2);
  \draw [densely dotted, thin, gray] (aBA2B) .. controls ($ (0.97017,-0.51366) + .1*(rand,rand) $) and ($ (1.5740007e-2,-0.69505) + .1*(rand,rand) $) .. (Ba2B);
  \draw [densely dotted, thin, gray] (aBAb) .. controls ($ (0.56574,-0.20040998) + .1*(rand,rand) $) and ($ (0.80285,-0.491) + .1*(rand,rand) $) .. (a2B2A);
  \draw [densely dotted, thin, gray] (aBAba) .. controls ($ (0.58412,-0.23958) + .1*(rand,rand) $) and ($ (-0.53374,-0.89907) + .1*(rand,rand) $) .. (BAba);
  \draw [densely dotted, thin, gray] (aBAba) .. controls ($ (0.91744995,-0.57291) + .1*(rand,rand) $) and ($ (-0.89907,0.53374) + .1*(rand,rand) $) .. (AbaB);
  \draw [densely dotted, thin, gray] (aBa2) .. controls ($ (0.73087,6.7310005e-2) + .1*(rand,rand) $) and ($ (-0.54717,-0.79245996) + .1*(rand,rand) $) .. (ABa);
  \draw [densely dotted, thin, gray] (aBa2) .. controls ($ (1.0642,-0.26602) + .1*(rand,rand) $) and ($ (-1.03676,0.12403) + .1*(rand,rand) $) .. (A2b);
  \draw [densely dotted, thin, gray] (aBa2B) .. controls ($ (1.0966,-0.29469) + .1*(rand,rand) $) and ($ (0.50935996,-7.647e-2) + .1*(rand,rand) $) .. (a2BabA);
  \draw [densely dotted, thin, gray] (aBab) .. controls ($ (0.7451,0.11026001) + .1*(rand,rand) $) and ($ (0.12403,-1.03676) + .1*(rand,rand) $) .. (B2a);
  \draw [densely dotted, thin, gray] (aBab) .. controls ($ (1.07843,-0.22307) + .1*(rand,rand) $) and ($ (-0.79245996,-0.54717) + .1*(rand,rand) $) .. (BAb);
  \draw [densely dotted, thin, gray] (aBabA) .. controls ($ (0.78821,0.11392) + .1*(rand,rand) $) and ($ (0.84269,-0.4098) + .1*(rand,rand) $) .. (a2BabA);
  \draw [densely dotted, thin, gray] (ab2A) .. controls ($ (0.64982,0.76073) + .1*(rand,rand) $) and ($ (-0.45913,0.21383998) + .1*(rand,rand) $) .. (bAB);
  \draw [densely dotted, thin, gray] (ab2A2) .. controls ($ (0.68358,0.78293) + .1*(rand,rand) $) and ($ (-0.69505,1.5740007e-2) + .1*(rand,rand) $) .. (Ab2A);
  \draw [densely dotted, thin, gray] (ab2a) .. controls ($ (0.69505,0.6824) + .1*(rand,rand) $) and ($ (0.81434,-0.25412) + .1*(rand,rand) $) .. (a2baB);
  \draw [densely dotted, thin, gray] (ab2a2) .. controls ($ (0.73115,0.70053) + .1*(rand,rand) $) and ($ (-0.13506,-1.09929) + .1*(rand,rand) $) .. (B2A2);
  \draw [densely dotted, thin, gray] (ab2a2) .. controls ($ (1.06448,0.3672) + .1*(rand,rand) $) and ($ (-1.09929,-0.13506) + .1*(rand,rand) $) .. (A2B2);
  \draw [densely dotted, thin, gray] (ab3) .. controls ($ (0.67358,0.72222) + .1*(rand,rand) $) and ($ (0.61859,-0.69047) + .1*(rand,rand) $) .. (aB2);
  \draw [densely dotted, thin, gray] (ab3) .. controls ($ (1.00691,0.38889) + .1*(rand,rand) $) and ($ (-0.33333,0.71429) + .1*(rand,rand) $) .. (b3);
  \draw [densely dotted, thin, gray] (abA) .. controls ($ (0.54717,0.79245996) + .1*(rand,rand) $) and ($ (0.77252,-0.42362002) + .1*(rand,rand) $) .. (a2Ba);
  \draw [densely dotted, thin, gray] (abA2b) .. controls ($ (0.63684,0.84699) + .1*(rand,rand) $) and ($ (-0.13506,0.43262997) + .1*(rand,rand) $) .. (b2A2);
  \draw [densely dotted, thin, gray] (abA2b) .. controls ($ (0.97017,0.51366) + .1*(rand,rand) $) and ($ (1.5740007e-2,0.69505) + .1*(rand,rand) $) .. (ba2b);
  \draw [densely dotted, thin, gray] (abABa) .. controls ($ (0.58412,0.90624) + .1*(rand,rand) $) and ($ (0.53374,-0.89907) + .1*(rand,rand) $) .. (BabA);
  \draw [densely dotted, thin, gray] (abABa) .. controls ($ (0.91744995,0.57291) + .1*(rand,rand) $) and ($ (-0.89907,-0.53374) + .1*(rand,rand) $) .. (ABab);
  \draw [densely dotted, thin, gray] (aba) .. controls ($ (0.67121,0.57763) + .1*(rand,rand) $) and ($ (-0.3,-0.85295) + .1*(rand,rand) $) .. (BA);
  \draw [densely dotted, thin, gray] (aba2) .. controls ($ (1.0642,0.26602) + .1*(rand,rand) $) and ($ (-1.03676,-0.12403) + .1*(rand,rand) $) .. (A2B);
  \draw [densely dotted, thin, gray] (aba2b) .. controls ($ (0.76327,0.62802005) + .1*(rand,rand) $) and ($ (-0.69505,-0.6824) + .1*(rand,rand) $) .. (AB2A);
  \draw [densely dotted, thin, gray] (aba2b) .. controls ($ (1.0966,0.29469) + .1*(rand,rand) $) and ($ (0.50935996,7.647e-2) + .1*(rand,rand) $) .. (a2baBA);
  \draw [densely dotted, thin, gray] (abaB) .. controls ($ (1.07843,0.22307) + .1*(rand,rand) $) and ($ (-0.79245996,0.54717) + .1*(rand,rand) $) .. (bAB);
  \draw [densely dotted, thin, gray] (abaBA) .. controls ($ (0.78821,0.55274) + .1*(rand,rand) $) and ($ (0.53374,0.23240998) + .1*(rand,rand) $) .. (baBA);
  \draw [densely dotted, thin, gray] (b2A) .. controls ($ (-0.12403,1.03676) + .1*(rand,rand) $) and ($ (-0.7451,-0.11026001) + .1*(rand,rand) $) .. (AbAB);
  \draw [densely dotted, thin, gray] (b2A2) .. controls ($ (-0.13506,1.09929) + .1*(rand,rand) $) and ($ (0.73115,-0.70053) + .1*(rand,rand) $) .. (aB2a2);
  \draw [densely dotted, thin, gray] (b2a) .. controls ($ (0.12403,1.03676) + .1*(rand,rand) $) and ($ (0.7451,-0.11026001) + .1*(rand,rand) $) .. (abaB);
  \draw [densely dotted, thin, gray] (b2a2) .. controls ($ (0.13506,1.09929) + .1*(rand,rand) $) and ($ (0.80836,-0.20691) + .1*(rand,rand) $) .. (a2b2a);
  \draw [densely dotted, thin, gray] (b2a2) .. controls ($ (0.46839,0.76596) + .1*(rand,rand) $) and ($ (0.48347,-4.757e-2) + .1*(rand,rand) $) .. (a3B2);
  \draw [densely dotted, thin, gray] (b3) .. controls ($ (0.0,1.04762) + .1*(rand,rand) $) and ($ (0.0,-0.93333006) + .1*(rand,rand) $) .. (B2);
  \draw [densely dotted, thin, gray] (b3) .. controls ($ (0.33333,0.71429) + .1*(rand,rand) $) and ($ (-1.00691,0.38889) + .1*(rand,rand) $) .. (Ab3);
  \draw [densely dotted, thin, gray] (bA) .. controls ($ (-0.3,0.85295) + .1*(rand,rand) $) and ($ (0.67121,-0.57763) + .1*(rand,rand) $) .. (aBa);
  \draw [densely dotted, thin, gray] (bA2b) .. controls ($ (-0.34907,1.02838) + .1*(rand,rand) $) and ($ (0.80836,-0.45975) + .1*(rand,rand) $) .. (a2B2a);
  \draw [densely dotted, thin, gray] (bA2b) .. controls ($ (-1.5740007e-2,0.69505) + .1*(rand,rand) $) and ($ (0.45899,-0.20409) + .1*(rand,rand) $) .. (a2BAb);
  \draw [densely dotted, thin, gray] (bABa) .. controls ($ (-0.20040998,0.56574) + .1*(rand,rand) $) and ($ (0.30350998,0.51366) + .1*(rand,rand) $) .. (abA2b);
  \draw [densely dotted, thin, gray] (bABa) .. controls ($ (-0.53374,0.89907) + .1*(rand,rand) $) and ($ (0.58412,0.23958) + .1*(rand,rand) $) .. (abABa);
  \draw [densely dotted, thin, gray] (ba) .. controls ($ (0.3,0.85295) + .1*(rand,rand) $) and ($ (-0.67121,-0.57763) + .1*(rand,rand) $) .. (ABA);
  \draw [densely dotted, thin, gray] (ba2) .. controls ($ (0.69047,0.61859) + .1*(rand,rand) $) and ($ (0.44313,-4.522e-2) + .1*(rand,rand) $) .. (a3B);
  \draw [densely dotted, thin, gray] (ba2b) .. controls ($ (0.34907,1.02838) + .1*(rand,rand) $) and ($ (0.73115,3.386998e-2) + .1*(rand,rand) $) .. (ab2a2);
  \draw [densely dotted, thin, gray] (ba2b) .. controls ($ (0.6824,0.69505) + .1*(rand,rand) $) and ($ (0.45487997,0.21941) + .1*(rand,rand) $) .. (abaBA);
  \draw [densely dotted, thin, gray] (baB) .. controls ($ (0.79245996,0.54717) + .1*(rand,rand) $) and ($ (-1.07843,0.22307) + .1*(rand,rand) $) .. (AbAB);
  \draw [densely dotted, thin, gray] (baBA) .. controls ($ (0.53374,0.89907) + .1*(rand,rand) $) and ($ (0.58412,-0.90624) + .1*(rand,rand) $) .. (aBAba);
\end{tikzpicture}

%% file: perm2.tex
\begin{tikzpicture}
[scale=1
,inner sep=.5mm
,outer sep=0mm
,font=\footnotesize
,a/.style={->}
,A/.style={<-}
,b/.style={->>}
,B/.style={<<-}
,p/.style={gray,|->}
,n/.style={outer sep=.5mm}
]
\draw (0.0,0.0) node[name={e},n] {$\epsilon$};
\draw (1.6,0.0) node[name={a},n] {$a$};
\draw [a] ({e}) -- ({a});
\draw (-1.6,0.0) node[name={A},n] {$A$};
\draw [A] ({e}) -- ({A});
\draw (0.0,1.6) node[name={b},n] {$b$};
\draw [b] ({e}) -- ({b});
\draw (0.0,-1.6) node[name={B},n] {$B$};
\draw [B] ({e}) -- ({B});
\draw [p] (e) to [bend right] (a);
\draw [p] (a) to [bend right] (A);
\draw [p] (A) to [bend right] (e);

\begin{scope}[xshift=5cm]
\draw (0.0,0.0) node[name={e},n] {$\epsilon$};
\draw (1.6,0.0) node[name={a},n] {$a$};
\draw [a] ({e}) -- ({a});
\draw (-1.6,0.0) node[name={A},n] {$A$};
\draw [A] ({e}) -- ({A});
\draw (0.0,1.6) node[name={b},n] {$b$};
\draw [b] ({e}) -- ({b});
\draw (0.0,-1.6) node[name={B},n] {$B$};
\draw [B] ({e}) -- ({B});
\draw (2.56,0.0) node[name={a^2},n] {$a^2$};
\draw [a] ({a}) -- ({a^2});
\draw (2.21702,1.28) node[name={ba},n] {$ba$};
\draw [b] ({a}) -- ({ba});
\draw (2.21702,-1.28) node[name={Ba},n] {$Ba$};
\draw [B] ({a}) -- ({Ba});
\draw (-2.56,0.0) node[name={A^2},n] {$A^2$};
\draw [A] ({A}) -- ({A^2});
\draw (-2.21702,1.28) node[name={bA},n] {$bA$};
\draw [b] ({A}) -- ({bA});
\draw (-2.21702,-1.28) node[name={BA},n] {$BA$};
\draw [B] ({A}) -- ({BA});
\draw (1.28,2.21703) node[name={ab},n] {$ab$};
\draw [a] ({b}) -- ({ab});
\draw (-1.28,2.21702) node[name={Ab},n] {$Ab$};
\draw [A] ({b}) -- ({Ab});
\draw (0.0,2.56) node[name={b^2},n] {$b^2$};
\draw [b] ({b}) -- ({b^2});
\draw (1.28,-2.21703) node[name={aB},n] {$aB$};
\draw [a] ({B}) -- ({aB});
\draw (-1.28,-2.21703) node[name={AB},n] {$AB$};
\draw [A] ({B}) -- ({AB});
\draw (0.0,-2.56) node[name={B^2},n] {$B^2$};
\draw [B] ({B}) -- ({B^2});
\draw [p] (e) to [bend right] (a);
\draw [p] (a) to [bend right] (a^2);
\draw [p] (A) to [bend right] (e);
\draw [p] (b) to [bend right] (ab);
\draw [p] (B) to [bend right] (aB);
\draw [p] (a^2) to [bend right] (A^2);
\draw [p] (A^2) to [bend right] (A);
\draw [p] (ab) to [bend right] (AB);
\draw [p] (Ab) to [bend right] (b);
\draw [p] (aB) to [bend right] (Ab);
\draw [p] (AB) to [bend right] (B);
\end{scope}
\end{tikzpicture}

%% file: perm3.tex
\begin{tikzpicture}
[scale=1
,inner sep=.5mm
,outer sep=0mm
,font=\footnotesize
,a/.style={->}
,A/.style={<-}
,b/.style={->>}
,B/.style={<<-}
,p/.style={gray,|->}
,n/.style={outer sep=.5mm}
]
\draw (0.0,0.0) node[name={e},outer sep=.5mm] {$\epsilon$};
\draw (1.6,0.0) node[name={a},outer sep=.5mm] {$a$};
\draw [a] ({e}) -- ({a});
\draw (-1.6,0.0) node[name={A},outer sep=.5mm] {$A$};
\draw [A] ({e}) -- ({A});
\draw (0.0,1.6) node[name={b},outer sep=.5mm] {$b$};
\draw [b] ({e}) -- ({b});
\draw (0.0,-1.6) node[name={B},outer sep=.5mm] {$B$};
\draw [B] ({e}) -- ({B});
\draw (1.28,2.21703) node[name={ab},outer sep=.5mm] {$ab$};
\draw [a] ({b}) -- ({ab});
\draw (2.21702,1.28) node[name={ba},outer sep=.5mm] {$ba$};
\draw [b] ({a}) -- ({ba});
\draw [p] (e) to [bend right] (a);
\draw [p] (a) to [bend right] (A);
\draw [p] (A) to [bend right] (e);
\draw [p] (b) to [bend right] (ab);
\draw [p] (ab) to [bend right] (b);

\begin{scope}[xshift=5cm]
\draw (0.0,0.0) node[name={e},outer sep=.5mm] {$\epsilon$};
\draw (1.6,0.0) node[name={a},outer sep=.5mm] {$a$};
\draw [a] ({e}) -- ({a});
\draw (-1.6,0.0) node[name={A},outer sep=.5mm] {$A$};
\draw [A] ({e}) -- ({A});
\draw (0.0,1.6) node[name={b},outer sep=.5mm] {$b$};
\draw [b] ({e}) -- ({b});
\draw (0.0,-1.6) node[name={B},outer sep=.5mm] {$B$};
\draw [B] ({e}) -- ({B});
\draw (1.28,2.21703) node[name={ab},outer sep=.5mm] {$ab$};
\draw [a] ({b}) -- ({ab});
\draw (2.21702,1.28) node[name={ba},outer sep=.5mm] {$ba$};
\draw [b] ({a}) -- ({ba});
\draw [p] (e) to [bend right] (b);
\draw [p] (a) to [bend right] (ba);
\draw [p] (b) to [bend right] (B);
\draw [p] (B) to [bend right] (e);
\draw [p] (ba) to [bend right] (a);
\end{scope}
\end{tikzpicture}

%% file: ecf60.tex
\begin{tikzpicture}[gnuplot,scale=.7]
\gpfill{color=\gprgb{1000}{1000}{1000}} (1.012,0.616)--(7.447,0.616)--(7.447,5.631)--(1.012,5.631)--cycle;
\gpcolor{gp lt color border}
\gpsetlinetype{gp lt border}
\gpsetlinewidth{1.00}
\draw[gp path] (1.012,0.616)--(1.012,5.631)--(7.447,5.631)--(7.447,0.616)--cycle;
\gpsetlinewidth{0.50}
\draw[gp path] (1.012,0.616)--(1.263,0.616);
\draw[gp path] (7.447,0.616)--(7.196,0.616);
\gpcolor{\gprgb{0}{0}{0}}
\node[gp node right,font=\gpfontsize{10.00pt}{12.00pt}] at (0.828,0.616) {0};
\gpcolor{gp lt color border}
\draw[gp path] (1.012,1.619)--(1.263,1.619);
\draw[gp path] (7.447,1.619)--(7.196,1.619);
\gpcolor{\gprgb{0}{0}{0}}
\node[gp node right,font=\gpfontsize{10.00pt}{12.00pt}] at (0.828,1.619) {0.2};
\gpcolor{gp lt color border}
\draw[gp path] (1.012,2.622)--(1.263,2.622);
\draw[gp path] (7.447,2.622)--(7.196,2.622);
\gpcolor{\gprgb{0}{0}{0}}
\node[gp node right,font=\gpfontsize{10.00pt}{12.00pt}] at (0.828,2.622) {0.4};
\gpcolor{gp lt color border}
\draw[gp path] (1.012,3.625)--(1.263,3.625);
\draw[gp path] (7.447,3.625)--(7.196,3.625);
\gpcolor{\gprgb{0}{0}{0}}
\node[gp node right,font=\gpfontsize{10.00pt}{12.00pt}] at (0.828,3.625) {0.6};
\gpcolor{gp lt color border}
\draw[gp path] (1.012,4.628)--(1.263,4.628);
\draw[gp path] (7.447,4.628)--(7.196,4.628);
\gpcolor{\gprgb{0}{0}{0}}
\node[gp node right,font=\gpfontsize{10.00pt}{12.00pt}] at (0.828,4.628) {0.8};
\gpcolor{gp lt color border}
\draw[gp path] (1.012,5.631)--(1.263,5.631);
\draw[gp path] (7.447,5.631)--(7.196,5.631);
\gpcolor{\gprgb{0}{0}{0}}
\node[gp node right,font=\gpfontsize{10.00pt}{12.00pt}] at (0.828,5.631) {1};
\gpcolor{gp lt color border}
\draw[gp path] (1.012,0.616)--(1.012,0.867);
\draw[gp path] (1.012,5.631)--(1.012,5.380);
\gpcolor{\gprgb{0}{0}{0}}
\node[gp node center,font=\gpfontsize{10.00pt}{12.00pt}] at (1.012,0.308) {-4};
\gpcolor{gp lt color border}
\draw[gp path] (1.816,0.616)--(1.816,0.867);
\draw[gp path] (1.816,5.631)--(1.816,5.380);
\gpcolor{\gprgb{0}{0}{0}}
\node[gp node center,font=\gpfontsize{10.00pt}{12.00pt}] at (1.816,0.308) {-3};
\gpcolor{gp lt color border}
\draw[gp path] (2.621,0.616)--(2.621,0.867);
\draw[gp path] (2.621,5.631)--(2.621,5.380);
\gpcolor{\gprgb{0}{0}{0}}
\node[gp node center,font=\gpfontsize{10.00pt}{12.00pt}] at (2.621,0.308) {-2};
\gpcolor{gp lt color border}
\draw[gp path] (3.425,0.616)--(3.425,0.867);
\draw[gp path] (3.425,5.631)--(3.425,5.380);
\gpcolor{\gprgb{0}{0}{0}}
\node[gp node center,font=\gpfontsize{10.00pt}{12.00pt}] at (3.425,0.308) {-1};
\gpcolor{gp lt color border}
\draw[gp path] (4.230,0.616)--(4.230,0.867);
\draw[gp path] (4.230,5.631)--(4.230,5.380);
\gpcolor{\gprgb{0}{0}{0}}
\node[gp node center,font=\gpfontsize{10.00pt}{12.00pt}] at (4.230,0.308) {0};
\gpcolor{gp lt color border}
\draw[gp path] (5.034,0.616)--(5.034,0.867);
\draw[gp path] (5.034,5.631)--(5.034,5.380);
\gpcolor{\gprgb{0}{0}{0}}
\node[gp node center,font=\gpfontsize{10.00pt}{12.00pt}] at (5.034,0.308) {1};
\gpcolor{gp lt color border}
\draw[gp path] (5.838,0.616)--(5.838,0.867);
\draw[gp path] (5.838,5.631)--(5.838,5.380);
\gpcolor{\gprgb{0}{0}{0}}
\node[gp node center,font=\gpfontsize{10.00pt}{12.00pt}] at (5.838,0.308) {2};
\gpcolor{gp lt color border}
\draw[gp path] (6.643,0.616)--(6.643,0.867);
\draw[gp path] (6.643,5.631)--(6.643,5.380);
\gpcolor{\gprgb{0}{0}{0}}
\node[gp node center,font=\gpfontsize{10.00pt}{12.00pt}] at (6.643,0.308) {3};
\gpcolor{gp lt color border}
\draw[gp path] (7.447,0.616)--(7.447,0.867);
\draw[gp path] (7.447,5.631)--(7.447,5.380);
\gpcolor{\gprgb{0}{0}{0}}
\node[gp node center,font=\gpfontsize{10.00pt}{12.00pt}] at (7.447,0.308) {4};
\gpcolor{gp lt color border}
\draw[gp path] (1.012,5.631)--(1.012,0.616)--(7.447,0.616)--(7.447,5.631)--cycle;
\gpcolor{\gprgb{0}{1000}{0}}
\gpsetlinetype{gp lt plot 0}
\draw[gp path] (1.443,0.616)--(1.449,0.789)--(1.454,0.857)--(1.460,0.908)--(1.465,0.949)%
  --(1.471,0.984)--(1.477,1.015)--(1.482,1.042)--(1.488,1.066)--(1.493,1.088)--(1.499,1.108)%
  --(1.504,1.126)--(1.510,1.143)--(1.516,1.159)--(1.521,1.173)--(1.527,1.187)--(1.532,1.200)%
  --(1.538,1.211)--(1.543,1.223)--(1.549,1.233)--(1.555,1.243)--(1.560,1.252)--(1.566,1.261)%
  --(1.571,1.269)--(1.577,1.277)--(1.582,1.284)--(1.588,1.291)--(1.594,1.297)--(1.599,1.304)%
  --(1.605,1.310)--(1.610,1.315)--(1.616,1.320)--(1.621,1.326)--(1.627,1.330)--(1.633,1.335)%
  --(1.638,1.339)--(1.644,1.343)--(1.649,1.347)--(1.655,1.351)--(1.660,1.354)--(1.666,1.358)%
  --(1.672,1.361)--(1.677,1.364)--(1.683,1.367)--(1.688,1.370)--(1.694,1.372)--(1.699,1.375)%
  --(1.705,1.377)--(1.711,1.379)--(1.716,1.382)--(1.722,1.384)--(1.727,1.386)--(1.733,1.387)%
  --(1.738,1.389)--(1.744,1.391)--(1.750,1.392)--(1.755,1.394)--(1.761,1.395)--(1.766,1.397)%
  --(1.772,1.398)--(1.777,1.399)--(1.783,1.400)--(1.789,1.401)--(1.794,1.402)--(1.800,1.403)%
  --(1.805,1.404)--(1.811,1.405)--(1.816,1.406)--(1.822,1.407)--(1.828,1.407)--(1.833,1.408)%
  --(1.839,1.409)--(1.844,1.409)--(1.850,1.410)--(1.855,1.410)--(1.861,1.411)--(1.867,1.411)%
  --(1.872,1.412)--(1.878,1.412)--(1.883,1.412)--(1.889,1.413)--(1.894,1.413)--(1.900,1.413)%
  --(1.906,1.413)--(1.911,1.414)--(1.917,1.414)--(1.922,1.414)--(1.928,1.414)--(1.933,1.414)%
  --(1.939,1.414)--(1.945,1.414)--(1.950,1.414)--(1.956,1.414)--(1.961,1.414)--(1.967,1.414)%
  --(1.972,1.414)--(1.978,1.414)--(1.984,1.414)--(1.989,1.414)--(1.995,1.414)--(2.000,1.414)%
  --(2.006,1.413)--(2.011,1.413)--(2.017,1.413)--(2.023,1.413)--(2.028,1.413)--(2.034,1.413)%
  --(2.039,1.412)--(2.045,1.412)--(2.051,1.412)--(2.056,1.412)--(2.062,1.411)--(2.067,1.411)%
  --(2.073,1.411)--(2.078,1.410)--(2.084,1.410)--(2.090,1.410)--(2.095,1.410)--(2.101,1.409)%
  --(2.106,1.409)--(2.112,1.409)--(2.117,1.408)--(2.123,1.408)--(2.129,1.408)--(2.134,1.407)%
  --(2.140,1.407)--(2.145,1.406)--(2.151,1.406)--(2.156,1.406)--(2.162,1.405)--(2.168,1.405)%
  --(2.173,1.405)--(2.179,1.404)--(2.184,1.404)--(2.190,1.403)--(2.195,1.403)--(2.201,1.403)%
  --(2.207,1.402)--(2.212,1.402)--(2.218,1.401)--(2.223,1.401)--(2.229,1.400)--(2.234,1.400)%
  --(2.240,1.400)--(2.246,1.399)--(2.251,1.399)--(2.257,1.398)--(2.262,1.398)--(2.268,1.397)%
  --(2.273,1.397)--(2.279,1.396)--(2.285,1.396)--(2.290,1.396)--(2.296,1.395)--(2.301,1.395)%
  --(2.307,1.394)--(2.312,1.394)--(2.318,1.393)--(2.324,1.393)--(2.329,1.392)--(2.335,1.392)%
  --(2.340,1.391)--(2.346,1.391)--(2.351,1.391)--(2.357,1.390)--(2.363,1.390)--(2.368,1.389)%
  --(2.374,1.389)--(2.379,1.388)--(2.385,1.388)--(2.390,1.387)--(2.396,1.387)--(2.402,1.386)%
  --(2.407,1.386)--(2.413,1.385)--(2.418,1.385)--(2.424,1.384)--(2.429,1.384)--(2.435,1.384)%
  --(2.441,1.383)--(2.446,1.383)--(2.452,1.382)--(2.457,1.382)--(2.463,1.381)--(2.468,1.381)%
  --(2.474,1.380)--(2.480,1.380)--(2.485,1.379)--(2.491,1.379)--(2.496,1.378)--(2.502,1.378)%
  --(2.507,1.378)--(2.513,1.377)--(2.519,1.377)--(2.524,1.376)--(2.530,1.376)--(2.535,1.375)%
  --(2.541,1.375)--(2.546,1.374)--(2.552,1.374)--(2.558,1.374)--(2.563,1.373)--(2.569,1.373)%
  --(2.574,1.372)--(2.580,1.372)--(2.586,1.371)--(2.591,1.371)--(2.597,1.370)--(2.602,1.370)%
  --(2.608,1.370)--(2.613,1.369)--(2.619,1.369)--(2.625,1.368)--(2.630,1.368)--(2.636,1.367)%
  --(2.641,1.367)--(2.647,1.366)--(2.652,1.366)--(2.658,1.366)--(2.664,1.365)--(2.669,1.365)%
  --(2.675,1.364)--(2.680,1.364)--(2.686,1.364)--(2.691,1.363)--(2.697,1.363)--(2.703,1.362)%
  --(2.708,1.362)--(2.714,1.361)--(2.719,1.361)--(2.725,1.361)--(2.730,1.360)--(2.736,1.360)%
  --(2.742,1.359)--(2.747,1.359)--(2.753,1.359)--(2.758,1.358)--(2.764,1.358)--(2.769,1.357)%
  --(2.775,1.357)--(2.781,1.357)--(2.786,1.356)--(2.792,1.356)--(2.797,1.355)--(2.803,1.355)%
  --(2.808,1.355)--(2.814,1.354)--(2.820,1.354)--(2.825,1.354)--(2.831,1.353)--(2.836,1.353)%
  --(2.842,1.352)--(2.847,1.352)--(2.853,1.352)--(2.859,1.351)--(2.864,1.351)--(2.870,1.351)%
  --(2.875,1.350)--(2.881,1.350)--(2.886,1.349)--(2.892,1.349)--(2.898,1.349)--(2.903,1.348)%
  --(2.909,1.348)--(2.914,1.348)--(2.920,1.347)--(2.925,1.347)--(2.931,1.347)--(2.937,1.346)%
  --(2.942,1.346)--(2.948,1.346)--(2.953,1.345)--(2.959,1.345)--(2.964,1.345)--(2.970,1.344)%
  --(2.976,1.344)--(2.981,1.343)--(2.987,1.343)--(2.992,1.343)--(2.998,1.342)--(3.003,1.342)%
  --(3.009,1.342)--(3.015,1.342)--(3.020,1.341)--(3.026,1.341)--(3.031,1.341)--(3.037,1.340)%
  --(3.042,1.340)--(3.048,1.340)--(3.054,1.339)--(3.059,1.339)--(3.065,1.339)--(3.070,1.338)%
  --(3.076,1.338)--(3.081,1.338)--(3.087,1.337)--(3.093,1.337)--(3.098,1.337)--(3.104,1.337)%
  --(3.109,1.336)--(3.115,1.336)--(3.120,1.336)--(3.126,1.335)--(3.132,1.335)--(3.137,1.335)%
  --(3.143,1.334)--(3.148,1.334)--(3.154,1.334)--(3.160,1.334)--(3.165,1.333)--(3.171,1.333)%
  --(3.176,1.333)--(3.182,1.332)--(3.187,1.332)--(3.193,1.332)--(3.199,1.332)--(3.204,1.331)%
  --(3.210,1.331)--(3.215,1.331)--(3.221,1.331)--(3.226,1.330)--(3.232,1.330)--(3.238,1.330)%
  --(3.243,1.330)--(3.249,1.329)--(3.254,1.329)--(3.260,1.329)--(3.265,1.329)--(3.271,1.328)%
  --(3.277,1.328)--(3.282,1.328)--(3.288,1.328)--(3.293,1.327)--(3.299,1.327)--(3.304,1.327)%
  --(3.310,1.327)--(3.316,1.326)--(3.321,1.326)--(3.327,1.326)--(3.332,1.326)--(3.338,1.325)%
  --(3.343,1.325)--(3.349,1.325)--(3.355,1.325)--(3.360,1.324)--(3.366,1.324)--(3.371,1.324)%
  --(3.377,1.324)--(3.382,1.324)--(3.388,1.323)--(3.394,1.323)--(3.399,1.323)--(3.405,1.323)%
  --(3.410,1.322)--(3.416,1.322)--(3.421,1.322)--(3.427,1.322)--(3.433,1.322)--(3.438,1.321)%
  --(3.444,1.321)--(3.449,1.321)--(3.455,1.321)--(3.460,1.321)--(3.466,1.320)--(3.472,1.320)%
  --(3.477,1.320)--(3.483,1.320)--(3.488,1.320)--(3.494,1.319)--(3.499,1.319)--(3.505,1.319)%
  --(3.511,1.319)--(3.516,1.319)--(3.522,1.319)--(3.527,1.318)--(3.533,1.318)--(3.538,1.318)%
  --(3.544,1.318)--(3.550,1.318)--(3.555,1.317)--(3.561,1.317)--(3.566,1.317)--(3.572,1.317)%
  --(3.577,1.317)--(3.583,1.317)--(3.589,1.316)--(3.594,1.316)--(3.600,1.316)--(3.605,1.316)%
  --(3.611,1.316)--(3.616,1.316)--(3.622,1.316)--(3.628,1.315)--(3.633,1.315)--(3.639,1.315)%
  --(3.644,1.315)--(3.650,1.315)--(3.655,1.315)--(3.661,1.314)--(3.667,1.314)--(3.672,1.314)%
  --(3.678,1.314)--(3.683,1.314)--(3.689,1.314)--(3.695,1.314)--(3.700,1.314)--(3.706,1.313)%
  --(3.711,1.313)--(3.717,1.313)--(3.722,1.313)--(3.728,1.313)--(3.734,1.313)--(3.739,1.313)%
  --(3.745,1.312)--(3.750,1.312)--(3.756,1.312)--(3.761,1.312)--(3.767,1.312)--(3.773,1.312)%
  --(3.778,1.312)--(3.784,1.312)--(3.789,1.312)--(3.795,1.311)--(3.800,1.311)--(3.806,1.311)%
  --(3.812,1.311)--(3.817,1.311)--(3.823,1.311)--(3.828,1.311)--(3.834,1.311)--(3.839,1.311)%
  --(3.845,1.311)--(3.851,1.310)--(3.856,1.310)--(3.862,1.310)--(3.867,1.310)--(3.873,1.310)%
  --(3.878,1.310)--(3.884,1.310)--(3.890,1.310)--(3.895,1.310)--(3.901,1.310)--(3.906,1.310)%
  --(3.912,1.309)--(3.917,1.309)--(3.923,1.309)--(3.929,1.309)--(3.934,1.309)--(3.940,1.309)%
  --(3.945,1.309)--(3.951,1.309)--(3.956,1.309)--(3.962,1.309)--(3.968,1.309)--(3.973,1.309)%
  --(3.979,1.309)--(3.984,1.309)--(3.990,1.309)--(3.995,1.308)--(4.001,1.308)--(4.007,1.308)%
  --(4.012,1.308)--(4.018,1.308)--(4.023,1.308)--(4.029,1.308)--(4.034,1.308)--(4.040,1.308)%
  --(4.046,1.308)--(4.051,1.308)--(4.057,1.308)--(4.062,1.308)--(4.068,1.308)--(4.073,1.308)%
  --(4.079,1.308)--(4.085,1.308)--(4.090,1.308)--(4.096,1.308)--(4.101,1.308)--(4.107,1.308)%
  --(4.112,1.308)--(4.118,1.308)--(4.124,1.307)--(4.129,1.307)--(4.135,1.307)--(4.140,1.307)%
  --(4.146,1.307)--(4.151,1.307)--(4.157,1.307)--(4.163,1.307)--(4.168,1.307)--(4.174,1.307)%
  --(4.179,1.307)--(4.185,1.307)--(4.190,1.307)--(4.196,1.307)--(4.202,1.307)--(4.207,1.307)%
  --(4.213,1.307)--(4.218,1.307)--(4.224,1.307)--(4.230,1.307)--(4.235,1.307)--(4.241,1.307)%
  --(4.246,1.307)--(4.252,1.307)--(4.257,1.307)--(4.263,1.307)--(4.269,1.307)--(4.274,1.307)%
  --(4.280,1.307)--(4.285,1.307)--(4.291,1.307)--(4.296,1.307)--(4.302,1.307)--(4.308,1.307)%
  --(4.313,1.307)--(4.319,1.307)--(4.324,1.307)--(4.330,1.307)--(4.335,1.307)--(4.341,1.308)%
  --(4.347,1.308)--(4.352,1.308)--(4.358,1.308)--(4.363,1.308)--(4.369,1.308)--(4.374,1.308)%
  --(4.380,1.308)--(4.386,1.308)--(4.391,1.308)--(4.397,1.308)--(4.402,1.308)--(4.408,1.308)%
  --(4.413,1.308)--(4.419,1.308)--(4.425,1.308)--(4.430,1.308)--(4.436,1.308)--(4.441,1.308)%
  --(4.447,1.308)--(4.452,1.308)--(4.458,1.308)--(4.464,1.308)--(4.469,1.309)--(4.475,1.309)%
  --(4.480,1.309)--(4.486,1.309)--(4.491,1.309)--(4.497,1.309)--(4.503,1.309)--(4.508,1.309)%
  --(4.514,1.309)--(4.519,1.309)--(4.525,1.309)--(4.530,1.309)--(4.536,1.309)--(4.542,1.309)%
  --(4.547,1.309)--(4.553,1.310)--(4.558,1.310)--(4.564,1.310)--(4.569,1.310)--(4.575,1.310)%
  --(4.581,1.310)--(4.586,1.310)--(4.592,1.310)--(4.597,1.310)--(4.603,1.310)--(4.608,1.310)%
  --(4.614,1.311)--(4.620,1.311)--(4.625,1.311)--(4.631,1.311)--(4.636,1.311)--(4.642,1.311)%
  --(4.647,1.311)--(4.653,1.311)--(4.659,1.311)--(4.664,1.311)--(4.670,1.312)--(4.675,1.312)%
  --(4.681,1.312)--(4.686,1.312)--(4.692,1.312)--(4.698,1.312)--(4.703,1.312)--(4.709,1.312)%
  --(4.714,1.312)--(4.720,1.313)--(4.725,1.313)--(4.731,1.313)--(4.737,1.313)--(4.742,1.313)%
  --(4.748,1.313)--(4.753,1.313)--(4.759,1.314)--(4.764,1.314)--(4.770,1.314)--(4.776,1.314)%
  --(4.781,1.314)--(4.787,1.314)--(4.792,1.314)--(4.798,1.314)--(4.804,1.315)--(4.809,1.315)%
  --(4.815,1.315)--(4.820,1.315)--(4.826,1.315)--(4.831,1.315)--(4.837,1.316)--(4.843,1.316)%
  --(4.848,1.316)--(4.854,1.316)--(4.859,1.316)--(4.865,1.316)--(4.870,1.316)--(4.876,1.317)%
  --(4.882,1.317)--(4.887,1.317)--(4.893,1.317)--(4.898,1.317)--(4.904,1.317)--(4.909,1.318)%
  --(4.915,1.318)--(4.921,1.318)--(4.926,1.318)--(4.932,1.318)--(4.937,1.319)--(4.943,1.319)%
  --(4.948,1.319)--(4.954,1.319)--(4.960,1.319)--(4.965,1.319)--(4.971,1.320)--(4.976,1.320)%
  --(4.982,1.320)--(4.987,1.320)--(4.993,1.320)--(4.999,1.321)--(5.004,1.321)--(5.010,1.321)%
  --(5.015,1.321)--(5.021,1.321)--(5.026,1.322)--(5.032,1.322)--(5.038,1.322)--(5.043,1.322)%
  --(5.049,1.322)--(5.054,1.323)--(5.060,1.323)--(5.065,1.323)--(5.071,1.323)--(5.077,1.324)%
  --(5.082,1.324)--(5.088,1.324)--(5.093,1.324)--(5.099,1.324)--(5.104,1.325)--(5.110,1.325)%
  --(5.116,1.325)--(5.121,1.325)--(5.127,1.326)--(5.132,1.326)--(5.138,1.326)--(5.143,1.326)%
  --(5.149,1.327)--(5.155,1.327)--(5.160,1.327)--(5.166,1.327)--(5.171,1.328)--(5.177,1.328)%
  --(5.182,1.328)--(5.188,1.328)--(5.194,1.329)--(5.199,1.329)--(5.205,1.329)--(5.210,1.329)%
  --(5.216,1.330)--(5.221,1.330)--(5.227,1.330)--(5.233,1.330)--(5.238,1.331)--(5.244,1.331)%
  --(5.249,1.331)--(5.255,1.331)--(5.260,1.332)--(5.266,1.332)--(5.272,1.332)--(5.277,1.332)%
  --(5.283,1.333)--(5.288,1.333)--(5.294,1.333)--(5.299,1.334)--(5.305,1.334)--(5.311,1.334)%
  --(5.316,1.334)--(5.322,1.335)--(5.327,1.335)--(5.333,1.335)--(5.339,1.336)--(5.344,1.336)%
  --(5.350,1.336)--(5.355,1.337)--(5.361,1.337)--(5.366,1.337)--(5.372,1.337)--(5.378,1.338)%
  --(5.383,1.338)--(5.389,1.338)--(5.394,1.339)--(5.400,1.339)--(5.405,1.339)--(5.411,1.340)%
  --(5.417,1.340)--(5.422,1.340)--(5.428,1.341)--(5.433,1.341)--(5.439,1.341)--(5.444,1.342)%
  --(5.450,1.342)--(5.456,1.342)--(5.461,1.342)--(5.467,1.343)--(5.472,1.343)--(5.478,1.343)%
  --(5.483,1.344)--(5.489,1.344)--(5.495,1.345)--(5.500,1.345)--(5.506,1.345)--(5.511,1.346)%
  --(5.517,1.346)--(5.522,1.346)--(5.528,1.347)--(5.534,1.347)--(5.539,1.347)--(5.545,1.348)%
  --(5.550,1.348)--(5.556,1.348)--(5.561,1.349)--(5.567,1.349)--(5.573,1.349)--(5.578,1.350)%
  --(5.584,1.350)--(5.589,1.351)--(5.595,1.351)--(5.600,1.351)--(5.606,1.352)--(5.612,1.352)%
  --(5.617,1.352)--(5.623,1.353)--(5.628,1.353)--(5.634,1.354)--(5.639,1.354)--(5.645,1.354)%
  --(5.651,1.355)--(5.656,1.355)--(5.662,1.355)--(5.667,1.356)--(5.673,1.356)--(5.678,1.357)%
  --(5.684,1.357)--(5.690,1.357)--(5.695,1.358)--(5.701,1.358)--(5.706,1.359)--(5.712,1.359)%
  --(5.717,1.359)--(5.723,1.360)--(5.729,1.360)--(5.734,1.361)--(5.740,1.361)--(5.745,1.361)%
  --(5.751,1.362)--(5.756,1.362)--(5.762,1.363)--(5.768,1.363)--(5.773,1.364)--(5.779,1.364)%
  --(5.784,1.364)--(5.790,1.365)--(5.795,1.365)--(5.801,1.366)--(5.807,1.366)--(5.812,1.366)%
  --(5.818,1.367)--(5.823,1.367)--(5.829,1.368)--(5.834,1.368)--(5.840,1.369)--(5.846,1.369)%
  --(5.851,1.370)--(5.857,1.370)--(5.862,1.370)--(5.868,1.371)--(5.873,1.371)--(5.879,1.372)%
  --(5.885,1.372)--(5.890,1.373)--(5.896,1.373)--(5.901,1.374)--(5.907,1.374)--(5.913,1.374)%
  --(5.918,1.375)--(5.924,1.375)--(5.929,1.376)--(5.935,1.376)--(5.940,1.377)--(5.946,1.377)%
  --(5.952,1.378)--(5.957,1.378)--(5.963,1.378)--(5.968,1.379)--(5.974,1.379)--(5.979,1.380)%
  --(5.985,1.380)--(5.991,1.381)--(5.996,1.381)--(6.002,1.382)--(6.007,1.382)--(6.013,1.383)%
  --(6.018,1.383)--(6.024,1.384)--(6.030,1.384)--(6.035,1.384)--(6.041,1.385)--(6.046,1.385)%
  --(6.052,1.386)--(6.057,1.386)--(6.063,1.387)--(6.069,1.387)--(6.074,1.388)--(6.080,1.388)%
  --(6.085,1.389)--(6.091,1.389)--(6.096,1.390)--(6.102,1.390)--(6.108,1.391)--(6.113,1.391)%
  --(6.119,1.391)--(6.124,1.392)--(6.130,1.392)--(6.135,1.393)--(6.141,1.393)--(6.147,1.394)%
  --(6.152,1.394)--(6.158,1.395)--(6.163,1.395)--(6.169,1.396)--(6.174,1.396)--(6.180,1.396)%
  --(6.186,1.397)--(6.191,1.397)--(6.197,1.398)--(6.202,1.398)--(6.208,1.399)--(6.213,1.399)%
  --(6.219,1.400)--(6.225,1.400)--(6.230,1.400)--(6.236,1.401)--(6.241,1.401)--(6.247,1.402)%
  --(6.252,1.402)--(6.258,1.403)--(6.264,1.403)--(6.269,1.403)--(6.275,1.404)--(6.280,1.404)%
  --(6.286,1.405)--(6.291,1.405)--(6.297,1.405)--(6.303,1.406)--(6.308,1.406)--(6.314,1.406)%
  --(6.319,1.407)--(6.325,1.407)--(6.330,1.408)--(6.336,1.408)--(6.342,1.408)--(6.347,1.409)%
  --(6.353,1.409)--(6.358,1.409)--(6.364,1.410)--(6.369,1.410)--(6.375,1.410)--(6.381,1.410)%
  --(6.386,1.411)--(6.392,1.411)--(6.397,1.411)--(6.403,1.412)--(6.408,1.412)--(6.414,1.412)%
  --(6.420,1.412)--(6.425,1.413)--(6.431,1.413)--(6.436,1.413)--(6.442,1.413)--(6.448,1.413)%
  --(6.453,1.413)--(6.459,1.414)--(6.464,1.414)--(6.470,1.414)--(6.475,1.414)--(6.481,1.414)%
  --(6.487,1.414)--(6.492,1.414)--(6.498,1.414)--(6.503,1.414)--(6.509,1.414)--(6.514,1.414)%
  --(6.520,1.414)--(6.526,1.414)--(6.531,1.414)--(6.537,1.414)--(6.542,1.414)--(6.548,1.414)%
  --(6.553,1.413)--(6.559,1.413)--(6.565,1.413)--(6.570,1.413)--(6.576,1.412)--(6.581,1.412)%
  --(6.587,1.412)--(6.592,1.411)--(6.598,1.411)--(6.604,1.410)--(6.609,1.410)--(6.615,1.409)%
  --(6.620,1.409)--(6.626,1.408)--(6.631,1.407)--(6.637,1.407)--(6.643,1.406)--(6.648,1.405)%
  --(6.654,1.404)--(6.659,1.403)--(6.665,1.402)--(6.670,1.401)--(6.676,1.400)--(6.682,1.399)%
  --(6.687,1.398)--(6.693,1.397)--(6.698,1.395)--(6.704,1.394)--(6.709,1.392)--(6.715,1.391)%
  --(6.721,1.389)--(6.726,1.387)--(6.732,1.386)--(6.737,1.384)--(6.743,1.382)--(6.748,1.379)%
  --(6.754,1.377)--(6.760,1.375)--(6.765,1.372)--(6.771,1.370)--(6.776,1.367)--(6.782,1.364)%
  --(6.787,1.361)--(6.793,1.358)--(6.799,1.354)--(6.804,1.351)--(6.810,1.347)--(6.815,1.343)%
  --(6.821,1.339)--(6.826,1.335)--(6.832,1.330)--(6.838,1.326)--(6.843,1.320)--(6.849,1.315)%
  --(6.854,1.310)--(6.860,1.304)--(6.865,1.297)--(6.871,1.291)--(6.877,1.284)--(6.882,1.277)%
  --(6.888,1.269)--(6.893,1.261)--(6.899,1.252)--(6.904,1.243)--(6.910,1.233)--(6.916,1.223)%
  --(6.921,1.211)--(6.927,1.200)--(6.932,1.187)--(6.938,1.173)--(6.943,1.159)--(6.949,1.143)%
  --(6.955,1.126)--(6.960,1.108)--(6.966,1.088)--(6.971,1.066)--(6.977,1.042)--(6.982,1.015)%
  --(6.988,0.984)--(6.994,0.949)--(6.999,0.908)--(7.005,0.857)--(7.010,0.789)--(7.016,0.616);
\gpcolor{\gprgb{1000}{0}{0}}
\draw[gp path] (1.443,0.616)--(1.449,0.617)--(1.454,0.619)--(1.460,0.621)--(1.465,0.623)%
  --(1.471,0.626)--(1.477,0.629)--(1.482,0.631)--(1.488,0.635)--(1.493,0.638)--(1.499,0.641)%
  --(1.504,0.645)--(1.510,0.648)--(1.516,0.652)--(1.521,0.656)--(1.527,0.660)--(1.532,0.664)%
  --(1.538,0.668)--(1.543,0.672)--(1.549,0.677)--(1.555,0.681)--(1.560,0.685)--(1.566,0.690)%
  --(1.571,0.694)--(1.577,0.699)--(1.582,0.704)--(1.588,0.708)--(1.594,0.713)--(1.599,0.718)%
  --(1.605,0.723)--(1.610,0.727)--(1.616,0.732)--(1.621,0.737)--(1.627,0.742)--(1.633,0.747)%
  --(1.638,0.752)--(1.644,0.757)--(1.649,0.762)--(1.655,0.767)--(1.660,0.772)--(1.666,0.778)%
  --(1.672,0.783)--(1.677,0.788)--(1.683,0.793)--(1.688,0.798)--(1.694,0.804)--(1.699,0.809)%
  --(1.705,0.814)--(1.711,0.819)--(1.716,0.825)--(1.722,0.830)--(1.727,0.835)--(1.733,0.841)%
  --(1.738,0.846)--(1.744,0.851)--(1.750,0.857)--(1.755,0.862)--(1.761,0.868)--(1.766,0.873)%
  --(1.772,0.878)--(1.777,0.884)--(1.783,0.889)--(1.789,0.895)--(1.794,0.900)--(1.800,0.906)%
  --(1.805,0.911)--(1.811,0.917)--(1.816,0.922)--(1.822,0.928)--(1.828,0.933)--(1.833,0.939)%
  --(1.839,0.944)--(1.844,0.949)--(1.850,0.955)--(1.855,0.960)--(1.861,0.966)--(1.867,0.972)%
  --(1.872,0.977)--(1.878,0.983)--(1.883,0.988)--(1.889,0.994)--(1.894,0.999)--(1.900,1.005)%
  --(1.906,1.010)--(1.911,1.016)--(1.917,1.021)--(1.922,1.027)--(1.928,1.032)--(1.933,1.038)%
  --(1.939,1.043)--(1.945,1.049)--(1.950,1.054)--(1.956,1.060)--(1.961,1.065)--(1.967,1.071)%
  --(1.972,1.076)--(1.978,1.082)--(1.984,1.088)--(1.989,1.093)--(1.995,1.099)--(2.000,1.104)%
  --(2.006,1.110)--(2.011,1.115)--(2.017,1.121)--(2.023,1.126)--(2.028,1.132)--(2.034,1.137)%
  --(2.039,1.143)--(2.045,1.148)--(2.051,1.154)--(2.056,1.159)--(2.062,1.165)--(2.067,1.170)%
  --(2.073,1.176)--(2.078,1.181)--(2.084,1.187)--(2.090,1.192)--(2.095,1.198)--(2.101,1.203)%
  --(2.106,1.209)--(2.112,1.214)--(2.117,1.220)--(2.123,1.225)--(2.129,1.231)--(2.134,1.236)%
  --(2.140,1.242)--(2.145,1.247)--(2.151,1.253)--(2.156,1.258)--(2.162,1.264)--(2.168,1.269)%
  --(2.173,1.275)--(2.179,1.280)--(2.184,1.285)--(2.190,1.291)--(2.195,1.296)--(2.201,1.302)%
  --(2.207,1.307)--(2.212,1.313)--(2.218,1.318)--(2.223,1.324)--(2.229,1.329)--(2.234,1.334)%
  --(2.240,1.340)--(2.246,1.345)--(2.251,1.351)--(2.257,1.356)--(2.262,1.362)--(2.268,1.367)%
  --(2.273,1.372)--(2.279,1.378)--(2.285,1.383)--(2.290,1.389)--(2.296,1.394)--(2.301,1.399)%
  --(2.307,1.405)--(2.312,1.410)--(2.318,1.416)--(2.324,1.421)--(2.329,1.426)--(2.335,1.432)%
  --(2.340,1.437)--(2.346,1.442)--(2.351,1.448)--(2.357,1.453)--(2.363,1.459)--(2.368,1.464)%
  --(2.374,1.469)--(2.379,1.475)--(2.385,1.480)--(2.390,1.485)--(2.396,1.491)--(2.402,1.496)%
  --(2.407,1.501)--(2.413,1.507)--(2.418,1.512)--(2.424,1.517)--(2.429,1.523)--(2.435,1.528)%
  --(2.441,1.533)--(2.446,1.539)--(2.452,1.544)--(2.457,1.549)--(2.463,1.554)--(2.468,1.560)%
  --(2.474,1.565)--(2.480,1.570)--(2.485,1.576)--(2.491,1.581)--(2.496,1.586)--(2.502,1.591)%
  --(2.507,1.597)--(2.513,1.602)--(2.519,1.607)--(2.524,1.613)--(2.530,1.618)--(2.535,1.623)%
  --(2.541,1.628)--(2.546,1.634)--(2.552,1.639)--(2.558,1.644)--(2.563,1.649)--(2.569,1.655)%
  --(2.574,1.660)--(2.580,1.665)--(2.586,1.670)--(2.591,1.676)--(2.597,1.681)--(2.602,1.686)%
  --(2.608,1.691)--(2.613,1.696)--(2.619,1.702)--(2.625,1.707)--(2.630,1.712)--(2.636,1.717)%
  --(2.641,1.722)--(2.647,1.728)--(2.652,1.733)--(2.658,1.738)--(2.664,1.743)--(2.669,1.748)%
  --(2.675,1.754)--(2.680,1.759)--(2.686,1.764)--(2.691,1.769)--(2.697,1.774)--(2.703,1.780)%
  --(2.708,1.785)--(2.714,1.790)--(2.719,1.795)--(2.725,1.800)--(2.730,1.805)--(2.736,1.810)%
  --(2.742,1.816)--(2.747,1.821)--(2.753,1.826)--(2.758,1.831)--(2.764,1.836)--(2.769,1.841)%
  --(2.775,1.846)--(2.781,1.852)--(2.786,1.857)--(2.792,1.862)--(2.797,1.867)--(2.803,1.872)%
  --(2.808,1.877)--(2.814,1.882)--(2.820,1.887)--(2.825,1.893)--(2.831,1.898)--(2.836,1.903)%
  --(2.842,1.908)--(2.847,1.913)--(2.853,1.918)--(2.859,1.923)--(2.864,1.928)--(2.870,1.933)%
  --(2.875,1.938)--(2.881,1.943)--(2.886,1.949)--(2.892,1.954)--(2.898,1.959)--(2.903,1.964)%
  --(2.909,1.969)--(2.914,1.974)--(2.920,1.979)--(2.925,1.984)--(2.931,1.989)--(2.937,1.994)%
  --(2.942,1.999)--(2.948,2.004)--(2.953,2.009)--(2.959,2.014)--(2.964,2.019)--(2.970,2.025)%
  --(2.976,2.030)--(2.981,2.035)--(2.987,2.040)--(2.992,2.045)--(2.998,2.050)--(3.003,2.055)%
  --(3.009,2.060)--(3.015,2.065)--(3.020,2.070)--(3.026,2.075)--(3.031,2.080)--(3.037,2.085)%
  --(3.042,2.090)--(3.048,2.095)--(3.054,2.100)--(3.059,2.105)--(3.065,2.110)--(3.070,2.115)%
  --(3.076,2.120)--(3.081,2.125)--(3.087,2.130)--(3.093,2.135)--(3.098,2.140)--(3.104,2.145)%
  --(3.109,2.150)--(3.115,2.155)--(3.120,2.160)--(3.126,2.165)--(3.132,2.170)--(3.137,2.175)%
  --(3.143,2.180)--(3.148,2.185)--(3.154,2.190)--(3.160,2.195)--(3.165,2.200)--(3.171,2.205)%
  --(3.176,2.210)--(3.182,2.215)--(3.187,2.220)--(3.193,2.225)--(3.199,2.229)--(3.204,2.234)%
  --(3.210,2.239)--(3.215,2.244)--(3.221,2.249)--(3.226,2.254)--(3.232,2.259)--(3.238,2.264)%
  --(3.243,2.269)--(3.249,2.274)--(3.254,2.279)--(3.260,2.284)--(3.265,2.289)--(3.271,2.294)%
  --(3.277,2.299)--(3.282,2.304)--(3.288,2.309)--(3.293,2.313)--(3.299,2.318)--(3.304,2.323)%
  --(3.310,2.328)--(3.316,2.333)--(3.321,2.338)--(3.327,2.343)--(3.332,2.348)--(3.338,2.353)%
  --(3.343,2.358)--(3.349,2.363)--(3.355,2.368)--(3.360,2.372)--(3.366,2.377)--(3.371,2.382)%
  --(3.377,2.387)--(3.382,2.392)--(3.388,2.397)--(3.394,2.402)--(3.399,2.407)--(3.405,2.412)%
  --(3.410,2.417)--(3.416,2.421)--(3.421,2.426)--(3.427,2.431)--(3.433,2.436)--(3.438,2.441)%
  --(3.444,2.446)--(3.449,2.451)--(3.455,2.456)--(3.460,2.461)--(3.466,2.465)--(3.472,2.470)%
  --(3.477,2.475)--(3.483,2.480)--(3.488,2.485)--(3.494,2.490)--(3.499,2.495)--(3.505,2.500)%
  --(3.511,2.504)--(3.516,2.509)--(3.522,2.514)--(3.527,2.519)--(3.533,2.524)--(3.538,2.529)%
  --(3.544,2.534)--(3.550,2.539)--(3.555,2.543)--(3.561,2.548)--(3.566,2.553)--(3.572,2.558)%
  --(3.577,2.563)--(3.583,2.568)--(3.589,2.572)--(3.594,2.577)--(3.600,2.582)--(3.605,2.587)%
  --(3.611,2.592)--(3.616,2.597)--(3.622,2.602)--(3.628,2.606)--(3.633,2.611)--(3.639,2.616)%
  --(3.644,2.621)--(3.650,2.626)--(3.655,2.631)--(3.661,2.635)--(3.667,2.640)--(3.672,2.645)%
  --(3.678,2.650)--(3.683,2.655)--(3.689,2.660)--(3.695,2.665)--(3.700,2.669)--(3.706,2.674)%
  --(3.711,2.679)--(3.717,2.684)--(3.722,2.689)--(3.728,2.693)--(3.734,2.698)--(3.739,2.703)%
  --(3.745,2.708)--(3.750,2.713)--(3.756,2.718)--(3.761,2.722)--(3.767,2.727)--(3.773,2.732)%
  --(3.778,2.737)--(3.784,2.742)--(3.789,2.747)--(3.795,2.751)--(3.800,2.756)--(3.806,2.761)%
  --(3.812,2.766)--(3.817,2.771)--(3.823,2.775)--(3.828,2.780)--(3.834,2.785)--(3.839,2.790)%
  --(3.845,2.795)--(3.851,2.800)--(3.856,2.804)--(3.862,2.809)--(3.867,2.814)--(3.873,2.819)%
  --(3.878,2.824)--(3.884,2.828)--(3.890,2.833)--(3.895,2.838)--(3.901,2.843)--(3.906,2.848)%
  --(3.912,2.852)--(3.917,2.857)--(3.923,2.862)--(3.929,2.867)--(3.934,2.872)--(3.940,2.876)%
  --(3.945,2.881)--(3.951,2.886)--(3.956,2.891)--(3.962,2.896)--(3.968,2.900)--(3.973,2.905)%
  --(3.979,2.910)--(3.984,2.915)--(3.990,2.920)--(3.995,2.924)--(4.001,2.929)--(4.007,2.934)%
  --(4.012,2.939)--(4.018,2.944)--(4.023,2.948)--(4.029,2.953)--(4.034,2.958)--(4.040,2.963)%
  --(4.046,2.968)--(4.051,2.972)--(4.057,2.977)--(4.062,2.982)--(4.068,2.987)--(4.073,2.992)%
  --(4.079,2.996)--(4.085,3.001)--(4.090,3.006)--(4.096,3.011)--(4.101,3.015)--(4.107,3.020)%
  --(4.112,3.025)--(4.118,3.030)--(4.124,3.035)--(4.129,3.039)--(4.135,3.044)--(4.140,3.049)%
  --(4.146,3.054)--(4.151,3.059)--(4.157,3.063)--(4.163,3.068)--(4.168,3.073)--(4.174,3.078)%
  --(4.179,3.083)--(4.185,3.087)--(4.190,3.092)--(4.196,3.097)--(4.202,3.102)--(4.207,3.106)%
  --(4.213,3.111)--(4.218,3.116)--(4.224,3.121)--(4.230,3.126)--(4.235,3.130)--(4.241,3.135)%
  --(4.246,3.140)--(4.252,3.145)--(4.257,3.150)--(4.263,3.154)--(4.269,3.159)--(4.274,3.164)%
  --(4.280,3.169)--(4.285,3.174)--(4.291,3.178)--(4.296,3.183)--(4.302,3.188)--(4.308,3.193)%
  --(4.313,3.197)--(4.319,3.202)--(4.324,3.207)--(4.330,3.212)--(4.335,3.217)--(4.341,3.221)%
  --(4.347,3.226)--(4.352,3.231)--(4.358,3.236)--(4.363,3.241)--(4.369,3.245)--(4.374,3.250)%
  --(4.380,3.255)--(4.386,3.260)--(4.391,3.265)--(4.397,3.269)--(4.402,3.274)--(4.408,3.279)%
  --(4.413,3.284)--(4.419,3.289)--(4.425,3.293)--(4.430,3.298)--(4.436,3.303)--(4.441,3.308)%
  --(4.447,3.313)--(4.452,3.317)--(4.458,3.322)--(4.464,3.327)--(4.469,3.332)--(4.475,3.336)%
  --(4.480,3.341)--(4.486,3.346)--(4.491,3.351)--(4.497,3.356)--(4.503,3.360)--(4.508,3.365)%
  --(4.514,3.370)--(4.519,3.375)--(4.525,3.380)--(4.530,3.385)--(4.536,3.389)--(4.542,3.394)%
  --(4.547,3.399)--(4.553,3.404)--(4.558,3.409)--(4.564,3.413)--(4.569,3.418)--(4.575,3.423)%
  --(4.581,3.428)--(4.586,3.433)--(4.592,3.437)--(4.597,3.442)--(4.603,3.447)--(4.608,3.452)%
  --(4.614,3.457)--(4.620,3.461)--(4.625,3.466)--(4.631,3.471)--(4.636,3.476)--(4.642,3.481)%
  --(4.647,3.486)--(4.653,3.490)--(4.659,3.495)--(4.664,3.500)--(4.670,3.505)--(4.675,3.510)%
  --(4.681,3.514)--(4.686,3.519)--(4.692,3.524)--(4.698,3.529)--(4.703,3.534)--(4.709,3.539)%
  --(4.714,3.543)--(4.720,3.548)--(4.725,3.553)--(4.731,3.558)--(4.737,3.563)--(4.742,3.567)%
  --(4.748,3.572)--(4.753,3.577)--(4.759,3.582)--(4.764,3.587)--(4.770,3.592)--(4.776,3.596)%
  --(4.781,3.601)--(4.787,3.606)--(4.792,3.611)--(4.798,3.616)--(4.804,3.621)--(4.809,3.626)%
  --(4.815,3.630)--(4.820,3.635)--(4.826,3.640)--(4.831,3.645)--(4.837,3.650)--(4.843,3.655)%
  --(4.848,3.659)--(4.854,3.664)--(4.859,3.669)--(4.865,3.674)--(4.870,3.679)--(4.876,3.684)%
  --(4.882,3.689)--(4.887,3.693)--(4.893,3.698)--(4.898,3.703)--(4.904,3.708)--(4.909,3.713)%
  --(4.915,3.718)--(4.921,3.723)--(4.926,3.727)--(4.932,3.732)--(4.937,3.737)--(4.943,3.742)%
  --(4.948,3.747)--(4.954,3.752)--(4.960,3.757)--(4.965,3.762)--(4.971,3.766)--(4.976,3.771)%
  --(4.982,3.776)--(4.987,3.781)--(4.993,3.786)--(4.999,3.791)--(5.004,3.796)--(5.010,3.801)%
  --(5.015,3.805)--(5.021,3.810)--(5.026,3.815)--(5.032,3.820)--(5.038,3.825)--(5.043,3.830)%
  --(5.049,3.835)--(5.054,3.840)--(5.060,3.845)--(5.065,3.850)--(5.071,3.854)--(5.077,3.859)%
  --(5.082,3.864)--(5.088,3.869)--(5.093,3.874)--(5.099,3.879)--(5.104,3.884)--(5.110,3.889)%
  --(5.116,3.894)--(5.121,3.899)--(5.127,3.903)--(5.132,3.908)--(5.138,3.913)--(5.143,3.918)%
  --(5.149,3.923)--(5.155,3.928)--(5.160,3.933)--(5.166,3.938)--(5.171,3.943)--(5.177,3.948)%
  --(5.182,3.953)--(5.188,3.958)--(5.194,3.963)--(5.199,3.968)--(5.205,3.972)--(5.210,3.977)%
  --(5.216,3.982)--(5.221,3.987)--(5.227,3.992)--(5.233,3.997)--(5.238,4.002)--(5.244,4.007)%
  --(5.249,4.012)--(5.255,4.017)--(5.260,4.022)--(5.266,4.027)--(5.272,4.032)--(5.277,4.037)%
  --(5.283,4.042)--(5.288,4.047)--(5.294,4.052)--(5.299,4.057)--(5.305,4.062)--(5.311,4.067)%
  --(5.316,4.072)--(5.322,4.077)--(5.327,4.082)--(5.333,4.087)--(5.339,4.092)--(5.344,4.097)%
  --(5.350,4.102)--(5.355,4.107)--(5.361,4.112)--(5.366,4.117)--(5.372,4.122)--(5.378,4.127)%
  --(5.383,4.132)--(5.389,4.137)--(5.394,4.142)--(5.400,4.147)--(5.405,4.152)--(5.411,4.157)%
  --(5.417,4.162)--(5.422,4.167)--(5.428,4.172)--(5.433,4.177)--(5.439,4.182)--(5.444,4.187)%
  --(5.450,4.192)--(5.456,4.197)--(5.461,4.202)--(5.467,4.207)--(5.472,4.212)--(5.478,4.217)%
  --(5.483,4.222)--(5.489,4.227)--(5.495,4.232)--(5.500,4.237)--(5.506,4.242)--(5.511,4.247)%
  --(5.517,4.252)--(5.522,4.257)--(5.528,4.262)--(5.534,4.267)--(5.539,4.273)--(5.545,4.278)%
  --(5.550,4.283)--(5.556,4.288)--(5.561,4.293)--(5.567,4.298)--(5.573,4.303)--(5.578,4.308)%
  --(5.584,4.313)--(5.589,4.318)--(5.595,4.323)--(5.600,4.328)--(5.606,4.334)--(5.612,4.339)%
  --(5.617,4.344)--(5.623,4.349)--(5.628,4.354)--(5.634,4.359)--(5.639,4.364)--(5.645,4.369)%
  --(5.651,4.374)--(5.656,4.380)--(5.662,4.385)--(5.667,4.390)--(5.673,4.395)--(5.678,4.400)%
  --(5.684,4.405)--(5.690,4.410)--(5.695,4.415)--(5.701,4.421)--(5.706,4.426)--(5.712,4.431)%
  --(5.717,4.436)--(5.723,4.441)--(5.729,4.446)--(5.734,4.451)--(5.740,4.457)--(5.745,4.462)%
  --(5.751,4.467)--(5.756,4.472)--(5.762,4.477)--(5.768,4.483)--(5.773,4.488)--(5.779,4.493)%
  --(5.784,4.498)--(5.790,4.503)--(5.795,4.508)--(5.801,4.514)--(5.807,4.519)--(5.812,4.524)%
  --(5.818,4.529)--(5.823,4.534)--(5.829,4.540)--(5.834,4.545)--(5.840,4.550)--(5.846,4.555)%
  --(5.851,4.561)--(5.857,4.566)--(5.862,4.571)--(5.868,4.576)--(5.873,4.581)--(5.879,4.587)%
  --(5.885,4.592)--(5.890,4.597)--(5.896,4.602)--(5.901,4.608)--(5.907,4.613)--(5.913,4.618)%
  --(5.918,4.623)--(5.924,4.629)--(5.929,4.634)--(5.935,4.639)--(5.940,4.644)--(5.946,4.650)%
  --(5.952,4.655)--(5.957,4.660)--(5.963,4.666)--(5.968,4.671)--(5.974,4.676)--(5.979,4.681)%
  --(5.985,4.687)--(5.991,4.692)--(5.996,4.697)--(6.002,4.703)--(6.007,4.708)--(6.013,4.713)%
  --(6.018,4.719)--(6.024,4.724)--(6.030,4.729)--(6.035,4.735)--(6.041,4.740)--(6.046,4.745)%
  --(6.052,4.751)--(6.057,4.756)--(6.063,4.761)--(6.069,4.767)--(6.074,4.772)--(6.080,4.777)%
  --(6.085,4.783)--(6.091,4.788)--(6.096,4.793)--(6.102,4.799)--(6.108,4.804)--(6.113,4.809)%
  --(6.119,4.815)--(6.124,4.820)--(6.130,4.826)--(6.135,4.831)--(6.141,4.836)--(6.147,4.842)%
  --(6.152,4.847)--(6.158,4.852)--(6.163,4.858)--(6.169,4.863)--(6.174,4.869)--(6.180,4.874)%
  --(6.186,4.879)--(6.191,4.885)--(6.197,4.890)--(6.202,4.896)--(6.208,4.901)--(6.213,4.907)%
  --(6.219,4.912)--(6.225,4.917)--(6.230,4.923)--(6.236,4.928)--(6.241,4.934)--(6.247,4.939)%
  --(6.252,4.945)--(6.258,4.950)--(6.264,4.956)--(6.269,4.961)--(6.275,4.966)--(6.280,4.972)%
  --(6.286,4.977)--(6.291,4.983)--(6.297,4.988)--(6.303,4.994)--(6.308,4.999)--(6.314,5.005)%
  --(6.319,5.010)--(6.325,5.016)--(6.330,5.021)--(6.336,5.027)--(6.342,5.032)--(6.347,5.038)%
  --(6.353,5.043)--(6.358,5.049)--(6.364,5.054)--(6.369,5.060)--(6.375,5.065)--(6.381,5.071)%
  --(6.386,5.076)--(6.392,5.082)--(6.397,5.087)--(6.403,5.093)--(6.408,5.098)--(6.414,5.104)%
  --(6.420,5.109)--(6.425,5.115)--(6.431,5.120)--(6.436,5.126)--(6.442,5.131)--(6.448,5.137)%
  --(6.453,5.142)--(6.459,5.148)--(6.464,5.153)--(6.470,5.159)--(6.475,5.164)--(6.481,5.170)%
  --(6.487,5.176)--(6.492,5.181)--(6.498,5.187)--(6.503,5.192)--(6.509,5.198)--(6.514,5.203)%
  --(6.520,5.209)--(6.526,5.214)--(6.531,5.220)--(6.537,5.225)--(6.542,5.231)--(6.548,5.236)%
  --(6.553,5.242)--(6.559,5.247)--(6.565,5.253)--(6.570,5.258)--(6.576,5.264)--(6.581,5.269)%
  --(6.587,5.275)--(6.592,5.280)--(6.598,5.286)--(6.604,5.292)--(6.609,5.297)--(6.615,5.302)%
  --(6.620,5.308)--(6.626,5.313)--(6.631,5.319)--(6.637,5.324)--(6.643,5.330)--(6.648,5.335)%
  --(6.654,5.341)--(6.659,5.346)--(6.665,5.352)--(6.670,5.357)--(6.676,5.363)--(6.682,5.368)%
  --(6.687,5.373)--(6.693,5.379)--(6.698,5.384)--(6.704,5.390)--(6.709,5.395)--(6.715,5.400)%
  --(6.721,5.406)--(6.726,5.411)--(6.732,5.416)--(6.737,5.422)--(6.743,5.427)--(6.748,5.432)%
  --(6.754,5.438)--(6.760,5.443)--(6.765,5.448)--(6.771,5.453)--(6.776,5.459)--(6.782,5.464)%
  --(6.787,5.469)--(6.793,5.474)--(6.799,5.479)--(6.804,5.484)--(6.810,5.489)--(6.815,5.494)%
  --(6.821,5.499)--(6.826,5.504)--(6.832,5.509)--(6.838,5.514)--(6.843,5.519)--(6.849,5.524)%
  --(6.854,5.529)--(6.860,5.533)--(6.865,5.538)--(6.871,5.543)--(6.877,5.548)--(6.882,5.552)%
  --(6.888,5.557)--(6.893,5.561)--(6.899,5.565)--(6.904,5.570)--(6.910,5.574)--(6.916,5.578)%
  --(6.921,5.582)--(6.927,5.586)--(6.932,5.590)--(6.938,5.594)--(6.943,5.598)--(6.949,5.602)%
  --(6.955,5.605)--(6.960,5.609)--(6.966,5.612)--(6.971,5.615)--(6.977,5.618)--(6.982,5.621)%
  --(6.988,5.623)--(6.994,5.626)--(6.999,5.628)--(7.005,5.629)--(7.010,5.630)--(7.016,5.630);
\gpcolor{\gprgb{0}{0}{1000}}
\draw[gp path] (2.621,0.700)--(2.621,0.783)--(2.621,0.867)--(2.621,0.950)--(2.621,1.034)%
  --(2.621,1.118)--(2.621,1.201)--(2.621,1.285)--(2.621,1.368)--(2.621,1.452)--(2.621,1.535)%
  --(2.621,1.619)--(2.621,1.703)--(2.621,1.786)--(2.621,1.870)--(2.621,1.953)--(2.621,2.037)%
  --(2.621,2.121)--(2.621,2.204)--(2.621,2.288)--(2.621,2.371)--(3.235,2.371)--(3.235,2.455)%
  --(3.235,2.538)--(3.235,2.622)--(3.235,2.706)--(3.235,2.789)--(3.235,2.873)--(3.425,2.873)%
  --(3.425,2.956)--(3.425,3.040)--(3.425,3.124)--(3.425,3.207)--(4.229,3.207)--(4.229,3.291)%
  --(4.230,3.291)--(4.230,3.374)--(4.230,3.458)--(4.230,3.541)--(4.230,3.625)--(5.034,3.625)%
  --(5.034,3.709)--(5.034,3.792)--(5.034,3.876)--(5.034,3.959)--(5.034,4.043)--(5.034,4.127)%
  --(5.034,4.210)--(5.034,4.294)--(5.838,4.294)--(5.838,4.377)--(5.838,4.461)--(5.838,4.544)%
  --(5.838,4.628)--(5.838,4.712)--(6.643,4.712)--(6.643,4.795)--(6.643,4.879)--(6.643,4.962)%
  --(6.643,5.046)--(6.833,5.046)--(6.833,5.130)--(6.833,5.213)--(6.833,5.297)--(6.833,5.380)%
  --(6.833,5.464)--(6.833,5.547)--(7.447,5.547);
\gpdefrectangularnode{gp plot 1}{\pgfpoint{1.012cm}{0.616cm}}{\pgfpoint{7.447cm}{5.631cm}}
\end{tikzpicture}

%% file: ecf120.tex
\begin{tikzpicture}[gnuplot,scale=.7]
\gpfill{color=\gprgb{1000}{1000}{1000}} (1.012,0.616)--(7.447,0.616)--(7.447,5.631)--(1.012,5.631)--cycle;
\gpcolor{gp lt color border}
\gpsetlinetype{gp lt border}
\gpsetlinewidth{1.00}
\draw[gp path] (1.012,0.616)--(1.012,5.631)--(7.447,5.631)--(7.447,0.616)--cycle;
\gpsetlinewidth{0.50}
\draw[gp path] (1.012,0.616)--(1.263,0.616);
\draw[gp path] (7.447,0.616)--(7.196,0.616);
\gpcolor{\gprgb{0}{0}{0}}
\node[gp node right,font=\gpfontsize{10.00pt}{12.00pt}] at (0.828,0.616) {0};
\gpcolor{gp lt color border}
\draw[gp path] (1.012,1.619)--(1.263,1.619);
\draw[gp path] (7.447,1.619)--(7.196,1.619);
\gpcolor{\gprgb{0}{0}{0}}
\node[gp node right,font=\gpfontsize{10.00pt}{12.00pt}] at (0.828,1.619) {0.2};
\gpcolor{gp lt color border}
\draw[gp path] (1.012,2.622)--(1.263,2.622);
\draw[gp path] (7.447,2.622)--(7.196,2.622);
\gpcolor{\gprgb{0}{0}{0}}
\node[gp node right,font=\gpfontsize{10.00pt}{12.00pt}] at (0.828,2.622) {0.4};
\gpcolor{gp lt color border}
\draw[gp path] (1.012,3.625)--(1.263,3.625);
\draw[gp path] (7.447,3.625)--(7.196,3.625);
\gpcolor{\gprgb{0}{0}{0}}
\node[gp node right,font=\gpfontsize{10.00pt}{12.00pt}] at (0.828,3.625) {0.6};
\gpcolor{gp lt color border}
\draw[gp path] (1.012,4.628)--(1.263,4.628);
\draw[gp path] (7.447,4.628)--(7.196,4.628);
\gpcolor{\gprgb{0}{0}{0}}
\node[gp node right,font=\gpfontsize{10.00pt}{12.00pt}] at (0.828,4.628) {0.8};
\gpcolor{gp lt color border}
\draw[gp path] (1.012,5.631)--(1.263,5.631);
\draw[gp path] (7.447,5.631)--(7.196,5.631);
\gpcolor{\gprgb{0}{0}{0}}
\node[gp node right,font=\gpfontsize{10.00pt}{12.00pt}] at (0.828,5.631) {1};
\gpcolor{gp lt color border}
\draw[gp path] (1.012,0.616)--(1.012,0.867);
\draw[gp path] (1.012,5.631)--(1.012,5.380);
\gpcolor{\gprgb{0}{0}{0}}
\node[gp node center,font=\gpfontsize{10.00pt}{12.00pt}] at (1.012,0.308) {-4};
\gpcolor{gp lt color border}
\draw[gp path] (1.816,0.616)--(1.816,0.867);
\draw[gp path] (1.816,5.631)--(1.816,5.380);
\gpcolor{\gprgb{0}{0}{0}}
\node[gp node center,font=\gpfontsize{10.00pt}{12.00pt}] at (1.816,0.308) {-3};
\gpcolor{gp lt color border}
\draw[gp path] (2.621,0.616)--(2.621,0.867);
\draw[gp path] (2.621,5.631)--(2.621,5.380);
\gpcolor{\gprgb{0}{0}{0}}
\node[gp node center,font=\gpfontsize{10.00pt}{12.00pt}] at (2.621,0.308) {-2};
\gpcolor{gp lt color border}
\draw[gp path] (3.425,0.616)--(3.425,0.867);
\draw[gp path] (3.425,5.631)--(3.425,5.380);
\gpcolor{\gprgb{0}{0}{0}}
\node[gp node center,font=\gpfontsize{10.00pt}{12.00pt}] at (3.425,0.308) {-1};
\gpcolor{gp lt color border}
\draw[gp path] (4.230,0.616)--(4.230,0.867);
\draw[gp path] (4.230,5.631)--(4.230,5.380);
\gpcolor{\gprgb{0}{0}{0}}
\node[gp node center,font=\gpfontsize{10.00pt}{12.00pt}] at (4.230,0.308) {0};
\gpcolor{gp lt color border}
\draw[gp path] (5.034,0.616)--(5.034,0.867);
\draw[gp path] (5.034,5.631)--(5.034,5.380);
\gpcolor{\gprgb{0}{0}{0}}
\node[gp node center,font=\gpfontsize{10.00pt}{12.00pt}] at (5.034,0.308) {1};
\gpcolor{gp lt color border}
\draw[gp path] (5.838,0.616)--(5.838,0.867);
\draw[gp path] (5.838,5.631)--(5.838,5.380);
\gpcolor{\gprgb{0}{0}{0}}
\node[gp node center,font=\gpfontsize{10.00pt}{12.00pt}] at (5.838,0.308) {2};
\gpcolor{gp lt color border}
\draw[gp path] (6.643,0.616)--(6.643,0.867);
\draw[gp path] (6.643,5.631)--(6.643,5.380);
\gpcolor{\gprgb{0}{0}{0}}
\node[gp node center,font=\gpfontsize{10.00pt}{12.00pt}] at (6.643,0.308) {3};
\gpcolor{gp lt color border}
\draw[gp path] (7.447,0.616)--(7.447,0.867);
\draw[gp path] (7.447,5.631)--(7.447,5.380);
\gpcolor{\gprgb{0}{0}{0}}
\node[gp node center,font=\gpfontsize{10.00pt}{12.00pt}] at (7.447,0.308) {4};
\gpcolor{gp lt color border}
\draw[gp path] (1.012,5.631)--(1.012,0.616)--(7.447,0.616)--(7.447,5.631)--cycle;
\gpcolor{\gprgb{0}{1000}{0}}
\gpsetlinetype{gp lt plot 0}
\draw[gp path] (1.443,0.616)--(1.449,0.789)--(1.454,0.857)--(1.460,0.908)--(1.465,0.949)%
  --(1.471,0.984)--(1.477,1.015)--(1.482,1.042)--(1.488,1.066)--(1.493,1.088)--(1.499,1.108)%
  --(1.504,1.126)--(1.510,1.143)--(1.516,1.159)--(1.521,1.173)--(1.527,1.187)--(1.532,1.200)%
  --(1.538,1.211)--(1.543,1.223)--(1.549,1.233)--(1.555,1.243)--(1.560,1.252)--(1.566,1.261)%
  --(1.571,1.269)--(1.577,1.277)--(1.582,1.284)--(1.588,1.291)--(1.594,1.297)--(1.599,1.304)%
  --(1.605,1.310)--(1.610,1.315)--(1.616,1.320)--(1.621,1.326)--(1.627,1.330)--(1.633,1.335)%
  --(1.638,1.339)--(1.644,1.343)--(1.649,1.347)--(1.655,1.351)--(1.660,1.354)--(1.666,1.358)%
  --(1.672,1.361)--(1.677,1.364)--(1.683,1.367)--(1.688,1.370)--(1.694,1.372)--(1.699,1.375)%
  --(1.705,1.377)--(1.711,1.379)--(1.716,1.382)--(1.722,1.384)--(1.727,1.386)--(1.733,1.387)%
  --(1.738,1.389)--(1.744,1.391)--(1.750,1.392)--(1.755,1.394)--(1.761,1.395)--(1.766,1.397)%
  --(1.772,1.398)--(1.777,1.399)--(1.783,1.400)--(1.789,1.401)--(1.794,1.402)--(1.800,1.403)%
  --(1.805,1.404)--(1.811,1.405)--(1.816,1.406)--(1.822,1.407)--(1.828,1.407)--(1.833,1.408)%
  --(1.839,1.409)--(1.844,1.409)--(1.850,1.410)--(1.855,1.410)--(1.861,1.411)--(1.867,1.411)%
  --(1.872,1.412)--(1.878,1.412)--(1.883,1.412)--(1.889,1.413)--(1.894,1.413)--(1.900,1.413)%
  --(1.906,1.413)--(1.911,1.414)--(1.917,1.414)--(1.922,1.414)--(1.928,1.414)--(1.933,1.414)%
  --(1.939,1.414)--(1.945,1.414)--(1.950,1.414)--(1.956,1.414)--(1.961,1.414)--(1.967,1.414)%
  --(1.972,1.414)--(1.978,1.414)--(1.984,1.414)--(1.989,1.414)--(1.995,1.414)--(2.000,1.414)%
  --(2.006,1.413)--(2.011,1.413)--(2.017,1.413)--(2.023,1.413)--(2.028,1.413)--(2.034,1.413)%
  --(2.039,1.412)--(2.045,1.412)--(2.051,1.412)--(2.056,1.412)--(2.062,1.411)--(2.067,1.411)%
  --(2.073,1.411)--(2.078,1.410)--(2.084,1.410)--(2.090,1.410)--(2.095,1.410)--(2.101,1.409)%
  --(2.106,1.409)--(2.112,1.409)--(2.117,1.408)--(2.123,1.408)--(2.129,1.408)--(2.134,1.407)%
  --(2.140,1.407)--(2.145,1.406)--(2.151,1.406)--(2.156,1.406)--(2.162,1.405)--(2.168,1.405)%
  --(2.173,1.405)--(2.179,1.404)--(2.184,1.404)--(2.190,1.403)--(2.195,1.403)--(2.201,1.403)%
  --(2.207,1.402)--(2.212,1.402)--(2.218,1.401)--(2.223,1.401)--(2.229,1.400)--(2.234,1.400)%
  --(2.240,1.400)--(2.246,1.399)--(2.251,1.399)--(2.257,1.398)--(2.262,1.398)--(2.268,1.397)%
  --(2.273,1.397)--(2.279,1.396)--(2.285,1.396)--(2.290,1.396)--(2.296,1.395)--(2.301,1.395)%
  --(2.307,1.394)--(2.312,1.394)--(2.318,1.393)--(2.324,1.393)--(2.329,1.392)--(2.335,1.392)%
  --(2.340,1.391)--(2.346,1.391)--(2.351,1.391)--(2.357,1.390)--(2.363,1.390)--(2.368,1.389)%
  --(2.374,1.389)--(2.379,1.388)--(2.385,1.388)--(2.390,1.387)--(2.396,1.387)--(2.402,1.386)%
  --(2.407,1.386)--(2.413,1.385)--(2.418,1.385)--(2.424,1.384)--(2.429,1.384)--(2.435,1.384)%
  --(2.441,1.383)--(2.446,1.383)--(2.452,1.382)--(2.457,1.382)--(2.463,1.381)--(2.468,1.381)%
  --(2.474,1.380)--(2.480,1.380)--(2.485,1.379)--(2.491,1.379)--(2.496,1.378)--(2.502,1.378)%
  --(2.507,1.378)--(2.513,1.377)--(2.519,1.377)--(2.524,1.376)--(2.530,1.376)--(2.535,1.375)%
  --(2.541,1.375)--(2.546,1.374)--(2.552,1.374)--(2.558,1.374)--(2.563,1.373)--(2.569,1.373)%
  --(2.574,1.372)--(2.580,1.372)--(2.586,1.371)--(2.591,1.371)--(2.597,1.370)--(2.602,1.370)%
  --(2.608,1.370)--(2.613,1.369)--(2.619,1.369)--(2.625,1.368)--(2.630,1.368)--(2.636,1.367)%
  --(2.641,1.367)--(2.647,1.366)--(2.652,1.366)--(2.658,1.366)--(2.664,1.365)--(2.669,1.365)%
  --(2.675,1.364)--(2.680,1.364)--(2.686,1.364)--(2.691,1.363)--(2.697,1.363)--(2.703,1.362)%
  --(2.708,1.362)--(2.714,1.361)--(2.719,1.361)--(2.725,1.361)--(2.730,1.360)--(2.736,1.360)%
  --(2.742,1.359)--(2.747,1.359)--(2.753,1.359)--(2.758,1.358)--(2.764,1.358)--(2.769,1.357)%
  --(2.775,1.357)--(2.781,1.357)--(2.786,1.356)--(2.792,1.356)--(2.797,1.355)--(2.803,1.355)%
  --(2.808,1.355)--(2.814,1.354)--(2.820,1.354)--(2.825,1.354)--(2.831,1.353)--(2.836,1.353)%
  --(2.842,1.352)--(2.847,1.352)--(2.853,1.352)--(2.859,1.351)--(2.864,1.351)--(2.870,1.351)%
  --(2.875,1.350)--(2.881,1.350)--(2.886,1.349)--(2.892,1.349)--(2.898,1.349)--(2.903,1.348)%
  --(2.909,1.348)--(2.914,1.348)--(2.920,1.347)--(2.925,1.347)--(2.931,1.347)--(2.937,1.346)%
  --(2.942,1.346)--(2.948,1.346)--(2.953,1.345)--(2.959,1.345)--(2.964,1.345)--(2.970,1.344)%
  --(2.976,1.344)--(2.981,1.343)--(2.987,1.343)--(2.992,1.343)--(2.998,1.342)--(3.003,1.342)%
  --(3.009,1.342)--(3.015,1.342)--(3.020,1.341)--(3.026,1.341)--(3.031,1.341)--(3.037,1.340)%
  --(3.042,1.340)--(3.048,1.340)--(3.054,1.339)--(3.059,1.339)--(3.065,1.339)--(3.070,1.338)%
  --(3.076,1.338)--(3.081,1.338)--(3.087,1.337)--(3.093,1.337)--(3.098,1.337)--(3.104,1.337)%
  --(3.109,1.336)--(3.115,1.336)--(3.120,1.336)--(3.126,1.335)--(3.132,1.335)--(3.137,1.335)%
  --(3.143,1.334)--(3.148,1.334)--(3.154,1.334)--(3.160,1.334)--(3.165,1.333)--(3.171,1.333)%
  --(3.176,1.333)--(3.182,1.332)--(3.187,1.332)--(3.193,1.332)--(3.199,1.332)--(3.204,1.331)%
  --(3.210,1.331)--(3.215,1.331)--(3.221,1.331)--(3.226,1.330)--(3.232,1.330)--(3.238,1.330)%
  --(3.243,1.330)--(3.249,1.329)--(3.254,1.329)--(3.260,1.329)--(3.265,1.329)--(3.271,1.328)%
  --(3.277,1.328)--(3.282,1.328)--(3.288,1.328)--(3.293,1.327)--(3.299,1.327)--(3.304,1.327)%
  --(3.310,1.327)--(3.316,1.326)--(3.321,1.326)--(3.327,1.326)--(3.332,1.326)--(3.338,1.325)%
  --(3.343,1.325)--(3.349,1.325)--(3.355,1.325)--(3.360,1.324)--(3.366,1.324)--(3.371,1.324)%
  --(3.377,1.324)--(3.382,1.324)--(3.388,1.323)--(3.394,1.323)--(3.399,1.323)--(3.405,1.323)%
  --(3.410,1.322)--(3.416,1.322)--(3.421,1.322)--(3.427,1.322)--(3.433,1.322)--(3.438,1.321)%
  --(3.444,1.321)--(3.449,1.321)--(3.455,1.321)--(3.460,1.321)--(3.466,1.320)--(3.472,1.320)%
  --(3.477,1.320)--(3.483,1.320)--(3.488,1.320)--(3.494,1.319)--(3.499,1.319)--(3.505,1.319)%
  --(3.511,1.319)--(3.516,1.319)--(3.522,1.319)--(3.527,1.318)--(3.533,1.318)--(3.538,1.318)%
  --(3.544,1.318)--(3.550,1.318)--(3.555,1.317)--(3.561,1.317)--(3.566,1.317)--(3.572,1.317)%
  --(3.577,1.317)--(3.583,1.317)--(3.589,1.316)--(3.594,1.316)--(3.600,1.316)--(3.605,1.316)%
  --(3.611,1.316)--(3.616,1.316)--(3.622,1.316)--(3.628,1.315)--(3.633,1.315)--(3.639,1.315)%
  --(3.644,1.315)--(3.650,1.315)--(3.655,1.315)--(3.661,1.314)--(3.667,1.314)--(3.672,1.314)%
  --(3.678,1.314)--(3.683,1.314)--(3.689,1.314)--(3.695,1.314)--(3.700,1.314)--(3.706,1.313)%
  --(3.711,1.313)--(3.717,1.313)--(3.722,1.313)--(3.728,1.313)--(3.734,1.313)--(3.739,1.313)%
  --(3.745,1.312)--(3.750,1.312)--(3.756,1.312)--(3.761,1.312)--(3.767,1.312)--(3.773,1.312)%
  --(3.778,1.312)--(3.784,1.312)--(3.789,1.312)--(3.795,1.311)--(3.800,1.311)--(3.806,1.311)%
  --(3.812,1.311)--(3.817,1.311)--(3.823,1.311)--(3.828,1.311)--(3.834,1.311)--(3.839,1.311)%
  --(3.845,1.311)--(3.851,1.310)--(3.856,1.310)--(3.862,1.310)--(3.867,1.310)--(3.873,1.310)%
  --(3.878,1.310)--(3.884,1.310)--(3.890,1.310)--(3.895,1.310)--(3.901,1.310)--(3.906,1.310)%
  --(3.912,1.309)--(3.917,1.309)--(3.923,1.309)--(3.929,1.309)--(3.934,1.309)--(3.940,1.309)%
  --(3.945,1.309)--(3.951,1.309)--(3.956,1.309)--(3.962,1.309)--(3.968,1.309)--(3.973,1.309)%
  --(3.979,1.309)--(3.984,1.309)--(3.990,1.309)--(3.995,1.308)--(4.001,1.308)--(4.007,1.308)%
  --(4.012,1.308)--(4.018,1.308)--(4.023,1.308)--(4.029,1.308)--(4.034,1.308)--(4.040,1.308)%
  --(4.046,1.308)--(4.051,1.308)--(4.057,1.308)--(4.062,1.308)--(4.068,1.308)--(4.073,1.308)%
  --(4.079,1.308)--(4.085,1.308)--(4.090,1.308)--(4.096,1.308)--(4.101,1.308)--(4.107,1.308)%
  --(4.112,1.308)--(4.118,1.308)--(4.124,1.307)--(4.129,1.307)--(4.135,1.307)--(4.140,1.307)%
  --(4.146,1.307)--(4.151,1.307)--(4.157,1.307)--(4.163,1.307)--(4.168,1.307)--(4.174,1.307)%
  --(4.179,1.307)--(4.185,1.307)--(4.190,1.307)--(4.196,1.307)--(4.202,1.307)--(4.207,1.307)%
  --(4.213,1.307)--(4.218,1.307)--(4.224,1.307)--(4.230,1.307)--(4.235,1.307)--(4.241,1.307)%
  --(4.246,1.307)--(4.252,1.307)--(4.257,1.307)--(4.263,1.307)--(4.269,1.307)--(4.274,1.307)%
  --(4.280,1.307)--(4.285,1.307)--(4.291,1.307)--(4.296,1.307)--(4.302,1.307)--(4.308,1.307)%
  --(4.313,1.307)--(4.319,1.307)--(4.324,1.307)--(4.330,1.307)--(4.335,1.307)--(4.341,1.308)%
  --(4.347,1.308)--(4.352,1.308)--(4.358,1.308)--(4.363,1.308)--(4.369,1.308)--(4.374,1.308)%
  --(4.380,1.308)--(4.386,1.308)--(4.391,1.308)--(4.397,1.308)--(4.402,1.308)--(4.408,1.308)%
  --(4.413,1.308)--(4.419,1.308)--(4.425,1.308)--(4.430,1.308)--(4.436,1.308)--(4.441,1.308)%
  --(4.447,1.308)--(4.452,1.308)--(4.458,1.308)--(4.464,1.308)--(4.469,1.309)--(4.475,1.309)%
  --(4.480,1.309)--(4.486,1.309)--(4.491,1.309)--(4.497,1.309)--(4.503,1.309)--(4.508,1.309)%
  --(4.514,1.309)--(4.519,1.309)--(4.525,1.309)--(4.530,1.309)--(4.536,1.309)--(4.542,1.309)%
  --(4.547,1.309)--(4.553,1.310)--(4.558,1.310)--(4.564,1.310)--(4.569,1.310)--(4.575,1.310)%
  --(4.581,1.310)--(4.586,1.310)--(4.592,1.310)--(4.597,1.310)--(4.603,1.310)--(4.608,1.310)%
  --(4.614,1.311)--(4.620,1.311)--(4.625,1.311)--(4.631,1.311)--(4.636,1.311)--(4.642,1.311)%
  --(4.647,1.311)--(4.653,1.311)--(4.659,1.311)--(4.664,1.311)--(4.670,1.312)--(4.675,1.312)%
  --(4.681,1.312)--(4.686,1.312)--(4.692,1.312)--(4.698,1.312)--(4.703,1.312)--(4.709,1.312)%
  --(4.714,1.312)--(4.720,1.313)--(4.725,1.313)--(4.731,1.313)--(4.737,1.313)--(4.742,1.313)%
  --(4.748,1.313)--(4.753,1.313)--(4.759,1.314)--(4.764,1.314)--(4.770,1.314)--(4.776,1.314)%
  --(4.781,1.314)--(4.787,1.314)--(4.792,1.314)--(4.798,1.314)--(4.804,1.315)--(4.809,1.315)%
  --(4.815,1.315)--(4.820,1.315)--(4.826,1.315)--(4.831,1.315)--(4.837,1.316)--(4.843,1.316)%
  --(4.848,1.316)--(4.854,1.316)--(4.859,1.316)--(4.865,1.316)--(4.870,1.316)--(4.876,1.317)%
  --(4.882,1.317)--(4.887,1.317)--(4.893,1.317)--(4.898,1.317)--(4.904,1.317)--(4.909,1.318)%
  --(4.915,1.318)--(4.921,1.318)--(4.926,1.318)--(4.932,1.318)--(4.937,1.319)--(4.943,1.319)%
  --(4.948,1.319)--(4.954,1.319)--(4.960,1.319)--(4.965,1.319)--(4.971,1.320)--(4.976,1.320)%
  --(4.982,1.320)--(4.987,1.320)--(4.993,1.320)--(4.999,1.321)--(5.004,1.321)--(5.010,1.321)%
  --(5.015,1.321)--(5.021,1.321)--(5.026,1.322)--(5.032,1.322)--(5.038,1.322)--(5.043,1.322)%
  --(5.049,1.322)--(5.054,1.323)--(5.060,1.323)--(5.065,1.323)--(5.071,1.323)--(5.077,1.324)%
  --(5.082,1.324)--(5.088,1.324)--(5.093,1.324)--(5.099,1.324)--(5.104,1.325)--(5.110,1.325)%
  --(5.116,1.325)--(5.121,1.325)--(5.127,1.326)--(5.132,1.326)--(5.138,1.326)--(5.143,1.326)%
  --(5.149,1.327)--(5.155,1.327)--(5.160,1.327)--(5.166,1.327)--(5.171,1.328)--(5.177,1.328)%
  --(5.182,1.328)--(5.188,1.328)--(5.194,1.329)--(5.199,1.329)--(5.205,1.329)--(5.210,1.329)%
  --(5.216,1.330)--(5.221,1.330)--(5.227,1.330)--(5.233,1.330)--(5.238,1.331)--(5.244,1.331)%
  --(5.249,1.331)--(5.255,1.331)--(5.260,1.332)--(5.266,1.332)--(5.272,1.332)--(5.277,1.332)%
  --(5.283,1.333)--(5.288,1.333)--(5.294,1.333)--(5.299,1.334)--(5.305,1.334)--(5.311,1.334)%
  --(5.316,1.334)--(5.322,1.335)--(5.327,1.335)--(5.333,1.335)--(5.339,1.336)--(5.344,1.336)%
  --(5.350,1.336)--(5.355,1.337)--(5.361,1.337)--(5.366,1.337)--(5.372,1.337)--(5.378,1.338)%
  --(5.383,1.338)--(5.389,1.338)--(5.394,1.339)--(5.400,1.339)--(5.405,1.339)--(5.411,1.340)%
  --(5.417,1.340)--(5.422,1.340)--(5.428,1.341)--(5.433,1.341)--(5.439,1.341)--(5.444,1.342)%
  --(5.450,1.342)--(5.456,1.342)--(5.461,1.342)--(5.467,1.343)--(5.472,1.343)--(5.478,1.343)%
  --(5.483,1.344)--(5.489,1.344)--(5.495,1.345)--(5.500,1.345)--(5.506,1.345)--(5.511,1.346)%
  --(5.517,1.346)--(5.522,1.346)--(5.528,1.347)--(5.534,1.347)--(5.539,1.347)--(5.545,1.348)%
  --(5.550,1.348)--(5.556,1.348)--(5.561,1.349)--(5.567,1.349)--(5.573,1.349)--(5.578,1.350)%
  --(5.584,1.350)--(5.589,1.351)--(5.595,1.351)--(5.600,1.351)--(5.606,1.352)--(5.612,1.352)%
  --(5.617,1.352)--(5.623,1.353)--(5.628,1.353)--(5.634,1.354)--(5.639,1.354)--(5.645,1.354)%
  --(5.651,1.355)--(5.656,1.355)--(5.662,1.355)--(5.667,1.356)--(5.673,1.356)--(5.678,1.357)%
  --(5.684,1.357)--(5.690,1.357)--(5.695,1.358)--(5.701,1.358)--(5.706,1.359)--(5.712,1.359)%
  --(5.717,1.359)--(5.723,1.360)--(5.729,1.360)--(5.734,1.361)--(5.740,1.361)--(5.745,1.361)%
  --(5.751,1.362)--(5.756,1.362)--(5.762,1.363)--(5.768,1.363)--(5.773,1.364)--(5.779,1.364)%
  --(5.784,1.364)--(5.790,1.365)--(5.795,1.365)--(5.801,1.366)--(5.807,1.366)--(5.812,1.366)%
  --(5.818,1.367)--(5.823,1.367)--(5.829,1.368)--(5.834,1.368)--(5.840,1.369)--(5.846,1.369)%
  --(5.851,1.370)--(5.857,1.370)--(5.862,1.370)--(5.868,1.371)--(5.873,1.371)--(5.879,1.372)%
  --(5.885,1.372)--(5.890,1.373)--(5.896,1.373)--(5.901,1.374)--(5.907,1.374)--(5.913,1.374)%
  --(5.918,1.375)--(5.924,1.375)--(5.929,1.376)--(5.935,1.376)--(5.940,1.377)--(5.946,1.377)%
  --(5.952,1.378)--(5.957,1.378)--(5.963,1.378)--(5.968,1.379)--(5.974,1.379)--(5.979,1.380)%
  --(5.985,1.380)--(5.991,1.381)--(5.996,1.381)--(6.002,1.382)--(6.007,1.382)--(6.013,1.383)%
  --(6.018,1.383)--(6.024,1.384)--(6.030,1.384)--(6.035,1.384)--(6.041,1.385)--(6.046,1.385)%
  --(6.052,1.386)--(6.057,1.386)--(6.063,1.387)--(6.069,1.387)--(6.074,1.388)--(6.080,1.388)%
  --(6.085,1.389)--(6.091,1.389)--(6.096,1.390)--(6.102,1.390)--(6.108,1.391)--(6.113,1.391)%
  --(6.119,1.391)--(6.124,1.392)--(6.130,1.392)--(6.135,1.393)--(6.141,1.393)--(6.147,1.394)%
  --(6.152,1.394)--(6.158,1.395)--(6.163,1.395)--(6.169,1.396)--(6.174,1.396)--(6.180,1.396)%
  --(6.186,1.397)--(6.191,1.397)--(6.197,1.398)--(6.202,1.398)--(6.208,1.399)--(6.213,1.399)%
  --(6.219,1.400)--(6.225,1.400)--(6.230,1.400)--(6.236,1.401)--(6.241,1.401)--(6.247,1.402)%
  --(6.252,1.402)--(6.258,1.403)--(6.264,1.403)--(6.269,1.403)--(6.275,1.404)--(6.280,1.404)%
  --(6.286,1.405)--(6.291,1.405)--(6.297,1.405)--(6.303,1.406)--(6.308,1.406)--(6.314,1.406)%
  --(6.319,1.407)--(6.325,1.407)--(6.330,1.408)--(6.336,1.408)--(6.342,1.408)--(6.347,1.409)%
  --(6.353,1.409)--(6.358,1.409)--(6.364,1.410)--(6.369,1.410)--(6.375,1.410)--(6.381,1.410)%
  --(6.386,1.411)--(6.392,1.411)--(6.397,1.411)--(6.403,1.412)--(6.408,1.412)--(6.414,1.412)%
  --(6.420,1.412)--(6.425,1.413)--(6.431,1.413)--(6.436,1.413)--(6.442,1.413)--(6.448,1.413)%
  --(6.453,1.413)--(6.459,1.414)--(6.464,1.414)--(6.470,1.414)--(6.475,1.414)--(6.481,1.414)%
  --(6.487,1.414)--(6.492,1.414)--(6.498,1.414)--(6.503,1.414)--(6.509,1.414)--(6.514,1.414)%
  --(6.520,1.414)--(6.526,1.414)--(6.531,1.414)--(6.537,1.414)--(6.542,1.414)--(6.548,1.414)%
  --(6.553,1.413)--(6.559,1.413)--(6.565,1.413)--(6.570,1.413)--(6.576,1.412)--(6.581,1.412)%
  --(6.587,1.412)--(6.592,1.411)--(6.598,1.411)--(6.604,1.410)--(6.609,1.410)--(6.615,1.409)%
  --(6.620,1.409)--(6.626,1.408)--(6.631,1.407)--(6.637,1.407)--(6.643,1.406)--(6.648,1.405)%
  --(6.654,1.404)--(6.659,1.403)--(6.665,1.402)--(6.670,1.401)--(6.676,1.400)--(6.682,1.399)%
  --(6.687,1.398)--(6.693,1.397)--(6.698,1.395)--(6.704,1.394)--(6.709,1.392)--(6.715,1.391)%
  --(6.721,1.389)--(6.726,1.387)--(6.732,1.386)--(6.737,1.384)--(6.743,1.382)--(6.748,1.379)%
  --(6.754,1.377)--(6.760,1.375)--(6.765,1.372)--(6.771,1.370)--(6.776,1.367)--(6.782,1.364)%
  --(6.787,1.361)--(6.793,1.358)--(6.799,1.354)--(6.804,1.351)--(6.810,1.347)--(6.815,1.343)%
  --(6.821,1.339)--(6.826,1.335)--(6.832,1.330)--(6.838,1.326)--(6.843,1.320)--(6.849,1.315)%
  --(6.854,1.310)--(6.860,1.304)--(6.865,1.297)--(6.871,1.291)--(6.877,1.284)--(6.882,1.277)%
  --(6.888,1.269)--(6.893,1.261)--(6.899,1.252)--(6.904,1.243)--(6.910,1.233)--(6.916,1.223)%
  --(6.921,1.211)--(6.927,1.200)--(6.932,1.187)--(6.938,1.173)--(6.943,1.159)--(6.949,1.143)%
  --(6.955,1.126)--(6.960,1.108)--(6.966,1.088)--(6.971,1.066)--(6.977,1.042)--(6.982,1.015)%
  --(6.988,0.984)--(6.994,0.949)--(6.999,0.908)--(7.005,0.857)--(7.010,0.789)--(7.016,0.616);
\gpcolor{\gprgb{1000}{0}{0}}
\draw[gp path] (1.443,0.616)--(1.449,0.617)--(1.454,0.619)--(1.460,0.621)--(1.465,0.623)%
  --(1.471,0.626)--(1.477,0.629)--(1.482,0.631)--(1.488,0.635)--(1.493,0.638)--(1.499,0.641)%
  --(1.504,0.645)--(1.510,0.648)--(1.516,0.652)--(1.521,0.656)--(1.527,0.660)--(1.532,0.664)%
  --(1.538,0.668)--(1.543,0.672)--(1.549,0.677)--(1.555,0.681)--(1.560,0.685)--(1.566,0.690)%
  --(1.571,0.694)--(1.577,0.699)--(1.582,0.704)--(1.588,0.708)--(1.594,0.713)--(1.599,0.718)%
  --(1.605,0.723)--(1.610,0.727)--(1.616,0.732)--(1.621,0.737)--(1.627,0.742)--(1.633,0.747)%
  --(1.638,0.752)--(1.644,0.757)--(1.649,0.762)--(1.655,0.767)--(1.660,0.772)--(1.666,0.778)%
  --(1.672,0.783)--(1.677,0.788)--(1.683,0.793)--(1.688,0.798)--(1.694,0.804)--(1.699,0.809)%
  --(1.705,0.814)--(1.711,0.819)--(1.716,0.825)--(1.722,0.830)--(1.727,0.835)--(1.733,0.841)%
  --(1.738,0.846)--(1.744,0.851)--(1.750,0.857)--(1.755,0.862)--(1.761,0.868)--(1.766,0.873)%
  --(1.772,0.878)--(1.777,0.884)--(1.783,0.889)--(1.789,0.895)--(1.794,0.900)--(1.800,0.906)%
  --(1.805,0.911)--(1.811,0.917)--(1.816,0.922)--(1.822,0.928)--(1.828,0.933)--(1.833,0.939)%
  --(1.839,0.944)--(1.844,0.949)--(1.850,0.955)--(1.855,0.960)--(1.861,0.966)--(1.867,0.972)%
  --(1.872,0.977)--(1.878,0.983)--(1.883,0.988)--(1.889,0.994)--(1.894,0.999)--(1.900,1.005)%
  --(1.906,1.010)--(1.911,1.016)--(1.917,1.021)--(1.922,1.027)--(1.928,1.032)--(1.933,1.038)%
  --(1.939,1.043)--(1.945,1.049)--(1.950,1.054)--(1.956,1.060)--(1.961,1.065)--(1.967,1.071)%
  --(1.972,1.076)--(1.978,1.082)--(1.984,1.088)--(1.989,1.093)--(1.995,1.099)--(2.000,1.104)%
  --(2.006,1.110)--(2.011,1.115)--(2.017,1.121)--(2.023,1.126)--(2.028,1.132)--(2.034,1.137)%
  --(2.039,1.143)--(2.045,1.148)--(2.051,1.154)--(2.056,1.159)--(2.062,1.165)--(2.067,1.170)%
  --(2.073,1.176)--(2.078,1.181)--(2.084,1.187)--(2.090,1.192)--(2.095,1.198)--(2.101,1.203)%
  --(2.106,1.209)--(2.112,1.214)--(2.117,1.220)--(2.123,1.225)--(2.129,1.231)--(2.134,1.236)%
  --(2.140,1.242)--(2.145,1.247)--(2.151,1.253)--(2.156,1.258)--(2.162,1.264)--(2.168,1.269)%
  --(2.173,1.275)--(2.179,1.280)--(2.184,1.285)--(2.190,1.291)--(2.195,1.296)--(2.201,1.302)%
  --(2.207,1.307)--(2.212,1.313)--(2.218,1.318)--(2.223,1.324)--(2.229,1.329)--(2.234,1.334)%
  --(2.240,1.340)--(2.246,1.345)--(2.251,1.351)--(2.257,1.356)--(2.262,1.362)--(2.268,1.367)%
  --(2.273,1.372)--(2.279,1.378)--(2.285,1.383)--(2.290,1.389)--(2.296,1.394)--(2.301,1.399)%
  --(2.307,1.405)--(2.312,1.410)--(2.318,1.416)--(2.324,1.421)--(2.329,1.426)--(2.335,1.432)%
  --(2.340,1.437)--(2.346,1.442)--(2.351,1.448)--(2.357,1.453)--(2.363,1.459)--(2.368,1.464)%
  --(2.374,1.469)--(2.379,1.475)--(2.385,1.480)--(2.390,1.485)--(2.396,1.491)--(2.402,1.496)%
  --(2.407,1.501)--(2.413,1.507)--(2.418,1.512)--(2.424,1.517)--(2.429,1.523)--(2.435,1.528)%
  --(2.441,1.533)--(2.446,1.539)--(2.452,1.544)--(2.457,1.549)--(2.463,1.554)--(2.468,1.560)%
  --(2.474,1.565)--(2.480,1.570)--(2.485,1.576)--(2.491,1.581)--(2.496,1.586)--(2.502,1.591)%
  --(2.507,1.597)--(2.513,1.602)--(2.519,1.607)--(2.524,1.613)--(2.530,1.618)--(2.535,1.623)%
  --(2.541,1.628)--(2.546,1.634)--(2.552,1.639)--(2.558,1.644)--(2.563,1.649)--(2.569,1.655)%
  --(2.574,1.660)--(2.580,1.665)--(2.586,1.670)--(2.591,1.676)--(2.597,1.681)--(2.602,1.686)%
  --(2.608,1.691)--(2.613,1.696)--(2.619,1.702)--(2.625,1.707)--(2.630,1.712)--(2.636,1.717)%
  --(2.641,1.722)--(2.647,1.728)--(2.652,1.733)--(2.658,1.738)--(2.664,1.743)--(2.669,1.748)%
  --(2.675,1.754)--(2.680,1.759)--(2.686,1.764)--(2.691,1.769)--(2.697,1.774)--(2.703,1.780)%
  --(2.708,1.785)--(2.714,1.790)--(2.719,1.795)--(2.725,1.800)--(2.730,1.805)--(2.736,1.810)%
  --(2.742,1.816)--(2.747,1.821)--(2.753,1.826)--(2.758,1.831)--(2.764,1.836)--(2.769,1.841)%
  --(2.775,1.846)--(2.781,1.852)--(2.786,1.857)--(2.792,1.862)--(2.797,1.867)--(2.803,1.872)%
  --(2.808,1.877)--(2.814,1.882)--(2.820,1.887)--(2.825,1.893)--(2.831,1.898)--(2.836,1.903)%
  --(2.842,1.908)--(2.847,1.913)--(2.853,1.918)--(2.859,1.923)--(2.864,1.928)--(2.870,1.933)%
  --(2.875,1.938)--(2.881,1.943)--(2.886,1.949)--(2.892,1.954)--(2.898,1.959)--(2.903,1.964)%
  --(2.909,1.969)--(2.914,1.974)--(2.920,1.979)--(2.925,1.984)--(2.931,1.989)--(2.937,1.994)%
  --(2.942,1.999)--(2.948,2.004)--(2.953,2.009)--(2.959,2.014)--(2.964,2.019)--(2.970,2.025)%
  --(2.976,2.030)--(2.981,2.035)--(2.987,2.040)--(2.992,2.045)--(2.998,2.050)--(3.003,2.055)%
  --(3.009,2.060)--(3.015,2.065)--(3.020,2.070)--(3.026,2.075)--(3.031,2.080)--(3.037,2.085)%
  --(3.042,2.090)--(3.048,2.095)--(3.054,2.100)--(3.059,2.105)--(3.065,2.110)--(3.070,2.115)%
  --(3.076,2.120)--(3.081,2.125)--(3.087,2.130)--(3.093,2.135)--(3.098,2.140)--(3.104,2.145)%
  --(3.109,2.150)--(3.115,2.155)--(3.120,2.160)--(3.126,2.165)--(3.132,2.170)--(3.137,2.175)%
  --(3.143,2.180)--(3.148,2.185)--(3.154,2.190)--(3.160,2.195)--(3.165,2.200)--(3.171,2.205)%
  --(3.176,2.210)--(3.182,2.215)--(3.187,2.220)--(3.193,2.225)--(3.199,2.229)--(3.204,2.234)%
  --(3.210,2.239)--(3.215,2.244)--(3.221,2.249)--(3.226,2.254)--(3.232,2.259)--(3.238,2.264)%
  --(3.243,2.269)--(3.249,2.274)--(3.254,2.279)--(3.260,2.284)--(3.265,2.289)--(3.271,2.294)%
  --(3.277,2.299)--(3.282,2.304)--(3.288,2.309)--(3.293,2.313)--(3.299,2.318)--(3.304,2.323)%
  --(3.310,2.328)--(3.316,2.333)--(3.321,2.338)--(3.327,2.343)--(3.332,2.348)--(3.338,2.353)%
  --(3.343,2.358)--(3.349,2.363)--(3.355,2.368)--(3.360,2.372)--(3.366,2.377)--(3.371,2.382)%
  --(3.377,2.387)--(3.382,2.392)--(3.388,2.397)--(3.394,2.402)--(3.399,2.407)--(3.405,2.412)%
  --(3.410,2.417)--(3.416,2.421)--(3.421,2.426)--(3.427,2.431)--(3.433,2.436)--(3.438,2.441)%
  --(3.444,2.446)--(3.449,2.451)--(3.455,2.456)--(3.460,2.461)--(3.466,2.465)--(3.472,2.470)%
  --(3.477,2.475)--(3.483,2.480)--(3.488,2.485)--(3.494,2.490)--(3.499,2.495)--(3.505,2.500)%
  --(3.511,2.504)--(3.516,2.509)--(3.522,2.514)--(3.527,2.519)--(3.533,2.524)--(3.538,2.529)%
  --(3.544,2.534)--(3.550,2.539)--(3.555,2.543)--(3.561,2.548)--(3.566,2.553)--(3.572,2.558)%
  --(3.577,2.563)--(3.583,2.568)--(3.589,2.572)--(3.594,2.577)--(3.600,2.582)--(3.605,2.587)%
  --(3.611,2.592)--(3.616,2.597)--(3.622,2.602)--(3.628,2.606)--(3.633,2.611)--(3.639,2.616)%
  --(3.644,2.621)--(3.650,2.626)--(3.655,2.631)--(3.661,2.635)--(3.667,2.640)--(3.672,2.645)%
  --(3.678,2.650)--(3.683,2.655)--(3.689,2.660)--(3.695,2.665)--(3.700,2.669)--(3.706,2.674)%
  --(3.711,2.679)--(3.717,2.684)--(3.722,2.689)--(3.728,2.693)--(3.734,2.698)--(3.739,2.703)%
  --(3.745,2.708)--(3.750,2.713)--(3.756,2.718)--(3.761,2.722)--(3.767,2.727)--(3.773,2.732)%
  --(3.778,2.737)--(3.784,2.742)--(3.789,2.747)--(3.795,2.751)--(3.800,2.756)--(3.806,2.761)%
  --(3.812,2.766)--(3.817,2.771)--(3.823,2.775)--(3.828,2.780)--(3.834,2.785)--(3.839,2.790)%
  --(3.845,2.795)--(3.851,2.800)--(3.856,2.804)--(3.862,2.809)--(3.867,2.814)--(3.873,2.819)%
  --(3.878,2.824)--(3.884,2.828)--(3.890,2.833)--(3.895,2.838)--(3.901,2.843)--(3.906,2.848)%
  --(3.912,2.852)--(3.917,2.857)--(3.923,2.862)--(3.929,2.867)--(3.934,2.872)--(3.940,2.876)%
  --(3.945,2.881)--(3.951,2.886)--(3.956,2.891)--(3.962,2.896)--(3.968,2.900)--(3.973,2.905)%
  --(3.979,2.910)--(3.984,2.915)--(3.990,2.920)--(3.995,2.924)--(4.001,2.929)--(4.007,2.934)%
  --(4.012,2.939)--(4.018,2.944)--(4.023,2.948)--(4.029,2.953)--(4.034,2.958)--(4.040,2.963)%
  --(4.046,2.968)--(4.051,2.972)--(4.057,2.977)--(4.062,2.982)--(4.068,2.987)--(4.073,2.992)%
  --(4.079,2.996)--(4.085,3.001)--(4.090,3.006)--(4.096,3.011)--(4.101,3.015)--(4.107,3.020)%
  --(4.112,3.025)--(4.118,3.030)--(4.124,3.035)--(4.129,3.039)--(4.135,3.044)--(4.140,3.049)%
  --(4.146,3.054)--(4.151,3.059)--(4.157,3.063)--(4.163,3.068)--(4.168,3.073)--(4.174,3.078)%
  --(4.179,3.083)--(4.185,3.087)--(4.190,3.092)--(4.196,3.097)--(4.202,3.102)--(4.207,3.106)%
  --(4.213,3.111)--(4.218,3.116)--(4.224,3.121)--(4.230,3.126)--(4.235,3.130)--(4.241,3.135)%
  --(4.246,3.140)--(4.252,3.145)--(4.257,3.150)--(4.263,3.154)--(4.269,3.159)--(4.274,3.164)%
  --(4.280,3.169)--(4.285,3.174)--(4.291,3.178)--(4.296,3.183)--(4.302,3.188)--(4.308,3.193)%
  --(4.313,3.197)--(4.319,3.202)--(4.324,3.207)--(4.330,3.212)--(4.335,3.217)--(4.341,3.221)%
  --(4.347,3.226)--(4.352,3.231)--(4.358,3.236)--(4.363,3.241)--(4.369,3.245)--(4.374,3.250)%
  --(4.380,3.255)--(4.386,3.260)--(4.391,3.265)--(4.397,3.269)--(4.402,3.274)--(4.408,3.279)%
  --(4.413,3.284)--(4.419,3.289)--(4.425,3.293)--(4.430,3.298)--(4.436,3.303)--(4.441,3.308)%
  --(4.447,3.313)--(4.452,3.317)--(4.458,3.322)--(4.464,3.327)--(4.469,3.332)--(4.475,3.336)%
  --(4.480,3.341)--(4.486,3.346)--(4.491,3.351)--(4.497,3.356)--(4.503,3.360)--(4.508,3.365)%
  --(4.514,3.370)--(4.519,3.375)--(4.525,3.380)--(4.530,3.385)--(4.536,3.389)--(4.542,3.394)%
  --(4.547,3.399)--(4.553,3.404)--(4.558,3.409)--(4.564,3.413)--(4.569,3.418)--(4.575,3.423)%
  --(4.581,3.428)--(4.586,3.433)--(4.592,3.437)--(4.597,3.442)--(4.603,3.447)--(4.608,3.452)%
  --(4.614,3.457)--(4.620,3.461)--(4.625,3.466)--(4.631,3.471)--(4.636,3.476)--(4.642,3.481)%
  --(4.647,3.486)--(4.653,3.490)--(4.659,3.495)--(4.664,3.500)--(4.670,3.505)--(4.675,3.510)%
  --(4.681,3.514)--(4.686,3.519)--(4.692,3.524)--(4.698,3.529)--(4.703,3.534)--(4.709,3.539)%
  --(4.714,3.543)--(4.720,3.548)--(4.725,3.553)--(4.731,3.558)--(4.737,3.563)--(4.742,3.567)%
  --(4.748,3.572)--(4.753,3.577)--(4.759,3.582)--(4.764,3.587)--(4.770,3.592)--(4.776,3.596)%
  --(4.781,3.601)--(4.787,3.606)--(4.792,3.611)--(4.798,3.616)--(4.804,3.621)--(4.809,3.626)%
  --(4.815,3.630)--(4.820,3.635)--(4.826,3.640)--(4.831,3.645)--(4.837,3.650)--(4.843,3.655)%
  --(4.848,3.659)--(4.854,3.664)--(4.859,3.669)--(4.865,3.674)--(4.870,3.679)--(4.876,3.684)%
  --(4.882,3.689)--(4.887,3.693)--(4.893,3.698)--(4.898,3.703)--(4.904,3.708)--(4.909,3.713)%
  --(4.915,3.718)--(4.921,3.723)--(4.926,3.727)--(4.932,3.732)--(4.937,3.737)--(4.943,3.742)%
  --(4.948,3.747)--(4.954,3.752)--(4.960,3.757)--(4.965,3.762)--(4.971,3.766)--(4.976,3.771)%
  --(4.982,3.776)--(4.987,3.781)--(4.993,3.786)--(4.999,3.791)--(5.004,3.796)--(5.010,3.801)%
  --(5.015,3.805)--(5.021,3.810)--(5.026,3.815)--(5.032,3.820)--(5.038,3.825)--(5.043,3.830)%
  --(5.049,3.835)--(5.054,3.840)--(5.060,3.845)--(5.065,3.850)--(5.071,3.854)--(5.077,3.859)%
  --(5.082,3.864)--(5.088,3.869)--(5.093,3.874)--(5.099,3.879)--(5.104,3.884)--(5.110,3.889)%
  --(5.116,3.894)--(5.121,3.899)--(5.127,3.903)--(5.132,3.908)--(5.138,3.913)--(5.143,3.918)%
  --(5.149,3.923)--(5.155,3.928)--(5.160,3.933)--(5.166,3.938)--(5.171,3.943)--(5.177,3.948)%
  --(5.182,3.953)--(5.188,3.958)--(5.194,3.963)--(5.199,3.968)--(5.205,3.972)--(5.210,3.977)%
  --(5.216,3.982)--(5.221,3.987)--(5.227,3.992)--(5.233,3.997)--(5.238,4.002)--(5.244,4.007)%
  --(5.249,4.012)--(5.255,4.017)--(5.260,4.022)--(5.266,4.027)--(5.272,4.032)--(5.277,4.037)%
  --(5.283,4.042)--(5.288,4.047)--(5.294,4.052)--(5.299,4.057)--(5.305,4.062)--(5.311,4.067)%
  --(5.316,4.072)--(5.322,4.077)--(5.327,4.082)--(5.333,4.087)--(5.339,4.092)--(5.344,4.097)%
  --(5.350,4.102)--(5.355,4.107)--(5.361,4.112)--(5.366,4.117)--(5.372,4.122)--(5.378,4.127)%
  --(5.383,4.132)--(5.389,4.137)--(5.394,4.142)--(5.400,4.147)--(5.405,4.152)--(5.411,4.157)%
  --(5.417,4.162)--(5.422,4.167)--(5.428,4.172)--(5.433,4.177)--(5.439,4.182)--(5.444,4.187)%
  --(5.450,4.192)--(5.456,4.197)--(5.461,4.202)--(5.467,4.207)--(5.472,4.212)--(5.478,4.217)%
  --(5.483,4.222)--(5.489,4.227)--(5.495,4.232)--(5.500,4.237)--(5.506,4.242)--(5.511,4.247)%
  --(5.517,4.252)--(5.522,4.257)--(5.528,4.262)--(5.534,4.267)--(5.539,4.273)--(5.545,4.278)%
  --(5.550,4.283)--(5.556,4.288)--(5.561,4.293)--(5.567,4.298)--(5.573,4.303)--(5.578,4.308)%
  --(5.584,4.313)--(5.589,4.318)--(5.595,4.323)--(5.600,4.328)--(5.606,4.334)--(5.612,4.339)%
  --(5.617,4.344)--(5.623,4.349)--(5.628,4.354)--(5.634,4.359)--(5.639,4.364)--(5.645,4.369)%
  --(5.651,4.374)--(5.656,4.380)--(5.662,4.385)--(5.667,4.390)--(5.673,4.395)--(5.678,4.400)%
  --(5.684,4.405)--(5.690,4.410)--(5.695,4.415)--(5.701,4.421)--(5.706,4.426)--(5.712,4.431)%
  --(5.717,4.436)--(5.723,4.441)--(5.729,4.446)--(5.734,4.451)--(5.740,4.457)--(5.745,4.462)%
  --(5.751,4.467)--(5.756,4.472)--(5.762,4.477)--(5.768,4.483)--(5.773,4.488)--(5.779,4.493)%
  --(5.784,4.498)--(5.790,4.503)--(5.795,4.508)--(5.801,4.514)--(5.807,4.519)--(5.812,4.524)%
  --(5.818,4.529)--(5.823,4.534)--(5.829,4.540)--(5.834,4.545)--(5.840,4.550)--(5.846,4.555)%
  --(5.851,4.561)--(5.857,4.566)--(5.862,4.571)--(5.868,4.576)--(5.873,4.581)--(5.879,4.587)%
  --(5.885,4.592)--(5.890,4.597)--(5.896,4.602)--(5.901,4.608)--(5.907,4.613)--(5.913,4.618)%
  --(5.918,4.623)--(5.924,4.629)--(5.929,4.634)--(5.935,4.639)--(5.940,4.644)--(5.946,4.650)%
  --(5.952,4.655)--(5.957,4.660)--(5.963,4.666)--(5.968,4.671)--(5.974,4.676)--(5.979,4.681)%
  --(5.985,4.687)--(5.991,4.692)--(5.996,4.697)--(6.002,4.703)--(6.007,4.708)--(6.013,4.713)%
  --(6.018,4.719)--(6.024,4.724)--(6.030,4.729)--(6.035,4.735)--(6.041,4.740)--(6.046,4.745)%
  --(6.052,4.751)--(6.057,4.756)--(6.063,4.761)--(6.069,4.767)--(6.074,4.772)--(6.080,4.777)%
  --(6.085,4.783)--(6.091,4.788)--(6.096,4.793)--(6.102,4.799)--(6.108,4.804)--(6.113,4.809)%
  --(6.119,4.815)--(6.124,4.820)--(6.130,4.826)--(6.135,4.831)--(6.141,4.836)--(6.147,4.842)%
  --(6.152,4.847)--(6.158,4.852)--(6.163,4.858)--(6.169,4.863)--(6.174,4.869)--(6.180,4.874)%
  --(6.186,4.879)--(6.191,4.885)--(6.197,4.890)--(6.202,4.896)--(6.208,4.901)--(6.213,4.907)%
  --(6.219,4.912)--(6.225,4.917)--(6.230,4.923)--(6.236,4.928)--(6.241,4.934)--(6.247,4.939)%
  --(6.252,4.945)--(6.258,4.950)--(6.264,4.956)--(6.269,4.961)--(6.275,4.966)--(6.280,4.972)%
  --(6.286,4.977)--(6.291,4.983)--(6.297,4.988)--(6.303,4.994)--(6.308,4.999)--(6.314,5.005)%
  --(6.319,5.010)--(6.325,5.016)--(6.330,5.021)--(6.336,5.027)--(6.342,5.032)--(6.347,5.038)%
  --(6.353,5.043)--(6.358,5.049)--(6.364,5.054)--(6.369,5.060)--(6.375,5.065)--(6.381,5.071)%
  --(6.386,5.076)--(6.392,5.082)--(6.397,5.087)--(6.403,5.093)--(6.408,5.098)--(6.414,5.104)%
  --(6.420,5.109)--(6.425,5.115)--(6.431,5.120)--(6.436,5.126)--(6.442,5.131)--(6.448,5.137)%
  --(6.453,5.142)--(6.459,5.148)--(6.464,5.153)--(6.470,5.159)--(6.475,5.164)--(6.481,5.170)%
  --(6.487,5.176)--(6.492,5.181)--(6.498,5.187)--(6.503,5.192)--(6.509,5.198)--(6.514,5.203)%
  --(6.520,5.209)--(6.526,5.214)--(6.531,5.220)--(6.537,5.225)--(6.542,5.231)--(6.548,5.236)%
  --(6.553,5.242)--(6.559,5.247)--(6.565,5.253)--(6.570,5.258)--(6.576,5.264)--(6.581,5.269)%
  --(6.587,5.275)--(6.592,5.280)--(6.598,5.286)--(6.604,5.292)--(6.609,5.297)--(6.615,5.302)%
  --(6.620,5.308)--(6.626,5.313)--(6.631,5.319)--(6.637,5.324)--(6.643,5.330)--(6.648,5.335)%
  --(6.654,5.341)--(6.659,5.346)--(6.665,5.352)--(6.670,5.357)--(6.676,5.363)--(6.682,5.368)%
  --(6.687,5.373)--(6.693,5.379)--(6.698,5.384)--(6.704,5.390)--(6.709,5.395)--(6.715,5.400)%
  --(6.721,5.406)--(6.726,5.411)--(6.732,5.416)--(6.737,5.422)--(6.743,5.427)--(6.748,5.432)%
  --(6.754,5.438)--(6.760,5.443)--(6.765,5.448)--(6.771,5.453)--(6.776,5.459)--(6.782,5.464)%
  --(6.787,5.469)--(6.793,5.474)--(6.799,5.479)--(6.804,5.484)--(6.810,5.489)--(6.815,5.494)%
  --(6.821,5.499)--(6.826,5.504)--(6.832,5.509)--(6.838,5.514)--(6.843,5.519)--(6.849,5.524)%
  --(6.854,5.529)--(6.860,5.533)--(6.865,5.538)--(6.871,5.543)--(6.877,5.548)--(6.882,5.552)%
  --(6.888,5.557)--(6.893,5.561)--(6.899,5.565)--(6.904,5.570)--(6.910,5.574)--(6.916,5.578)%
  --(6.921,5.582)--(6.927,5.586)--(6.932,5.590)--(6.938,5.594)--(6.943,5.598)--(6.949,5.602)%
  --(6.955,5.605)--(6.960,5.609)--(6.966,5.612)--(6.971,5.615)--(6.977,5.618)--(6.982,5.621)%
  --(6.988,5.623)--(6.994,5.626)--(6.999,5.628)--(7.005,5.629)--(7.010,5.630)--(7.016,5.630);
\gpcolor{\gprgb{0}{0}{1000}}
\draw[gp path] (1.012,0.658)--(1.816,0.658)--(1.816,0.700)--(1.816,0.741)--(1.816,0.783)%
  --(1.816,0.825)--(1.816,0.867)--(1.816,0.909)--(1.816,0.950)--(1.816,0.992)--(1.816,1.034)%
  --(1.816,1.076)--(1.816,1.118)--(1.816,1.159)--(2.621,1.159)--(2.621,1.201)--(2.621,1.243)%
  --(2.621,1.285)--(2.621,1.326)--(2.621,1.368)--(2.621,1.410)--(2.621,1.452)--(2.621,1.494)%
  --(2.621,1.535)--(2.621,1.577)--(2.621,1.619)--(2.621,1.661)--(2.621,1.703)--(2.621,1.744)%
  --(2.621,1.786)--(2.621,1.828)--(2.621,1.870)--(2.621,1.912)--(2.621,1.953)--(2.621,1.995)%
  --(2.621,2.037)--(2.621,2.079)--(2.621,2.121)--(2.621,2.162)--(2.621,2.204)--(2.621,2.246)%
  --(2.621,2.288)--(2.621,2.329)--(3.425,2.329)--(3.425,2.371)--(3.425,2.413)--(3.425,2.455)%
  --(3.425,2.497)--(4.229,2.497)--(4.229,2.538)--(4.229,2.580)--(4.229,2.622)--(4.229,2.664)%
  --(4.229,2.706)--(4.229,2.747)--(4.229,2.789)--(4.229,2.831)--(4.229,2.873)--(4.229,2.915)%
  --(4.230,2.915)--(4.230,2.956)--(4.230,2.998)--(4.230,3.040)--(4.230,3.082)--(4.230,3.124)%
  --(4.230,3.165)--(4.230,3.207)--(4.230,3.249)--(4.230,3.291)--(4.230,3.332)--(4.230,3.374)%
  --(4.230,3.416)--(4.230,3.458)--(4.230,3.500)--(4.230,3.541)--(4.230,3.583)--(4.230,3.625)%
  --(4.230,3.667)--(4.230,3.709)--(4.230,3.750)--(5.034,3.750)--(5.034,3.792)--(5.034,3.834)%
  --(5.034,3.876)--(5.034,3.918)--(5.838,3.918)--(5.838,3.959)--(5.838,4.001)--(5.838,4.043)%
  --(5.838,4.085)--(5.838,4.127)--(5.838,4.168)--(5.838,4.210)--(5.838,4.252)--(5.838,4.294)%
  --(5.838,4.335)--(5.838,4.377)--(5.838,4.419)--(5.838,4.461)--(5.838,4.503)--(5.838,4.544)%
  --(5.838,4.586)--(5.838,4.628)--(5.838,4.670)--(5.838,4.712)--(5.838,4.753)--(5.838,4.795)%
  --(5.838,4.837)--(5.838,4.879)--(5.838,4.921)--(5.838,4.962)--(5.838,5.004)--(5.838,5.046)%
  --(5.838,5.088)--(6.643,5.088)--(6.643,5.130)--(6.643,5.171)--(6.643,5.213)--(6.643,5.255)%
  --(6.643,5.297)--(6.643,5.338)--(6.643,5.380)--(6.643,5.422)--(6.643,5.464)--(6.643,5.506)%
  --(6.643,5.547)--(6.643,5.589)--(7.447,5.589)--(7.447,5.631);
\gpdefrectangularnode{gp plot 1}{\pgfpoint{1.012cm}{0.616cm}}{\pgfpoint{7.447cm}{5.631cm}}
\end{tikzpicture}

%% file: ecf720.tex
\begin{tikzpicture}[gnuplot,scale=.7]
\gpfill{color=\gprgb{1000}{1000}{1000}} (1.012,0.616)--(7.447,0.616)--(7.447,5.631)--(1.012,5.631)--cycle;
\gpcolor{gp lt color border}
\gpsetlinetype{gp lt border}
\gpsetlinewidth{1.00}
\draw[gp path] (1.012,0.616)--(1.012,5.631)--(7.447,5.631)--(7.447,0.616)--cycle;
\gpsetlinewidth{0.50}
\draw[gp path] (1.012,0.616)--(1.263,0.616);
\draw[gp path] (7.447,0.616)--(7.196,0.616);
\gpcolor{\gprgb{0}{0}{0}}
\node[gp node right,font=\gpfontsize{10.00pt}{12.00pt}] at (0.828,0.616) {0};
\gpcolor{gp lt color border}
\draw[gp path] (1.012,1.619)--(1.263,1.619);
\draw[gp path] (7.447,1.619)--(7.196,1.619);
\gpcolor{\gprgb{0}{0}{0}}
\node[gp node right,font=\gpfontsize{10.00pt}{12.00pt}] at (0.828,1.619) {0.2};
\gpcolor{gp lt color border}
\draw[gp path] (1.012,2.622)--(1.263,2.622);
\draw[gp path] (7.447,2.622)--(7.196,2.622);
\gpcolor{\gprgb{0}{0}{0}}
\node[gp node right,font=\gpfontsize{10.00pt}{12.00pt}] at (0.828,2.622) {0.4};
\gpcolor{gp lt color border}
\draw[gp path] (1.012,3.625)--(1.263,3.625);
\draw[gp path] (7.447,3.625)--(7.196,3.625);
\gpcolor{\gprgb{0}{0}{0}}
\node[gp node right,font=\gpfontsize{10.00pt}{12.00pt}] at (0.828,3.625) {0.6};
\gpcolor{gp lt color border}
\draw[gp path] (1.012,4.628)--(1.263,4.628);
\draw[gp path] (7.447,4.628)--(7.196,4.628);
\gpcolor{\gprgb{0}{0}{0}}
\node[gp node right,font=\gpfontsize{10.00pt}{12.00pt}] at (0.828,4.628) {0.8};
\gpcolor{gp lt color border}
\draw[gp path] (1.012,5.631)--(1.263,5.631);
\draw[gp path] (7.447,5.631)--(7.196,5.631);
\gpcolor{\gprgb{0}{0}{0}}
\node[gp node right,font=\gpfontsize{10.00pt}{12.00pt}] at (0.828,5.631) {1};
\gpcolor{gp lt color border}
\draw[gp path] (1.012,0.616)--(1.012,0.867);
\draw[gp path] (1.012,5.631)--(1.012,5.380);
\gpcolor{\gprgb{0}{0}{0}}
\node[gp node center,font=\gpfontsize{10.00pt}{12.00pt}] at (1.012,0.308) {-4};
\gpcolor{gp lt color border}
\draw[gp path] (1.816,0.616)--(1.816,0.867);
\draw[gp path] (1.816,5.631)--(1.816,5.380);
\gpcolor{\gprgb{0}{0}{0}}
\node[gp node center,font=\gpfontsize{10.00pt}{12.00pt}] at (1.816,0.308) {-3};
\gpcolor{gp lt color border}
\draw[gp path] (2.621,0.616)--(2.621,0.867);
\draw[gp path] (2.621,5.631)--(2.621,5.380);
\gpcolor{\gprgb{0}{0}{0}}
\node[gp node center,font=\gpfontsize{10.00pt}{12.00pt}] at (2.621,0.308) {-2};
\gpcolor{gp lt color border}
\draw[gp path] (3.425,0.616)--(3.425,0.867);
\draw[gp path] (3.425,5.631)--(3.425,5.380);
\gpcolor{\gprgb{0}{0}{0}}
\node[gp node center,font=\gpfontsize{10.00pt}{12.00pt}] at (3.425,0.308) {-1};
\gpcolor{gp lt color border}
\draw[gp path] (4.230,0.616)--(4.230,0.867);
\draw[gp path] (4.230,5.631)--(4.230,5.380);
\gpcolor{\gprgb{0}{0}{0}}
\node[gp node center,font=\gpfontsize{10.00pt}{12.00pt}] at (4.230,0.308) {0};
\gpcolor{gp lt color border}
\draw[gp path] (5.034,0.616)--(5.034,0.867);
\draw[gp path] (5.034,5.631)--(5.034,5.380);
\gpcolor{\gprgb{0}{0}{0}}
\node[gp node center,font=\gpfontsize{10.00pt}{12.00pt}] at (5.034,0.308) {1};
\gpcolor{gp lt color border}
\draw[gp path] (5.838,0.616)--(5.838,0.867);
\draw[gp path] (5.838,5.631)--(5.838,5.380);
\gpcolor{\gprgb{0}{0}{0}}
\node[gp node center,font=\gpfontsize{10.00pt}{12.00pt}] at (5.838,0.308) {2};
\gpcolor{gp lt color border}
\draw[gp path] (6.643,0.616)--(6.643,0.867);
\draw[gp path] (6.643,5.631)--(6.643,5.380);
\gpcolor{\gprgb{0}{0}{0}}
\node[gp node center,font=\gpfontsize{10.00pt}{12.00pt}] at (6.643,0.308) {3};
\gpcolor{gp lt color border}
\draw[gp path] (7.447,0.616)--(7.447,0.867);
\draw[gp path] (7.447,5.631)--(7.447,5.380);
\gpcolor{\gprgb{0}{0}{0}}
\node[gp node center,font=\gpfontsize{10.00pt}{12.00pt}] at (7.447,0.308) {4};
\gpcolor{gp lt color border}
\draw[gp path] (1.012,5.631)--(1.012,0.616)--(7.447,0.616)--(7.447,5.631)--cycle;
\gpcolor{\gprgb{0}{1000}{0}}
\gpsetlinetype{gp lt plot 0}
\draw[gp path] (1.443,0.616)--(1.449,0.789)--(1.454,0.857)--(1.460,0.908)--(1.465,0.949)%
  --(1.471,0.984)--(1.477,1.015)--(1.482,1.042)--(1.488,1.066)--(1.493,1.088)--(1.499,1.108)%
  --(1.504,1.126)--(1.510,1.143)--(1.516,1.159)--(1.521,1.173)--(1.527,1.187)--(1.532,1.200)%
  --(1.538,1.211)--(1.543,1.223)--(1.549,1.233)--(1.555,1.243)--(1.560,1.252)--(1.566,1.261)%
  --(1.571,1.269)--(1.577,1.277)--(1.582,1.284)--(1.588,1.291)--(1.594,1.297)--(1.599,1.304)%
  --(1.605,1.310)--(1.610,1.315)--(1.616,1.320)--(1.621,1.326)--(1.627,1.330)--(1.633,1.335)%
  --(1.638,1.339)--(1.644,1.343)--(1.649,1.347)--(1.655,1.351)--(1.660,1.354)--(1.666,1.358)%
  --(1.672,1.361)--(1.677,1.364)--(1.683,1.367)--(1.688,1.370)--(1.694,1.372)--(1.699,1.375)%
  --(1.705,1.377)--(1.711,1.379)--(1.716,1.382)--(1.722,1.384)--(1.727,1.386)--(1.733,1.387)%
  --(1.738,1.389)--(1.744,1.391)--(1.750,1.392)--(1.755,1.394)--(1.761,1.395)--(1.766,1.397)%
  --(1.772,1.398)--(1.777,1.399)--(1.783,1.400)--(1.789,1.401)--(1.794,1.402)--(1.800,1.403)%
  --(1.805,1.404)--(1.811,1.405)--(1.816,1.406)--(1.822,1.407)--(1.828,1.407)--(1.833,1.408)%
  --(1.839,1.409)--(1.844,1.409)--(1.850,1.410)--(1.855,1.410)--(1.861,1.411)--(1.867,1.411)%
  --(1.872,1.412)--(1.878,1.412)--(1.883,1.412)--(1.889,1.413)--(1.894,1.413)--(1.900,1.413)%
  --(1.906,1.413)--(1.911,1.414)--(1.917,1.414)--(1.922,1.414)--(1.928,1.414)--(1.933,1.414)%
  --(1.939,1.414)--(1.945,1.414)--(1.950,1.414)--(1.956,1.414)--(1.961,1.414)--(1.967,1.414)%
  --(1.972,1.414)--(1.978,1.414)--(1.984,1.414)--(1.989,1.414)--(1.995,1.414)--(2.000,1.414)%
  --(2.006,1.413)--(2.011,1.413)--(2.017,1.413)--(2.023,1.413)--(2.028,1.413)--(2.034,1.413)%
  --(2.039,1.412)--(2.045,1.412)--(2.051,1.412)--(2.056,1.412)--(2.062,1.411)--(2.067,1.411)%
  --(2.073,1.411)--(2.078,1.410)--(2.084,1.410)--(2.090,1.410)--(2.095,1.410)--(2.101,1.409)%
  --(2.106,1.409)--(2.112,1.409)--(2.117,1.408)--(2.123,1.408)--(2.129,1.408)--(2.134,1.407)%
  --(2.140,1.407)--(2.145,1.406)--(2.151,1.406)--(2.156,1.406)--(2.162,1.405)--(2.168,1.405)%
  --(2.173,1.405)--(2.179,1.404)--(2.184,1.404)--(2.190,1.403)--(2.195,1.403)--(2.201,1.403)%
  --(2.207,1.402)--(2.212,1.402)--(2.218,1.401)--(2.223,1.401)--(2.229,1.400)--(2.234,1.400)%
  --(2.240,1.400)--(2.246,1.399)--(2.251,1.399)--(2.257,1.398)--(2.262,1.398)--(2.268,1.397)%
  --(2.273,1.397)--(2.279,1.396)--(2.285,1.396)--(2.290,1.396)--(2.296,1.395)--(2.301,1.395)%
  --(2.307,1.394)--(2.312,1.394)--(2.318,1.393)--(2.324,1.393)--(2.329,1.392)--(2.335,1.392)%
  --(2.340,1.391)--(2.346,1.391)--(2.351,1.391)--(2.357,1.390)--(2.363,1.390)--(2.368,1.389)%
  --(2.374,1.389)--(2.379,1.388)--(2.385,1.388)--(2.390,1.387)--(2.396,1.387)--(2.402,1.386)%
  --(2.407,1.386)--(2.413,1.385)--(2.418,1.385)--(2.424,1.384)--(2.429,1.384)--(2.435,1.384)%
  --(2.441,1.383)--(2.446,1.383)--(2.452,1.382)--(2.457,1.382)--(2.463,1.381)--(2.468,1.381)%
  --(2.474,1.380)--(2.480,1.380)--(2.485,1.379)--(2.491,1.379)--(2.496,1.378)--(2.502,1.378)%
  --(2.507,1.378)--(2.513,1.377)--(2.519,1.377)--(2.524,1.376)--(2.530,1.376)--(2.535,1.375)%
  --(2.541,1.375)--(2.546,1.374)--(2.552,1.374)--(2.558,1.374)--(2.563,1.373)--(2.569,1.373)%
  --(2.574,1.372)--(2.580,1.372)--(2.586,1.371)--(2.591,1.371)--(2.597,1.370)--(2.602,1.370)%
  --(2.608,1.370)--(2.613,1.369)--(2.619,1.369)--(2.625,1.368)--(2.630,1.368)--(2.636,1.367)%
  --(2.641,1.367)--(2.647,1.366)--(2.652,1.366)--(2.658,1.366)--(2.664,1.365)--(2.669,1.365)%
  --(2.675,1.364)--(2.680,1.364)--(2.686,1.364)--(2.691,1.363)--(2.697,1.363)--(2.703,1.362)%
  --(2.708,1.362)--(2.714,1.361)--(2.719,1.361)--(2.725,1.361)--(2.730,1.360)--(2.736,1.360)%
  --(2.742,1.359)--(2.747,1.359)--(2.753,1.359)--(2.758,1.358)--(2.764,1.358)--(2.769,1.357)%
  --(2.775,1.357)--(2.781,1.357)--(2.786,1.356)--(2.792,1.356)--(2.797,1.355)--(2.803,1.355)%
  --(2.808,1.355)--(2.814,1.354)--(2.820,1.354)--(2.825,1.354)--(2.831,1.353)--(2.836,1.353)%
  --(2.842,1.352)--(2.847,1.352)--(2.853,1.352)--(2.859,1.351)--(2.864,1.351)--(2.870,1.351)%
  --(2.875,1.350)--(2.881,1.350)--(2.886,1.349)--(2.892,1.349)--(2.898,1.349)--(2.903,1.348)%
  --(2.909,1.348)--(2.914,1.348)--(2.920,1.347)--(2.925,1.347)--(2.931,1.347)--(2.937,1.346)%
  --(2.942,1.346)--(2.948,1.346)--(2.953,1.345)--(2.959,1.345)--(2.964,1.345)--(2.970,1.344)%
  --(2.976,1.344)--(2.981,1.343)--(2.987,1.343)--(2.992,1.343)--(2.998,1.342)--(3.003,1.342)%
  --(3.009,1.342)--(3.015,1.342)--(3.020,1.341)--(3.026,1.341)--(3.031,1.341)--(3.037,1.340)%
  --(3.042,1.340)--(3.048,1.340)--(3.054,1.339)--(3.059,1.339)--(3.065,1.339)--(3.070,1.338)%
  --(3.076,1.338)--(3.081,1.338)--(3.087,1.337)--(3.093,1.337)--(3.098,1.337)--(3.104,1.337)%
  --(3.109,1.336)--(3.115,1.336)--(3.120,1.336)--(3.126,1.335)--(3.132,1.335)--(3.137,1.335)%
  --(3.143,1.334)--(3.148,1.334)--(3.154,1.334)--(3.160,1.334)--(3.165,1.333)--(3.171,1.333)%
  --(3.176,1.333)--(3.182,1.332)--(3.187,1.332)--(3.193,1.332)--(3.199,1.332)--(3.204,1.331)%
  --(3.210,1.331)--(3.215,1.331)--(3.221,1.331)--(3.226,1.330)--(3.232,1.330)--(3.238,1.330)%
  --(3.243,1.330)--(3.249,1.329)--(3.254,1.329)--(3.260,1.329)--(3.265,1.329)--(3.271,1.328)%
  --(3.277,1.328)--(3.282,1.328)--(3.288,1.328)--(3.293,1.327)--(3.299,1.327)--(3.304,1.327)%
  --(3.310,1.327)--(3.316,1.326)--(3.321,1.326)--(3.327,1.326)--(3.332,1.326)--(3.338,1.325)%
  --(3.343,1.325)--(3.349,1.325)--(3.355,1.325)--(3.360,1.324)--(3.366,1.324)--(3.371,1.324)%
  --(3.377,1.324)--(3.382,1.324)--(3.388,1.323)--(3.394,1.323)--(3.399,1.323)--(3.405,1.323)%
  --(3.410,1.322)--(3.416,1.322)--(3.421,1.322)--(3.427,1.322)--(3.433,1.322)--(3.438,1.321)%
  --(3.444,1.321)--(3.449,1.321)--(3.455,1.321)--(3.460,1.321)--(3.466,1.320)--(3.472,1.320)%
  --(3.477,1.320)--(3.483,1.320)--(3.488,1.320)--(3.494,1.319)--(3.499,1.319)--(3.505,1.319)%
  --(3.511,1.319)--(3.516,1.319)--(3.522,1.319)--(3.527,1.318)--(3.533,1.318)--(3.538,1.318)%
  --(3.544,1.318)--(3.550,1.318)--(3.555,1.317)--(3.561,1.317)--(3.566,1.317)--(3.572,1.317)%
  --(3.577,1.317)--(3.583,1.317)--(3.589,1.316)--(3.594,1.316)--(3.600,1.316)--(3.605,1.316)%
  --(3.611,1.316)--(3.616,1.316)--(3.622,1.316)--(3.628,1.315)--(3.633,1.315)--(3.639,1.315)%
  --(3.644,1.315)--(3.650,1.315)--(3.655,1.315)--(3.661,1.314)--(3.667,1.314)--(3.672,1.314)%
  --(3.678,1.314)--(3.683,1.314)--(3.689,1.314)--(3.695,1.314)--(3.700,1.314)--(3.706,1.313)%
  --(3.711,1.313)--(3.717,1.313)--(3.722,1.313)--(3.728,1.313)--(3.734,1.313)--(3.739,1.313)%
  --(3.745,1.312)--(3.750,1.312)--(3.756,1.312)--(3.761,1.312)--(3.767,1.312)--(3.773,1.312)%
  --(3.778,1.312)--(3.784,1.312)--(3.789,1.312)--(3.795,1.311)--(3.800,1.311)--(3.806,1.311)%
  --(3.812,1.311)--(3.817,1.311)--(3.823,1.311)--(3.828,1.311)--(3.834,1.311)--(3.839,1.311)%
  --(3.845,1.311)--(3.851,1.310)--(3.856,1.310)--(3.862,1.310)--(3.867,1.310)--(3.873,1.310)%
  --(3.878,1.310)--(3.884,1.310)--(3.890,1.310)--(3.895,1.310)--(3.901,1.310)--(3.906,1.310)%
  --(3.912,1.309)--(3.917,1.309)--(3.923,1.309)--(3.929,1.309)--(3.934,1.309)--(3.940,1.309)%
  --(3.945,1.309)--(3.951,1.309)--(3.956,1.309)--(3.962,1.309)--(3.968,1.309)--(3.973,1.309)%
  --(3.979,1.309)--(3.984,1.309)--(3.990,1.309)--(3.995,1.308)--(4.001,1.308)--(4.007,1.308)%
  --(4.012,1.308)--(4.018,1.308)--(4.023,1.308)--(4.029,1.308)--(4.034,1.308)--(4.040,1.308)%
  --(4.046,1.308)--(4.051,1.308)--(4.057,1.308)--(4.062,1.308)--(4.068,1.308)--(4.073,1.308)%
  --(4.079,1.308)--(4.085,1.308)--(4.090,1.308)--(4.096,1.308)--(4.101,1.308)--(4.107,1.308)%
  --(4.112,1.308)--(4.118,1.308)--(4.124,1.307)--(4.129,1.307)--(4.135,1.307)--(4.140,1.307)%
  --(4.146,1.307)--(4.151,1.307)--(4.157,1.307)--(4.163,1.307)--(4.168,1.307)--(4.174,1.307)%
  --(4.179,1.307)--(4.185,1.307)--(4.190,1.307)--(4.196,1.307)--(4.202,1.307)--(4.207,1.307)%
  --(4.213,1.307)--(4.218,1.307)--(4.224,1.307)--(4.230,1.307)--(4.235,1.307)--(4.241,1.307)%
  --(4.246,1.307)--(4.252,1.307)--(4.257,1.307)--(4.263,1.307)--(4.269,1.307)--(4.274,1.307)%
  --(4.280,1.307)--(4.285,1.307)--(4.291,1.307)--(4.296,1.307)--(4.302,1.307)--(4.308,1.307)%
  --(4.313,1.307)--(4.319,1.307)--(4.324,1.307)--(4.330,1.307)--(4.335,1.307)--(4.341,1.308)%
  --(4.347,1.308)--(4.352,1.308)--(4.358,1.308)--(4.363,1.308)--(4.369,1.308)--(4.374,1.308)%
  --(4.380,1.308)--(4.386,1.308)--(4.391,1.308)--(4.397,1.308)--(4.402,1.308)--(4.408,1.308)%
  --(4.413,1.308)--(4.419,1.308)--(4.425,1.308)--(4.430,1.308)--(4.436,1.308)--(4.441,1.308)%
  --(4.447,1.308)--(4.452,1.308)--(4.458,1.308)--(4.464,1.308)--(4.469,1.309)--(4.475,1.309)%
  --(4.480,1.309)--(4.486,1.309)--(4.491,1.309)--(4.497,1.309)--(4.503,1.309)--(4.508,1.309)%
  --(4.514,1.309)--(4.519,1.309)--(4.525,1.309)--(4.530,1.309)--(4.536,1.309)--(4.542,1.309)%
  --(4.547,1.309)--(4.553,1.310)--(4.558,1.310)--(4.564,1.310)--(4.569,1.310)--(4.575,1.310)%
  --(4.581,1.310)--(4.586,1.310)--(4.592,1.310)--(4.597,1.310)--(4.603,1.310)--(4.608,1.310)%
  --(4.614,1.311)--(4.620,1.311)--(4.625,1.311)--(4.631,1.311)--(4.636,1.311)--(4.642,1.311)%
  --(4.647,1.311)--(4.653,1.311)--(4.659,1.311)--(4.664,1.311)--(4.670,1.312)--(4.675,1.312)%
  --(4.681,1.312)--(4.686,1.312)--(4.692,1.312)--(4.698,1.312)--(4.703,1.312)--(4.709,1.312)%
  --(4.714,1.312)--(4.720,1.313)--(4.725,1.313)--(4.731,1.313)--(4.737,1.313)--(4.742,1.313)%
  --(4.748,1.313)--(4.753,1.313)--(4.759,1.314)--(4.764,1.314)--(4.770,1.314)--(4.776,1.314)%
  --(4.781,1.314)--(4.787,1.314)--(4.792,1.314)--(4.798,1.314)--(4.804,1.315)--(4.809,1.315)%
  --(4.815,1.315)--(4.820,1.315)--(4.826,1.315)--(4.831,1.315)--(4.837,1.316)--(4.843,1.316)%
  --(4.848,1.316)--(4.854,1.316)--(4.859,1.316)--(4.865,1.316)--(4.870,1.316)--(4.876,1.317)%
  --(4.882,1.317)--(4.887,1.317)--(4.893,1.317)--(4.898,1.317)--(4.904,1.317)--(4.909,1.318)%
  --(4.915,1.318)--(4.921,1.318)--(4.926,1.318)--(4.932,1.318)--(4.937,1.319)--(4.943,1.319)%
  --(4.948,1.319)--(4.954,1.319)--(4.960,1.319)--(4.965,1.319)--(4.971,1.320)--(4.976,1.320)%
  --(4.982,1.320)--(4.987,1.320)--(4.993,1.320)--(4.999,1.321)--(5.004,1.321)--(5.010,1.321)%
  --(5.015,1.321)--(5.021,1.321)--(5.026,1.322)--(5.032,1.322)--(5.038,1.322)--(5.043,1.322)%
  --(5.049,1.322)--(5.054,1.323)--(5.060,1.323)--(5.065,1.323)--(5.071,1.323)--(5.077,1.324)%
  --(5.082,1.324)--(5.088,1.324)--(5.093,1.324)--(5.099,1.324)--(5.104,1.325)--(5.110,1.325)%
  --(5.116,1.325)--(5.121,1.325)--(5.127,1.326)--(5.132,1.326)--(5.138,1.326)--(5.143,1.326)%
  --(5.149,1.327)--(5.155,1.327)--(5.160,1.327)--(5.166,1.327)--(5.171,1.328)--(5.177,1.328)%
  --(5.182,1.328)--(5.188,1.328)--(5.194,1.329)--(5.199,1.329)--(5.205,1.329)--(5.210,1.329)%
  --(5.216,1.330)--(5.221,1.330)--(5.227,1.330)--(5.233,1.330)--(5.238,1.331)--(5.244,1.331)%
  --(5.249,1.331)--(5.255,1.331)--(5.260,1.332)--(5.266,1.332)--(5.272,1.332)--(5.277,1.332)%
  --(5.283,1.333)--(5.288,1.333)--(5.294,1.333)--(5.299,1.334)--(5.305,1.334)--(5.311,1.334)%
  --(5.316,1.334)--(5.322,1.335)--(5.327,1.335)--(5.333,1.335)--(5.339,1.336)--(5.344,1.336)%
  --(5.350,1.336)--(5.355,1.337)--(5.361,1.337)--(5.366,1.337)--(5.372,1.337)--(5.378,1.338)%
  --(5.383,1.338)--(5.389,1.338)--(5.394,1.339)--(5.400,1.339)--(5.405,1.339)--(5.411,1.340)%
  --(5.417,1.340)--(5.422,1.340)--(5.428,1.341)--(5.433,1.341)--(5.439,1.341)--(5.444,1.342)%
  --(5.450,1.342)--(5.456,1.342)--(5.461,1.342)--(5.467,1.343)--(5.472,1.343)--(5.478,1.343)%
  --(5.483,1.344)--(5.489,1.344)--(5.495,1.345)--(5.500,1.345)--(5.506,1.345)--(5.511,1.346)%
  --(5.517,1.346)--(5.522,1.346)--(5.528,1.347)--(5.534,1.347)--(5.539,1.347)--(5.545,1.348)%
  --(5.550,1.348)--(5.556,1.348)--(5.561,1.349)--(5.567,1.349)--(5.573,1.349)--(5.578,1.350)%
  --(5.584,1.350)--(5.589,1.351)--(5.595,1.351)--(5.600,1.351)--(5.606,1.352)--(5.612,1.352)%
  --(5.617,1.352)--(5.623,1.353)--(5.628,1.353)--(5.634,1.354)--(5.639,1.354)--(5.645,1.354)%
  --(5.651,1.355)--(5.656,1.355)--(5.662,1.355)--(5.667,1.356)--(5.673,1.356)--(5.678,1.357)%
  --(5.684,1.357)--(5.690,1.357)--(5.695,1.358)--(5.701,1.358)--(5.706,1.359)--(5.712,1.359)%
  --(5.717,1.359)--(5.723,1.360)--(5.729,1.360)--(5.734,1.361)--(5.740,1.361)--(5.745,1.361)%
  --(5.751,1.362)--(5.756,1.362)--(5.762,1.363)--(5.768,1.363)--(5.773,1.364)--(5.779,1.364)%
  --(5.784,1.364)--(5.790,1.365)--(5.795,1.365)--(5.801,1.366)--(5.807,1.366)--(5.812,1.366)%
  --(5.818,1.367)--(5.823,1.367)--(5.829,1.368)--(5.834,1.368)--(5.840,1.369)--(5.846,1.369)%
  --(5.851,1.370)--(5.857,1.370)--(5.862,1.370)--(5.868,1.371)--(5.873,1.371)--(5.879,1.372)%
  --(5.885,1.372)--(5.890,1.373)--(5.896,1.373)--(5.901,1.374)--(5.907,1.374)--(5.913,1.374)%
  --(5.918,1.375)--(5.924,1.375)--(5.929,1.376)--(5.935,1.376)--(5.940,1.377)--(5.946,1.377)%
  --(5.952,1.378)--(5.957,1.378)--(5.963,1.378)--(5.968,1.379)--(5.974,1.379)--(5.979,1.380)%
  --(5.985,1.380)--(5.991,1.381)--(5.996,1.381)--(6.002,1.382)--(6.007,1.382)--(6.013,1.383)%
  --(6.018,1.383)--(6.024,1.384)--(6.030,1.384)--(6.035,1.384)--(6.041,1.385)--(6.046,1.385)%
  --(6.052,1.386)--(6.057,1.386)--(6.063,1.387)--(6.069,1.387)--(6.074,1.388)--(6.080,1.388)%
  --(6.085,1.389)--(6.091,1.389)--(6.096,1.390)--(6.102,1.390)--(6.108,1.391)--(6.113,1.391)%
  --(6.119,1.391)--(6.124,1.392)--(6.130,1.392)--(6.135,1.393)--(6.141,1.393)--(6.147,1.394)%
  --(6.152,1.394)--(6.158,1.395)--(6.163,1.395)--(6.169,1.396)--(6.174,1.396)--(6.180,1.396)%
  --(6.186,1.397)--(6.191,1.397)--(6.197,1.398)--(6.202,1.398)--(6.208,1.399)--(6.213,1.399)%
  --(6.219,1.400)--(6.225,1.400)--(6.230,1.400)--(6.236,1.401)--(6.241,1.401)--(6.247,1.402)%
  --(6.252,1.402)--(6.258,1.403)--(6.264,1.403)--(6.269,1.403)--(6.275,1.404)--(6.280,1.404)%
  --(6.286,1.405)--(6.291,1.405)--(6.297,1.405)--(6.303,1.406)--(6.308,1.406)--(6.314,1.406)%
  --(6.319,1.407)--(6.325,1.407)--(6.330,1.408)--(6.336,1.408)--(6.342,1.408)--(6.347,1.409)%
  --(6.353,1.409)--(6.358,1.409)--(6.364,1.410)--(6.369,1.410)--(6.375,1.410)--(6.381,1.410)%
  --(6.386,1.411)--(6.392,1.411)--(6.397,1.411)--(6.403,1.412)--(6.408,1.412)--(6.414,1.412)%
  --(6.420,1.412)--(6.425,1.413)--(6.431,1.413)--(6.436,1.413)--(6.442,1.413)--(6.448,1.413)%
  --(6.453,1.413)--(6.459,1.414)--(6.464,1.414)--(6.470,1.414)--(6.475,1.414)--(6.481,1.414)%
  --(6.487,1.414)--(6.492,1.414)--(6.498,1.414)--(6.503,1.414)--(6.509,1.414)--(6.514,1.414)%
  --(6.520,1.414)--(6.526,1.414)--(6.531,1.414)--(6.537,1.414)--(6.542,1.414)--(6.548,1.414)%
  --(6.553,1.413)--(6.559,1.413)--(6.565,1.413)--(6.570,1.413)--(6.576,1.412)--(6.581,1.412)%
  --(6.587,1.412)--(6.592,1.411)--(6.598,1.411)--(6.604,1.410)--(6.609,1.410)--(6.615,1.409)%
  --(6.620,1.409)--(6.626,1.408)--(6.631,1.407)--(6.637,1.407)--(6.643,1.406)--(6.648,1.405)%
  --(6.654,1.404)--(6.659,1.403)--(6.665,1.402)--(6.670,1.401)--(6.676,1.400)--(6.682,1.399)%
  --(6.687,1.398)--(6.693,1.397)--(6.698,1.395)--(6.704,1.394)--(6.709,1.392)--(6.715,1.391)%
  --(6.721,1.389)--(6.726,1.387)--(6.732,1.386)--(6.737,1.384)--(6.743,1.382)--(6.748,1.379)%
  --(6.754,1.377)--(6.760,1.375)--(6.765,1.372)--(6.771,1.370)--(6.776,1.367)--(6.782,1.364)%
  --(6.787,1.361)--(6.793,1.358)--(6.799,1.354)--(6.804,1.351)--(6.810,1.347)--(6.815,1.343)%
  --(6.821,1.339)--(6.826,1.335)--(6.832,1.330)--(6.838,1.326)--(6.843,1.320)--(6.849,1.315)%
  --(6.854,1.310)--(6.860,1.304)--(6.865,1.297)--(6.871,1.291)--(6.877,1.284)--(6.882,1.277)%
  --(6.888,1.269)--(6.893,1.261)--(6.899,1.252)--(6.904,1.243)--(6.910,1.233)--(6.916,1.223)%
  --(6.921,1.211)--(6.927,1.200)--(6.932,1.187)--(6.938,1.173)--(6.943,1.159)--(6.949,1.143)%
  --(6.955,1.126)--(6.960,1.108)--(6.966,1.088)--(6.971,1.066)--(6.977,1.042)--(6.982,1.015)%
  --(6.988,0.984)--(6.994,0.949)--(6.999,0.908)--(7.005,0.857)--(7.010,0.789)--(7.016,0.616);
\gpcolor{\gprgb{1000}{0}{0}}
\draw[gp path] (1.443,0.616)--(1.449,0.617)--(1.454,0.619)--(1.460,0.621)--(1.465,0.623)%
  --(1.471,0.626)--(1.477,0.629)--(1.482,0.631)--(1.488,0.635)--(1.493,0.638)--(1.499,0.641)%
  --(1.504,0.645)--(1.510,0.648)--(1.516,0.652)--(1.521,0.656)--(1.527,0.660)--(1.532,0.664)%
  --(1.538,0.668)--(1.543,0.672)--(1.549,0.677)--(1.555,0.681)--(1.560,0.685)--(1.566,0.690)%
  --(1.571,0.694)--(1.577,0.699)--(1.582,0.704)--(1.588,0.708)--(1.594,0.713)--(1.599,0.718)%
  --(1.605,0.723)--(1.610,0.727)--(1.616,0.732)--(1.621,0.737)--(1.627,0.742)--(1.633,0.747)%
  --(1.638,0.752)--(1.644,0.757)--(1.649,0.762)--(1.655,0.767)--(1.660,0.772)--(1.666,0.778)%
  --(1.672,0.783)--(1.677,0.788)--(1.683,0.793)--(1.688,0.798)--(1.694,0.804)--(1.699,0.809)%
  --(1.705,0.814)--(1.711,0.819)--(1.716,0.825)--(1.722,0.830)--(1.727,0.835)--(1.733,0.841)%
  --(1.738,0.846)--(1.744,0.851)--(1.750,0.857)--(1.755,0.862)--(1.761,0.868)--(1.766,0.873)%
  --(1.772,0.878)--(1.777,0.884)--(1.783,0.889)--(1.789,0.895)--(1.794,0.900)--(1.800,0.906)%
  --(1.805,0.911)--(1.811,0.917)--(1.816,0.922)--(1.822,0.928)--(1.828,0.933)--(1.833,0.939)%
  --(1.839,0.944)--(1.844,0.949)--(1.850,0.955)--(1.855,0.960)--(1.861,0.966)--(1.867,0.972)%
  --(1.872,0.977)--(1.878,0.983)--(1.883,0.988)--(1.889,0.994)--(1.894,0.999)--(1.900,1.005)%
  --(1.906,1.010)--(1.911,1.016)--(1.917,1.021)--(1.922,1.027)--(1.928,1.032)--(1.933,1.038)%
  --(1.939,1.043)--(1.945,1.049)--(1.950,1.054)--(1.956,1.060)--(1.961,1.065)--(1.967,1.071)%
  --(1.972,1.076)--(1.978,1.082)--(1.984,1.088)--(1.989,1.093)--(1.995,1.099)--(2.000,1.104)%
  --(2.006,1.110)--(2.011,1.115)--(2.017,1.121)--(2.023,1.126)--(2.028,1.132)--(2.034,1.137)%
  --(2.039,1.143)--(2.045,1.148)--(2.051,1.154)--(2.056,1.159)--(2.062,1.165)--(2.067,1.170)%
  --(2.073,1.176)--(2.078,1.181)--(2.084,1.187)--(2.090,1.192)--(2.095,1.198)--(2.101,1.203)%
  --(2.106,1.209)--(2.112,1.214)--(2.117,1.220)--(2.123,1.225)--(2.129,1.231)--(2.134,1.236)%
  --(2.140,1.242)--(2.145,1.247)--(2.151,1.253)--(2.156,1.258)--(2.162,1.264)--(2.168,1.269)%
  --(2.173,1.275)--(2.179,1.280)--(2.184,1.285)--(2.190,1.291)--(2.195,1.296)--(2.201,1.302)%
  --(2.207,1.307)--(2.212,1.313)--(2.218,1.318)--(2.223,1.324)--(2.229,1.329)--(2.234,1.334)%
  --(2.240,1.340)--(2.246,1.345)--(2.251,1.351)--(2.257,1.356)--(2.262,1.362)--(2.268,1.367)%
  --(2.273,1.372)--(2.279,1.378)--(2.285,1.383)--(2.290,1.389)--(2.296,1.394)--(2.301,1.399)%
  --(2.307,1.405)--(2.312,1.410)--(2.318,1.416)--(2.324,1.421)--(2.329,1.426)--(2.335,1.432)%
  --(2.340,1.437)--(2.346,1.442)--(2.351,1.448)--(2.357,1.453)--(2.363,1.459)--(2.368,1.464)%
  --(2.374,1.469)--(2.379,1.475)--(2.385,1.480)--(2.390,1.485)--(2.396,1.491)--(2.402,1.496)%
  --(2.407,1.501)--(2.413,1.507)--(2.418,1.512)--(2.424,1.517)--(2.429,1.523)--(2.435,1.528)%
  --(2.441,1.533)--(2.446,1.539)--(2.452,1.544)--(2.457,1.549)--(2.463,1.554)--(2.468,1.560)%
  --(2.474,1.565)--(2.480,1.570)--(2.485,1.576)--(2.491,1.581)--(2.496,1.586)--(2.502,1.591)%
  --(2.507,1.597)--(2.513,1.602)--(2.519,1.607)--(2.524,1.613)--(2.530,1.618)--(2.535,1.623)%
  --(2.541,1.628)--(2.546,1.634)--(2.552,1.639)--(2.558,1.644)--(2.563,1.649)--(2.569,1.655)%
  --(2.574,1.660)--(2.580,1.665)--(2.586,1.670)--(2.591,1.676)--(2.597,1.681)--(2.602,1.686)%
  --(2.608,1.691)--(2.613,1.696)--(2.619,1.702)--(2.625,1.707)--(2.630,1.712)--(2.636,1.717)%
  --(2.641,1.722)--(2.647,1.728)--(2.652,1.733)--(2.658,1.738)--(2.664,1.743)--(2.669,1.748)%
  --(2.675,1.754)--(2.680,1.759)--(2.686,1.764)--(2.691,1.769)--(2.697,1.774)--(2.703,1.780)%
  --(2.708,1.785)--(2.714,1.790)--(2.719,1.795)--(2.725,1.800)--(2.730,1.805)--(2.736,1.810)%
  --(2.742,1.816)--(2.747,1.821)--(2.753,1.826)--(2.758,1.831)--(2.764,1.836)--(2.769,1.841)%
  --(2.775,1.846)--(2.781,1.852)--(2.786,1.857)--(2.792,1.862)--(2.797,1.867)--(2.803,1.872)%
  --(2.808,1.877)--(2.814,1.882)--(2.820,1.887)--(2.825,1.893)--(2.831,1.898)--(2.836,1.903)%
  --(2.842,1.908)--(2.847,1.913)--(2.853,1.918)--(2.859,1.923)--(2.864,1.928)--(2.870,1.933)%
  --(2.875,1.938)--(2.881,1.943)--(2.886,1.949)--(2.892,1.954)--(2.898,1.959)--(2.903,1.964)%
  --(2.909,1.969)--(2.914,1.974)--(2.920,1.979)--(2.925,1.984)--(2.931,1.989)--(2.937,1.994)%
  --(2.942,1.999)--(2.948,2.004)--(2.953,2.009)--(2.959,2.014)--(2.964,2.019)--(2.970,2.025)%
  --(2.976,2.030)--(2.981,2.035)--(2.987,2.040)--(2.992,2.045)--(2.998,2.050)--(3.003,2.055)%
  --(3.009,2.060)--(3.015,2.065)--(3.020,2.070)--(3.026,2.075)--(3.031,2.080)--(3.037,2.085)%
  --(3.042,2.090)--(3.048,2.095)--(3.054,2.100)--(3.059,2.105)--(3.065,2.110)--(3.070,2.115)%
  --(3.076,2.120)--(3.081,2.125)--(3.087,2.130)--(3.093,2.135)--(3.098,2.140)--(3.104,2.145)%
  --(3.109,2.150)--(3.115,2.155)--(3.120,2.160)--(3.126,2.165)--(3.132,2.170)--(3.137,2.175)%
  --(3.143,2.180)--(3.148,2.185)--(3.154,2.190)--(3.160,2.195)--(3.165,2.200)--(3.171,2.205)%
  --(3.176,2.210)--(3.182,2.215)--(3.187,2.220)--(3.193,2.225)--(3.199,2.229)--(3.204,2.234)%
  --(3.210,2.239)--(3.215,2.244)--(3.221,2.249)--(3.226,2.254)--(3.232,2.259)--(3.238,2.264)%
  --(3.243,2.269)--(3.249,2.274)--(3.254,2.279)--(3.260,2.284)--(3.265,2.289)--(3.271,2.294)%
  --(3.277,2.299)--(3.282,2.304)--(3.288,2.309)--(3.293,2.313)--(3.299,2.318)--(3.304,2.323)%
  --(3.310,2.328)--(3.316,2.333)--(3.321,2.338)--(3.327,2.343)--(3.332,2.348)--(3.338,2.353)%
  --(3.343,2.358)--(3.349,2.363)--(3.355,2.368)--(3.360,2.372)--(3.366,2.377)--(3.371,2.382)%
  --(3.377,2.387)--(3.382,2.392)--(3.388,2.397)--(3.394,2.402)--(3.399,2.407)--(3.405,2.412)%
  --(3.410,2.417)--(3.416,2.421)--(3.421,2.426)--(3.427,2.431)--(3.433,2.436)--(3.438,2.441)%
  --(3.444,2.446)--(3.449,2.451)--(3.455,2.456)--(3.460,2.461)--(3.466,2.465)--(3.472,2.470)%
  --(3.477,2.475)--(3.483,2.480)--(3.488,2.485)--(3.494,2.490)--(3.499,2.495)--(3.505,2.500)%
  --(3.511,2.504)--(3.516,2.509)--(3.522,2.514)--(3.527,2.519)--(3.533,2.524)--(3.538,2.529)%
  --(3.544,2.534)--(3.550,2.539)--(3.555,2.543)--(3.561,2.548)--(3.566,2.553)--(3.572,2.558)%
  --(3.577,2.563)--(3.583,2.568)--(3.589,2.572)--(3.594,2.577)--(3.600,2.582)--(3.605,2.587)%
  --(3.611,2.592)--(3.616,2.597)--(3.622,2.602)--(3.628,2.606)--(3.633,2.611)--(3.639,2.616)%
  --(3.644,2.621)--(3.650,2.626)--(3.655,2.631)--(3.661,2.635)--(3.667,2.640)--(3.672,2.645)%
  --(3.678,2.650)--(3.683,2.655)--(3.689,2.660)--(3.695,2.665)--(3.700,2.669)--(3.706,2.674)%
  --(3.711,2.679)--(3.717,2.684)--(3.722,2.689)--(3.728,2.693)--(3.734,2.698)--(3.739,2.703)%
  --(3.745,2.708)--(3.750,2.713)--(3.756,2.718)--(3.761,2.722)--(3.767,2.727)--(3.773,2.732)%
  --(3.778,2.737)--(3.784,2.742)--(3.789,2.747)--(3.795,2.751)--(3.800,2.756)--(3.806,2.761)%
  --(3.812,2.766)--(3.817,2.771)--(3.823,2.775)--(3.828,2.780)--(3.834,2.785)--(3.839,2.790)%
  --(3.845,2.795)--(3.851,2.800)--(3.856,2.804)--(3.862,2.809)--(3.867,2.814)--(3.873,2.819)%
  --(3.878,2.824)--(3.884,2.828)--(3.890,2.833)--(3.895,2.838)--(3.901,2.843)--(3.906,2.848)%
  --(3.912,2.852)--(3.917,2.857)--(3.923,2.862)--(3.929,2.867)--(3.934,2.872)--(3.940,2.876)%
  --(3.945,2.881)--(3.951,2.886)--(3.956,2.891)--(3.962,2.896)--(3.968,2.900)--(3.973,2.905)%
  --(3.979,2.910)--(3.984,2.915)--(3.990,2.920)--(3.995,2.924)--(4.001,2.929)--(4.007,2.934)%
  --(4.012,2.939)--(4.018,2.944)--(4.023,2.948)--(4.029,2.953)--(4.034,2.958)--(4.040,2.963)%
  --(4.046,2.968)--(4.051,2.972)--(4.057,2.977)--(4.062,2.982)--(4.068,2.987)--(4.073,2.992)%
  --(4.079,2.996)--(4.085,3.001)--(4.090,3.006)--(4.096,3.011)--(4.101,3.015)--(4.107,3.020)%
  --(4.112,3.025)--(4.118,3.030)--(4.124,3.035)--(4.129,3.039)--(4.135,3.044)--(4.140,3.049)%
  --(4.146,3.054)--(4.151,3.059)--(4.157,3.063)--(4.163,3.068)--(4.168,3.073)--(4.174,3.078)%
  --(4.179,3.083)--(4.185,3.087)--(4.190,3.092)--(4.196,3.097)--(4.202,3.102)--(4.207,3.106)%
  --(4.213,3.111)--(4.218,3.116)--(4.224,3.121)--(4.230,3.126)--(4.235,3.130)--(4.241,3.135)%
  --(4.246,3.140)--(4.252,3.145)--(4.257,3.150)--(4.263,3.154)--(4.269,3.159)--(4.274,3.164)%
  --(4.280,3.169)--(4.285,3.174)--(4.291,3.178)--(4.296,3.183)--(4.302,3.188)--(4.308,3.193)%
  --(4.313,3.197)--(4.319,3.202)--(4.324,3.207)--(4.330,3.212)--(4.335,3.217)--(4.341,3.221)%
  --(4.347,3.226)--(4.352,3.231)--(4.358,3.236)--(4.363,3.241)--(4.369,3.245)--(4.374,3.250)%
  --(4.380,3.255)--(4.386,3.260)--(4.391,3.265)--(4.397,3.269)--(4.402,3.274)--(4.408,3.279)%
  --(4.413,3.284)--(4.419,3.289)--(4.425,3.293)--(4.430,3.298)--(4.436,3.303)--(4.441,3.308)%
  --(4.447,3.313)--(4.452,3.317)--(4.458,3.322)--(4.464,3.327)--(4.469,3.332)--(4.475,3.336)%
  --(4.480,3.341)--(4.486,3.346)--(4.491,3.351)--(4.497,3.356)--(4.503,3.360)--(4.508,3.365)%
  --(4.514,3.370)--(4.519,3.375)--(4.525,3.380)--(4.530,3.385)--(4.536,3.389)--(4.542,3.394)%
  --(4.547,3.399)--(4.553,3.404)--(4.558,3.409)--(4.564,3.413)--(4.569,3.418)--(4.575,3.423)%
  --(4.581,3.428)--(4.586,3.433)--(4.592,3.437)--(4.597,3.442)--(4.603,3.447)--(4.608,3.452)%
  --(4.614,3.457)--(4.620,3.461)--(4.625,3.466)--(4.631,3.471)--(4.636,3.476)--(4.642,3.481)%
  --(4.647,3.486)--(4.653,3.490)--(4.659,3.495)--(4.664,3.500)--(4.670,3.505)--(4.675,3.510)%
  --(4.681,3.514)--(4.686,3.519)--(4.692,3.524)--(4.698,3.529)--(4.703,3.534)--(4.709,3.539)%
  --(4.714,3.543)--(4.720,3.548)--(4.725,3.553)--(4.731,3.558)--(4.737,3.563)--(4.742,3.567)%
  --(4.748,3.572)--(4.753,3.577)--(4.759,3.582)--(4.764,3.587)--(4.770,3.592)--(4.776,3.596)%
  --(4.781,3.601)--(4.787,3.606)--(4.792,3.611)--(4.798,3.616)--(4.804,3.621)--(4.809,3.626)%
  --(4.815,3.630)--(4.820,3.635)--(4.826,3.640)--(4.831,3.645)--(4.837,3.650)--(4.843,3.655)%
  --(4.848,3.659)--(4.854,3.664)--(4.859,3.669)--(4.865,3.674)--(4.870,3.679)--(4.876,3.684)%
  --(4.882,3.689)--(4.887,3.693)--(4.893,3.698)--(4.898,3.703)--(4.904,3.708)--(4.909,3.713)%
  --(4.915,3.718)--(4.921,3.723)--(4.926,3.727)--(4.932,3.732)--(4.937,3.737)--(4.943,3.742)%
  --(4.948,3.747)--(4.954,3.752)--(4.960,3.757)--(4.965,3.762)--(4.971,3.766)--(4.976,3.771)%
  --(4.982,3.776)--(4.987,3.781)--(4.993,3.786)--(4.999,3.791)--(5.004,3.796)--(5.010,3.801)%
  --(5.015,3.805)--(5.021,3.810)--(5.026,3.815)--(5.032,3.820)--(5.038,3.825)--(5.043,3.830)%
  --(5.049,3.835)--(5.054,3.840)--(5.060,3.845)--(5.065,3.850)--(5.071,3.854)--(5.077,3.859)%
  --(5.082,3.864)--(5.088,3.869)--(5.093,3.874)--(5.099,3.879)--(5.104,3.884)--(5.110,3.889)%
  --(5.116,3.894)--(5.121,3.899)--(5.127,3.903)--(5.132,3.908)--(5.138,3.913)--(5.143,3.918)%
  --(5.149,3.923)--(5.155,3.928)--(5.160,3.933)--(5.166,3.938)--(5.171,3.943)--(5.177,3.948)%
  --(5.182,3.953)--(5.188,3.958)--(5.194,3.963)--(5.199,3.968)--(5.205,3.972)--(5.210,3.977)%
  --(5.216,3.982)--(5.221,3.987)--(5.227,3.992)--(5.233,3.997)--(5.238,4.002)--(5.244,4.007)%
  --(5.249,4.012)--(5.255,4.017)--(5.260,4.022)--(5.266,4.027)--(5.272,4.032)--(5.277,4.037)%
  --(5.283,4.042)--(5.288,4.047)--(5.294,4.052)--(5.299,4.057)--(5.305,4.062)--(5.311,4.067)%
  --(5.316,4.072)--(5.322,4.077)--(5.327,4.082)--(5.333,4.087)--(5.339,4.092)--(5.344,4.097)%
  --(5.350,4.102)--(5.355,4.107)--(5.361,4.112)--(5.366,4.117)--(5.372,4.122)--(5.378,4.127)%
  --(5.383,4.132)--(5.389,4.137)--(5.394,4.142)--(5.400,4.147)--(5.405,4.152)--(5.411,4.157)%
  --(5.417,4.162)--(5.422,4.167)--(5.428,4.172)--(5.433,4.177)--(5.439,4.182)--(5.444,4.187)%
  --(5.450,4.192)--(5.456,4.197)--(5.461,4.202)--(5.467,4.207)--(5.472,4.212)--(5.478,4.217)%
  --(5.483,4.222)--(5.489,4.227)--(5.495,4.232)--(5.500,4.237)--(5.506,4.242)--(5.511,4.247)%
  --(5.517,4.252)--(5.522,4.257)--(5.528,4.262)--(5.534,4.267)--(5.539,4.273)--(5.545,4.278)%
  --(5.550,4.283)--(5.556,4.288)--(5.561,4.293)--(5.567,4.298)--(5.573,4.303)--(5.578,4.308)%
  --(5.584,4.313)--(5.589,4.318)--(5.595,4.323)--(5.600,4.328)--(5.606,4.334)--(5.612,4.339)%
  --(5.617,4.344)--(5.623,4.349)--(5.628,4.354)--(5.634,4.359)--(5.639,4.364)--(5.645,4.369)%
  --(5.651,4.374)--(5.656,4.380)--(5.662,4.385)--(5.667,4.390)--(5.673,4.395)--(5.678,4.400)%
  --(5.684,4.405)--(5.690,4.410)--(5.695,4.415)--(5.701,4.421)--(5.706,4.426)--(5.712,4.431)%
  --(5.717,4.436)--(5.723,4.441)--(5.729,4.446)--(5.734,4.451)--(5.740,4.457)--(5.745,4.462)%
  --(5.751,4.467)--(5.756,4.472)--(5.762,4.477)--(5.768,4.483)--(5.773,4.488)--(5.779,4.493)%
  --(5.784,4.498)--(5.790,4.503)--(5.795,4.508)--(5.801,4.514)--(5.807,4.519)--(5.812,4.524)%
  --(5.818,4.529)--(5.823,4.534)--(5.829,4.540)--(5.834,4.545)--(5.840,4.550)--(5.846,4.555)%
  --(5.851,4.561)--(5.857,4.566)--(5.862,4.571)--(5.868,4.576)--(5.873,4.581)--(5.879,4.587)%
  --(5.885,4.592)--(5.890,4.597)--(5.896,4.602)--(5.901,4.608)--(5.907,4.613)--(5.913,4.618)%
  --(5.918,4.623)--(5.924,4.629)--(5.929,4.634)--(5.935,4.639)--(5.940,4.644)--(5.946,4.650)%
  --(5.952,4.655)--(5.957,4.660)--(5.963,4.666)--(5.968,4.671)--(5.974,4.676)--(5.979,4.681)%
  --(5.985,4.687)--(5.991,4.692)--(5.996,4.697)--(6.002,4.703)--(6.007,4.708)--(6.013,4.713)%
  --(6.018,4.719)--(6.024,4.724)--(6.030,4.729)--(6.035,4.735)--(6.041,4.740)--(6.046,4.745)%
  --(6.052,4.751)--(6.057,4.756)--(6.063,4.761)--(6.069,4.767)--(6.074,4.772)--(6.080,4.777)%
  --(6.085,4.783)--(6.091,4.788)--(6.096,4.793)--(6.102,4.799)--(6.108,4.804)--(6.113,4.809)%
  --(6.119,4.815)--(6.124,4.820)--(6.130,4.826)--(6.135,4.831)--(6.141,4.836)--(6.147,4.842)%
  --(6.152,4.847)--(6.158,4.852)--(6.163,4.858)--(6.169,4.863)--(6.174,4.869)--(6.180,4.874)%
  --(6.186,4.879)--(6.191,4.885)--(6.197,4.890)--(6.202,4.896)--(6.208,4.901)--(6.213,4.907)%
  --(6.219,4.912)--(6.225,4.917)--(6.230,4.923)--(6.236,4.928)--(6.241,4.934)--(6.247,4.939)%
  --(6.252,4.945)--(6.258,4.950)--(6.264,4.956)--(6.269,4.961)--(6.275,4.966)--(6.280,4.972)%
  --(6.286,4.977)--(6.291,4.983)--(6.297,4.988)--(6.303,4.994)--(6.308,4.999)--(6.314,5.005)%
  --(6.319,5.010)--(6.325,5.016)--(6.330,5.021)--(6.336,5.027)--(6.342,5.032)--(6.347,5.038)%
  --(6.353,5.043)--(6.358,5.049)--(6.364,5.054)--(6.369,5.060)--(6.375,5.065)--(6.381,5.071)%
  --(6.386,5.076)--(6.392,5.082)--(6.397,5.087)--(6.403,5.093)--(6.408,5.098)--(6.414,5.104)%
  --(6.420,5.109)--(6.425,5.115)--(6.431,5.120)--(6.436,5.126)--(6.442,5.131)--(6.448,5.137)%
  --(6.453,5.142)--(6.459,5.148)--(6.464,5.153)--(6.470,5.159)--(6.475,5.164)--(6.481,5.170)%
  --(6.487,5.176)--(6.492,5.181)--(6.498,5.187)--(6.503,5.192)--(6.509,5.198)--(6.514,5.203)%
  --(6.520,5.209)--(6.526,5.214)--(6.531,5.220)--(6.537,5.225)--(6.542,5.231)--(6.548,5.236)%
  --(6.553,5.242)--(6.559,5.247)--(6.565,5.253)--(6.570,5.258)--(6.576,5.264)--(6.581,5.269)%
  --(6.587,5.275)--(6.592,5.280)--(6.598,5.286)--(6.604,5.292)--(6.609,5.297)--(6.615,5.302)%
  --(6.620,5.308)--(6.626,5.313)--(6.631,5.319)--(6.637,5.324)--(6.643,5.330)--(6.648,5.335)%
  --(6.654,5.341)--(6.659,5.346)--(6.665,5.352)--(6.670,5.357)--(6.676,5.363)--(6.682,5.368)%
  --(6.687,5.373)--(6.693,5.379)--(6.698,5.384)--(6.704,5.390)--(6.709,5.395)--(6.715,5.400)%
  --(6.721,5.406)--(6.726,5.411)--(6.732,5.416)--(6.737,5.422)--(6.743,5.427)--(6.748,5.432)%
  --(6.754,5.438)--(6.760,5.443)--(6.765,5.448)--(6.771,5.453)--(6.776,5.459)--(6.782,5.464)%
  --(6.787,5.469)--(6.793,5.474)--(6.799,5.479)--(6.804,5.484)--(6.810,5.489)--(6.815,5.494)%
  --(6.821,5.499)--(6.826,5.504)--(6.832,5.509)--(6.838,5.514)--(6.843,5.519)--(6.849,5.524)%
  --(6.854,5.529)--(6.860,5.533)--(6.865,5.538)--(6.871,5.543)--(6.877,5.548)--(6.882,5.552)%
  --(6.888,5.557)--(6.893,5.561)--(6.899,5.565)--(6.904,5.570)--(6.910,5.574)--(6.916,5.578)%
  --(6.921,5.582)--(6.927,5.586)--(6.932,5.590)--(6.938,5.594)--(6.943,5.598)--(6.949,5.602)%
  --(6.955,5.605)--(6.960,5.609)--(6.966,5.612)--(6.971,5.615)--(6.977,5.618)--(6.982,5.621)%
  --(6.988,5.623)--(6.994,5.626)--(6.999,5.628)--(7.005,5.629)--(7.010,5.630)--(7.016,5.630);
\gpcolor{\gprgb{0}{0}{1000}}
\draw[gp path] (1.858,0.623)--(1.858,0.630)--(1.858,0.637)--(1.858,0.644)--(1.858,0.651)%
  --(1.858,0.658)--(1.858,0.665)--(1.858,0.672)--(1.858,0.679)--(1.894,0.679)--(1.894,0.686)%
  --(1.894,0.693)--(1.894,0.700)--(1.894,0.707)--(1.894,0.714)--(1.905,0.714)--(1.905,0.720)%
  --(1.905,0.727)--(1.905,0.734)--(1.905,0.741)--(1.905,0.748)--(1.905,0.755)--(1.905,0.762)%
  --(1.905,0.769)--(1.905,0.776)--(1.905,0.783)--(2.005,0.783)--(2.005,0.790)--(2.005,0.797)%
  --(2.005,0.804)--(2.005,0.811)--(2.005,0.818)--(2.005,0.825)--(2.005,0.832)--(2.005,0.839)%
  --(2.005,0.846)--(2.005,0.853)--(2.005,0.860)--(2.005,0.867)--(2.005,0.874)--(2.005,0.881)%
  --(2.005,0.888)--(2.005,0.895)--(2.050,0.895)--(2.050,0.902)--(2.050,0.909)--(2.050,0.916)%
  --(2.050,0.922)--(2.050,0.929)--(2.050,0.936)--(2.050,0.943)--(2.050,0.950)--(2.050,0.957)%
  --(2.076,0.957)--(2.076,0.964)--(2.076,0.971)--(2.076,0.978)--(2.076,0.985)--(2.076,0.992)%
  --(2.116,0.992)--(2.116,0.999)--(2.116,1.006)--(2.116,1.013)--(2.116,1.020)--(2.116,1.027)%
  --(2.116,1.034)--(2.116,1.041)--(2.116,1.048)--(2.116,1.055)--(2.162,1.055)--(2.162,1.062)%
  --(2.162,1.069)--(2.162,1.076)--(2.162,1.083)--(2.162,1.090)--(2.162,1.097)--(2.162,1.104)%
  --(2.162,1.111)--(2.162,1.118)--(2.162,1.124)--(2.196,1.124)--(2.196,1.131)--(2.196,1.138)%
  --(2.196,1.145)--(2.196,1.152)--(2.196,1.159)--(2.196,1.166)--(2.196,1.173)--(2.196,1.180)%
  --(2.196,1.187)--(2.196,1.194)--(2.211,1.194)--(2.211,1.201)--(2.211,1.208)--(2.211,1.215)%
  --(2.211,1.222)--(2.211,1.229)--(2.211,1.236)--(2.211,1.243)--(2.211,1.250)--(2.211,1.257)%
  --(2.211,1.264)--(2.211,1.271)--(2.211,1.278)--(2.211,1.285)--(2.211,1.292)--(2.211,1.299)%
  --(2.211,1.306)--(2.535,1.306)--(2.535,1.313)--(2.535,1.319)--(2.535,1.326)--(2.535,1.333)%
  --(2.535,1.340)--(2.535,1.347)--(2.535,1.354)--(2.535,1.361)--(2.535,1.368)--(2.535,1.375)%
  --(2.621,1.375)--(2.621,1.382)--(2.621,1.389)--(2.621,1.396)--(2.621,1.403)--(2.621,1.410)%
  --(2.621,1.417)--(2.621,1.424)--(2.621,1.431)--(2.621,1.438)--(2.621,1.445)--(2.621,1.452)%
  --(2.621,1.459)--(2.621,1.466)--(2.621,1.473)--(2.621,1.480)--(2.621,1.487)--(2.621,1.494)%
  --(2.621,1.501)--(2.621,1.508)--(2.621,1.515)--(2.621,1.521)--(2.621,1.528)--(2.621,1.535)%
  --(2.621,1.542)--(2.621,1.549)--(2.621,1.556)--(2.621,1.563)--(2.621,1.570)--(2.621,1.577)%
  --(2.621,1.584)--(2.621,1.591)--(2.621,1.598)--(2.621,1.605)--(2.621,1.612)--(2.621,1.619)%
  --(2.621,1.626)--(2.621,1.633)--(2.621,1.640)--(2.621,1.647)--(2.621,1.654)--(2.621,1.661)%
  --(2.621,1.668)--(2.621,1.675)--(2.621,1.682)--(2.621,1.689)--(2.621,1.696)--(2.621,1.703)%
  --(2.621,1.710)--(2.621,1.717)--(2.621,1.723)--(2.621,1.730)--(2.724,1.730)--(2.724,1.737)%
  --(2.724,1.744)--(2.724,1.751)--(2.724,1.758)--(2.724,1.765)--(2.724,1.772)--(2.724,1.779)%
  --(2.724,1.786)--(2.724,1.793)--(2.724,1.800)--(2.727,1.800)--(2.727,1.807)--(2.727,1.814)%
  --(2.727,1.821)--(2.727,1.828)--(2.727,1.835)--(2.781,1.835)--(2.781,1.842)--(2.781,1.849)%
  --(2.781,1.856)--(2.781,1.863)--(2.781,1.870)--(2.808,1.870)--(2.808,1.877)--(2.808,1.884)%
  --(2.808,1.891)--(2.808,1.898)--(2.808,1.905)--(2.808,1.912)--(2.808,1.919)--(2.808,1.925)%
  --(2.808,1.932)--(2.808,1.939)--(2.808,1.946)--(2.808,1.953)--(2.808,1.960)--(2.808,1.967)%
  --(2.808,1.974)--(2.808,1.981)--(2.826,1.981)--(2.826,1.988)--(2.826,1.995)--(2.826,2.002)%
  --(2.826,2.009)--(2.826,2.016)--(2.826,2.023)--(2.826,2.030)--(2.826,2.037)--(2.826,2.044)%
  --(2.826,2.051)--(2.855,2.051)--(2.855,2.058)--(2.855,2.065)--(2.855,2.072)--(2.855,2.079)%
  --(2.855,2.086)--(2.855,2.093)--(2.855,2.100)--(2.855,2.107)--(2.855,2.114)--(2.855,2.121)%
  --(2.855,2.127)--(2.855,2.134)--(2.855,2.141)--(2.855,2.148)--(2.950,2.148)--(2.950,2.155)%
  --(2.950,2.162)--(2.950,2.169)--(2.950,2.176)--(2.950,2.183)--(2.950,2.190)--(2.950,2.197)%
  --(2.950,2.204)--(2.950,2.211)--(2.950,2.218)--(2.950,2.225)--(2.950,2.232)--(2.950,2.239)%
  --(2.950,2.246)--(2.950,2.253)--(2.950,2.260)--(2.973,2.260)--(2.973,2.267)--(2.973,2.274)%
  --(2.973,2.281)--(2.973,2.288)--(2.973,2.295)--(2.973,2.302)--(2.973,2.309)--(2.973,2.316)%
  --(2.973,2.322)--(3.253,2.322)--(3.253,2.329)--(3.253,2.336)--(3.253,2.343)--(3.253,2.350)%
  --(3.253,2.357)--(3.253,2.364)--(3.253,2.371)--(3.253,2.378)--(3.253,2.385)--(3.253,2.392)%
  --(3.253,2.399)--(3.253,2.406)--(3.253,2.413)--(3.253,2.420)--(3.253,2.427)--(3.253,2.434)%
  --(3.368,2.434)--(3.368,2.441)--(3.368,2.448)--(3.368,2.455)--(3.368,2.462)--(3.368,2.469)%
  --(3.368,2.476)--(3.368,2.483)--(3.368,2.490)--(3.368,2.497)--(3.425,2.497)--(3.425,2.504)%
  --(3.425,2.511)--(3.425,2.518)--(3.425,2.524)--(3.425,2.531)--(3.425,2.538)--(3.425,2.545)%
  --(3.425,2.552)--(3.425,2.559)--(3.425,2.566)--(3.425,2.573)--(3.425,2.580)--(3.425,2.587)%
  --(3.425,2.594)--(3.425,2.601)--(3.425,2.608)--(3.425,2.615)--(3.425,2.622)--(3.425,2.629)%
  --(3.581,2.629)--(3.581,2.636)--(3.581,2.643)--(3.581,2.650)--(3.581,2.657)--(3.581,2.664)%
  --(3.596,2.664)--(3.596,2.671)--(3.596,2.678)--(3.596,2.685)--(3.596,2.692)--(3.596,2.699)%
  --(3.596,2.706)--(3.596,2.713)--(3.596,2.720)--(3.596,2.726)--(3.596,2.733)--(3.609,2.733)%
  --(3.609,2.740)--(3.609,2.747)--(3.609,2.754)--(3.609,2.761)--(3.609,2.768)--(3.609,2.775)%
  --(3.609,2.782)--(3.609,2.789)--(3.609,2.796)--(3.609,2.803)--(3.609,2.810)--(3.609,2.817)%
  --(3.609,2.824)--(3.609,2.831)--(3.609,2.838)--(3.609,2.845)--(3.666,2.845)--(3.666,2.852)%
  --(3.666,2.859)--(3.666,2.866)--(3.666,2.873)--(3.666,2.880)--(3.666,2.887)--(3.666,2.894)%
  --(3.666,2.901)--(3.666,2.908)--(3.666,2.915)--(3.723,2.915)--(3.723,2.922)--(3.723,2.928)%
  --(3.723,2.935)--(3.723,2.942)--(3.723,2.949)--(4.074,2.949)--(4.074,2.956)--(4.074,2.963)%
  --(4.074,2.970)--(4.074,2.977)--(4.074,2.984)--(4.074,2.991)--(4.074,2.998)--(4.074,3.005)%
  --(4.074,3.012)--(4.083,3.012)--(4.083,3.019)--(4.083,3.026)--(4.083,3.033)--(4.083,3.040)%
  --(4.083,3.047)--(4.083,3.054)--(4.083,3.061)--(4.083,3.068)--(4.083,3.075)--(4.083,3.082)%
  --(4.229,3.082)--(4.229,3.089)--(4.229,3.096)--(4.229,3.103)--(4.229,3.110)--(4.229,3.117)%
  --(4.229,3.124)--(4.229,3.130)--(4.229,3.137)--(4.229,3.144)--(4.229,3.151)--(4.229,3.158)%
  --(4.229,3.165)--(4.229,3.172)--(4.229,3.179)--(4.229,3.186)--(4.229,3.193)--(4.229,3.200)%
  --(4.229,3.207)--(4.229,3.214)--(4.229,3.221)--(4.229,3.228)--(4.229,3.235)--(4.229,3.242)%
  --(4.230,3.242)--(4.230,3.249)--(4.230,3.256)--(4.230,3.263)--(4.230,3.270)--(4.230,3.277)%
  --(4.230,3.284)--(4.230,3.291)--(4.230,3.298)--(4.230,3.305)--(4.230,3.312)--(4.230,3.319)%
  --(4.230,3.325)--(4.230,3.332)--(4.230,3.339)--(4.230,3.346)--(4.230,3.353)--(4.230,3.360)%
  --(4.230,3.367)--(4.230,3.374)--(4.230,3.381)--(4.230,3.388)--(4.230,3.395)--(4.230,3.402)%
  --(4.230,3.409)--(4.230,3.416)--(4.230,3.423)--(4.230,3.430)--(4.230,3.437)--(4.230,3.444)%
  --(4.436,3.444)--(4.436,3.451)--(4.436,3.458)--(4.436,3.465)--(4.436,3.472)--(4.436,3.479)%
  --(4.436,3.486)--(4.436,3.493)--(4.436,3.500)--(4.436,3.507)--(4.436,3.514)--(4.436,3.521)%
  --(4.436,3.527)--(4.436,3.534)--(4.436,3.541)--(4.436,3.548)--(4.436,3.555)--(4.482,3.555)%
  --(4.482,3.562)--(4.482,3.569)--(4.482,3.576)--(4.482,3.583)--(4.482,3.590)--(4.482,3.597)%
  --(4.482,3.604)--(4.482,3.611)--(4.482,3.618)--(4.482,3.625)--(4.507,3.625)--(4.507,3.632)%
  --(4.507,3.639)--(4.507,3.646)--(4.507,3.653)--(4.507,3.660)--(4.507,3.667)--(4.507,3.674)%
  --(4.507,3.681)--(4.507,3.688)--(4.658,3.688)--(4.658,3.695)--(4.658,3.702)--(4.658,3.709)%
  --(4.658,3.716)--(4.658,3.723)--(4.658,3.729)--(4.658,3.736)--(4.658,3.743)--(4.658,3.750)%
  --(4.658,3.757)--(4.658,3.764)--(4.658,3.771)--(4.658,3.778)--(4.658,3.785)--(4.658,3.792)%
  --(4.658,3.799)--(4.673,3.799)--(4.673,3.806)--(4.673,3.813)--(4.673,3.820)--(4.673,3.827)%
  --(4.673,3.834)--(4.673,3.841)--(4.673,3.848)--(4.673,3.855)--(4.673,3.862)--(4.673,3.869)%
  --(4.878,3.869)--(4.878,3.876)--(4.878,3.883)--(4.878,3.890)--(4.878,3.897)--(4.878,3.904)%
  --(4.878,3.911)--(4.878,3.918)--(4.878,3.925)--(4.878,3.931)--(4.878,3.938)--(4.878,3.945)%
  --(4.878,3.952)--(4.878,3.959)--(4.878,3.966)--(5.034,3.966)--(5.034,3.973)--(5.034,3.980)%
  --(5.034,3.987)--(5.034,3.994)--(5.034,4.001)--(5.180,4.001)--(5.180,4.008)--(5.180,4.015)%
  --(5.180,4.022)--(5.180,4.029)--(5.180,4.036)--(5.180,4.043)--(5.180,4.050)--(5.180,4.057)%
  --(5.180,4.064)--(5.180,4.071)--(5.442,4.071)--(5.442,4.078)--(5.442,4.085)--(5.442,4.092)%
  --(5.442,4.099)--(5.442,4.106)--(5.442,4.113)--(5.442,4.120)--(5.442,4.127)--(5.442,4.133)%
  --(5.442,4.140)--(5.442,4.147)--(5.442,4.154)--(5.442,4.161)--(5.442,4.168)--(5.442,4.175)%
  --(5.442,4.182)--(5.457,4.182)--(5.457,4.189)--(5.457,4.196)--(5.457,4.203)--(5.457,4.210)%
  --(5.457,4.217)--(5.457,4.224)--(5.457,4.231)--(5.457,4.238)--(5.457,4.245)--(5.457,4.252)%
  --(5.604,4.252)--(5.604,4.259)--(5.604,4.266)--(5.604,4.273)--(5.604,4.280)--(5.604,4.287)%
  --(5.730,4.287)--(5.730,4.294)--(5.730,4.301)--(5.730,4.308)--(5.730,4.315)--(5.730,4.322)%
  --(5.730,4.328)--(5.730,4.335)--(5.730,4.342)--(5.730,4.349)--(5.730,4.356)--(5.760,4.356)%
  --(5.760,4.363)--(5.760,4.370)--(5.760,4.377)--(5.760,4.384)--(5.760,4.391)--(5.760,4.398)%
  --(5.760,4.405)--(5.760,4.412)--(5.760,4.419)--(5.838,4.419)--(5.838,4.426)--(5.838,4.433)%
  --(5.838,4.440)--(5.838,4.447)--(5.838,4.454)--(5.838,4.461)--(5.838,4.468)--(5.838,4.475)%
  --(5.838,4.482)--(5.899,4.482)--(5.899,4.489)--(5.899,4.496)--(5.899,4.503)--(5.899,4.510)%
  --(5.899,4.517)--(5.965,4.517)--(5.965,4.524)--(5.965,4.530)--(5.965,4.537)--(5.965,4.544)%
  --(5.965,4.551)--(5.992,4.551)--(5.992,4.558)--(5.992,4.565)--(5.992,4.572)--(5.992,4.579)%
  --(5.992,4.586)--(5.992,4.593)--(5.992,4.600)--(5.992,4.607)--(5.992,4.614)--(5.992,4.621)%
  --(5.992,4.628)--(5.992,4.635)--(5.992,4.642)--(5.992,4.649)--(5.992,4.656)--(5.992,4.663)%
  --(6.135,4.663)--(6.135,4.670)--(6.135,4.677)--(6.135,4.684)--(6.135,4.691)--(6.135,4.698)%
  --(6.135,4.705)--(6.135,4.712)--(6.135,4.719)--(6.135,4.726)--(6.135,4.732)--(6.157,4.732)%
  --(6.157,4.739)--(6.157,4.746)--(6.157,4.753)--(6.157,4.760)--(6.157,4.767)--(6.157,4.774)%
  --(6.157,4.781)--(6.157,4.788)--(6.157,4.795)--(6.157,4.802)--(6.191,4.802)--(6.191,4.809)%
  --(6.191,4.816)--(6.191,4.823)--(6.191,4.830)--(6.191,4.837)--(6.191,4.844)--(6.191,4.851)%
  --(6.191,4.858)--(6.191,4.865)--(6.290,4.865)--(6.290,4.872)--(6.290,4.879)--(6.290,4.886)%
  --(6.290,4.893)--(6.290,4.900)--(6.290,4.907)--(6.290,4.914)--(6.290,4.921)--(6.290,4.928)%
  --(6.508,4.928)--(6.508,4.934)--(6.508,4.941)--(6.508,4.948)--(6.508,4.955)--(6.508,4.962)%
  --(6.508,4.969)--(6.508,4.976)--(6.508,4.983)--(6.508,4.990)--(6.508,4.997)--(6.508,5.004)%
  --(6.508,5.011)--(6.508,5.018)--(6.508,5.025)--(6.508,5.032)--(6.508,5.039)--(6.532,5.039)%
  --(6.532,5.046)--(6.532,5.053)--(6.532,5.060)--(6.532,5.067)--(6.532,5.074)--(6.532,5.081)%
  --(6.532,5.088)--(6.532,5.095)--(6.532,5.102)--(6.565,5.102)--(6.565,5.109)--(6.565,5.116)%
  --(6.565,5.123)--(6.565,5.130)--(6.565,5.136)--(6.565,5.143)--(6.565,5.150)--(6.565,5.157)%
  --(6.565,5.164)--(6.565,5.171)--(6.565,5.178)--(6.565,5.185)--(6.565,5.192)--(6.565,5.199)%
  --(6.643,5.199)--(6.643,5.206)--(6.643,5.213)--(6.643,5.220)--(6.643,5.227)--(6.643,5.234)%
  --(6.862,5.234)--(6.862,5.241)--(6.862,5.248)--(6.862,5.255)--(6.862,5.262)--(6.862,5.269)%
  --(6.884,5.269)--(6.884,5.276)--(6.884,5.283)--(6.884,5.290)--(6.884,5.297)--(6.884,5.304)%
  --(6.884,5.311)--(6.884,5.318)--(6.884,5.325)--(6.884,5.331)--(6.884,5.338)--(6.884,5.345)%
  --(6.884,5.352)--(6.884,5.359)--(6.884,5.366)--(6.884,5.373)--(6.884,5.380)--(6.910,5.380)%
  --(6.910,5.387)--(6.910,5.394)--(6.910,5.401)--(6.910,5.408)--(6.910,5.415)--(6.910,5.422)%
  --(6.910,5.429)--(6.910,5.436)--(6.910,5.443)--(6.910,5.450)--(6.923,5.450)--(6.923,5.457)%
  --(6.923,5.464)--(6.923,5.471)--(6.923,5.478)--(6.923,5.485)--(6.923,5.492)--(6.923,5.499)%
  --(6.923,5.506)--(6.923,5.513)--(6.923,5.520)--(7.020,5.520)--(7.020,5.527)--(7.020,5.533)%
  --(7.020,5.540)--(7.020,5.547)--(7.020,5.554)--(7.043,5.554)--(7.043,5.561)--(7.043,5.568)%
  --(7.043,5.575)--(7.043,5.582)--(7.043,5.589)--(7.043,5.596)--(7.043,5.603)--(7.043,5.610)%
  --(7.043,5.617)--(7.043,5.624)--(7.447,5.624)--(7.447,5.631);
\gpdefrectangularnode{gp plot 1}{\pgfpoint{1.012cm}{0.616cm}}{\pgfpoint{7.447cm}{5.631cm}}
\end{tikzpicture}